\documentclass[11pt]{amsart}

\usepackage{geometry}
\usepackage{amsmath}
\usepackage{amssymb}
\usepackage{dsfont}
\usepackage{caption}
\usepackage{subcaption}
\usepackage{amsthm}
\usepackage[hidelinks]{hyperref}
\usepackage{thmtools}
\usepackage[capitalise]{cleveref}
\usepackage[nocompress]{cite}
\usepackage{mathtools}
\usepackage{tikz}
\usepackage{circuitikz}
\usepackage[british]{babel}
\usepackage{enumerate}
\tikzset{vtx/.style={inner sep=1.7pt, outer sep=0pt, circle, fill,draw}}
\usepackage{graphicx}
\usepackage{float}
\usepackage[shortlabels]{enumitem}
\usepackage{etoolbox}

\title{Fractional Clique Decompositions in Random Hypergraphs}
\author{Felix Joos}
\author{Zak Smith}
\thanks{The research leading to these results was partially supported by the Deutsche Forschungsgemeinschaft (DFG, German Research Foundation) -- 428212407.}

\theoremstyle{definition}
\newtheorem{definition}{Definition}[section]
\theoremstyle{plain}
\newtheorem{theorem}[definition]{Theorem}

\newtheorem{lemma}[definition]{Lemma}

\newtheorem{claim}{Claim}
\newtheorem{fact}[definition]{Fact}

\newtheorem{conjecture}[definition]{Conjecture}

\newtoggle{inclaimproof}
\togglefalse{inclaimproof}
\newcounter{proofcounter} % Track top-level proofs

\AtBeginEnvironment{proof}{%
  \iftoggle{inclaimproof}{}{%
    \stepcounter{proofcounter}%
    \setcounter{claim}{0}%
  }%
}

\newenvironment{claimproof}{%
  \toggletrue{inclaimproof}% Set the toggle to true before calling proof
  \let\origqed=\qedsymbol%
  \renewcommand{\qedsymbol}{$\blacksquare$}%
  \begin{proof}[Proof of claim]%
}{%
  \end{proof}%
  \let\qedsymbol=\origqed%
  \togglefalse{inclaimproof}% Reset the toggle upon exit
}

\Crefname{fact}{Fact}{Facts}
\Crefname{claim}{Claim}{Claims}

\numberwithin{equation}{section}

\newcommand{\eps}{\varepsilon}

\renewcommand{\restriction}{\mathord{\upharpoonright}}
\newcommand{\norm}[1]{\| #1 \|}

\graphicspath{ {./Figures/} }

\def\COMMENT#1{}

\begin{document}

\begin{abstract}
    We prove that, whenever $ p \ge n^{-1/2 + o(1)} $, with high probability $ G(n, p) $ admits a fractional triangle decomposition, that is, a non-negative weight function on its triangles for which the total weight of all triangles containing each edge is equal to 1.
    This bound on $ p $ is optimal up to the asymptotic error term, improving upon the recent state of the art, due to Mahabaduge and Simkin, that $ p \ge n^{-4/11 + o(1)} $ suffices.

    Our main tool is a deterministic theorem guaranteeing the existence of fractional clique decompositions in all hypergraphs satisfying suitable \emph{clique-regularity} properties.
    We prove this by analysing an extension (and generalisation to hypergraphs) of an algorithm proposed by Mahabaduge and Simkin, in which, at each time step, the discrepancy at each edge is spread among its containing triangles.

    By showing the concentration of the relevant quantities in random $ k $-uniform hypergraphs, we obtain for all $ k \ge 2 $ and $ r \ge k + 1 $ that w.h.p.\ $ G^{(k)}(n, p) $ admits a fractional $ K^{(k)}_r $-decomposition whenever $ p \ge n^{-\frac{r - k}{\binom{r}{k} - 1} + o(1)} $, which improves upon results of Delcourt, Kelly, and Postle, and is best possible up to subpolynomial factors.
\end{abstract}

\maketitle

\section{Introduction} \label{section_intro}

Many classic questions in combinatorics ask under what conditions it is possible to decompose the edges of a graph $ G $ into copies of some natural structure.
For example, in a \emph{triangle decomposition}, or more generally a \emph{$ K_r $-decomposition} for $ r \ge 3 $, we seek a collection of triangles (respectively copies of $ K_r $) in $ G $ such that every edge is contained in exactly one triangle (respectively copy of $ K_r $).
When $ G = K_n $, a triangle decomposition is known as a \emph{Steiner triple system}, the existence of which was proven for all $ n $ satisfying the obviously necessary divisibility conditions by Kirkman~\cite{kirkman} in 1847.
This laid the foundations for the field of design theory, which has seen huge progress in recent years~\cite{delcourt2024cliquedecompositionsrandomgraphs,delcourt2024thresholdsnq2steinersystemsrefined,delcourt2024proofhighgirthexistence,glock2020existencedesignsiterativeabsorption,keevash2024shortproofexistencedesigns,postle2025refinedabsorptionnewproof}, especially following the proof of the existence of so-called combinatorial designs in general by Keevash~\cite{keevash2024existencedesigns} in 2014, which generalised the work of Kirkman, as well as that of Wilson~\cite{Wilson1972PBDI,Wilson1972PBDII,Wilson1975PBDIII} in the 1970s.

In general, even the problem of determining whether $ G $ admits a triangle decomposition is NP-hard, so it is natural to ask whether certain density or pseudorandomness conditions on $ G $ are sufficient to guarantee the existence of clique decompositions.
In the dense case, a major line of research in recent years has sought to find minimum degree thresholds above which all graphs satisfying the obviously necessary divisibility conditions admit $ K_r $-decompositions.
In fact, much of this research~\cite{Dukes2012RationalDecomposition,drossfractional,BarberEtAl2017FractionalCliques,Montgomery2019FractionalCliques,DukesHorsley2020TriangleDecomposition,delcourtprogress,DelcourtLesgourguesPostle2026FractionalCliques} has focussed on obtaining the minimum degree threshold for a graph to admit a \emph{fractional triangle decomposition}, or more generally a \emph{fractional $ K_r $-decomposition}, following pivotal work of Barber, Kühn, Lo, and Osthus~\cite{BARBER2016337} from 2016 which uses their iterative absorption method to turn a fractional decomposition into an integral one, building upon a result of Haxell and Rödl~\cite{HaxellRodl2001IntegerFractionalPackings} from 2001.
A \emph{fractional $ K_r $-decomposition} is a non-negative weight function $ \varphi $ on the copies of $ K_r $ in $ G $ with the property that, for each edge $ e \in E(G) $, the sum of the weights of cliques containing $ e $ is exactly 1.
Observe that a $ K_r $-decomposition is exactly a fractional $ K_r $-decomposition whose image is a subset of $ \{ 0, 1 \} $, so seeking the fractional variant is a natural relaxation of the integral problem, as is common for many objects in combinatorics.
Proving the existence of fractional triangle decompositions in sufficiently dense graphs was also a central component of the very recent breakthrough by Delcourt~and~Postle~\cite{delcourt2026proofnashwilliamsconjecture} confirming Nash-Williams' famous conjecture from 1970, which asserts that every graph on $ n $ vertices (for $ n $ large enough) with minimum degree at least $ \frac{3}{4} n $ and satisfying the obvious divisibility conditions has a triangle decomposition.

In the sparse counterpart, it has also been asked~\cite{delcourt2024cliquedecompositionsrandomgraphs,yuster2007} whether, given $ r \ge 3 $, there is a threshold probability above which the binomial random graph $ G(n, p) $ admits a $ K_r $-decomposition \emph{with high probability (w.h.p.)}, that is, with probability tending to $ 1 $ as $ n \to \infty $.
Since w.h.p.\ the random graph $ G(n, p) $ does not satisfy the required divisibility conditions, we cannot expect there to exist a threshold for the property of containing an exact $ K_r $-decomposition, and must instead look for a weaker property, for example by excluding a \emph{leftover} of at most linearly many edges from our decomposition~\cite{delcourt2024cliquedecompositionsrandomgraphs}, or indeed by seeking a fractional decomposition.
The specific problem of finding a sharp threshold for a fractional triangle decomposition was posed by Yuster~\cite{yuster2007} in 2007.
Even the existence of a threshold is non-trivial, since admitting a fractional clique decomposition is a non-monotone property.
Observe that an obvious necessary condition for the existence of a fractional (or integral) $ K_r $-decomposition is that every edge is contained in a copy of $ K_r $, for which it is easy to show that the threshold in $ G(n, p) $ is $ \Theta(n^{-\frac{2}{r+1}} \log^{1/(\binom{r}{2} - 1)}{n}) $.
It is widely believed that this condition should also be sufficient~\cite{delcourt2024cliquedecompositionsrandomgraphs,mahabaduge2025fractionaltriangledecompositionsrandom,yuster2007}, and Mahabaduge and Simkin~\cite{mahabaduge2025fractionaltriangledecompositionsrandom} explicitly conjectured the following sharp threshold in the triangle case, in response to the problem of Yuster~\cite{yuster2007}.

\begin{conjecture}[Mahabaduge and Simkin~\cite{mahabaduge2025fractionaltriangledecompositionsrandom}] \label{conj:triangle_threshold}
    For every $ \eps > 0 $ and $ p \ge (1 + \eps) \sqrt{\frac{3 \log{n}}{2n}} $, w.h.p.\ $ G(n, p) $ admits a fractional triangle decomposition.
\end{conjecture}

There has been some progress towards \cref{conj:triangle_threshold} but, until now, even obtaining the correct power of $ n $ was out of reach.
Recently, Delcourt, Kelly, and Postle~\cite{delcourt2024cliquedecompositionsrandomgraphs} used their refined absorption framework to prove the existence of clique decompositions (with a linear leftover) in $ G(n, p) $ for sufficiently large $ p $ (a polynomial factor above the conjectured threshold).
Their method also allowed them to show that $ G(n, p) $ admits a fractional triangle decomposition w.h.p.\ for $ p \ge n^{-1/3 + o(1)} $, and more generally a fractional $ K_r $-decomposition for $ p \ge n^{-\frac{1}{r+0.5} + o(1)} $.
Since the main focus of their work was finding integral decompositions, their approach is limited by the use of denser absorbing structures, which they remark should not inherently be necessary to obtain fractional decompositions. 
Very recently, Mahabaduge and Simkin~\cite{mahabaduge2025fractionaltriangledecompositionsrandom} improved the required probability for a fractional triangle decomposition to $ p \ge n^{-4/11 + o(1)} $ by analysing an algorithm for redistributing edge weights in an approximate fractional decomposition using gadgets based on so-called `octagonal pinwheel' graphs.
Our main contribution is to make significant progress on Conjecture~\ref{conj:triangle_threshold}, and indeed its natural generalisation in two directions, showing that just a subpolynomial factor above the obviously necessary minimum probability is sufficient to ensure the existence of fractional clique decompositions in random $ k $-graphs.
We write $ G^{(k)}(n, p) $ for the random $ k $-uniform hypergraph (\emph{$ k $-graph}) on $ n $ vertices, where edges are included independently at random with probability $ p $, and write $ K^{(k)}_r $ for a $ k $-uniform clique on $ r $ vertices, or simply $ K_r $ if $ k $ is clear from context.
Our main result is the following.

\begin{theorem} \label{thm:gnp_fcd}
    For all integers $ k \ge 2 $ and $ r \ge k + 1 $, and any real $ \eps > 0 $, if $ p \ge n^{-\frac{r - k}{\binom{r}{k} - 1} + \eps} $, then w.h.p.\ $ G^{(k)}(n, p) $ admits a fractional $ K^{(k)}_r $-decomposition.
\end{theorem}

In particular, in the case $ k = 2 $ and $ r = 3 $, we obtain a fractional triangle decomposition whenever $ p \ge n^{-1/2 + o(1)} $, closing the polynomial gap in~\cite{mahabaduge2025fractionaltriangledecompositionsrandom}.
To prove \cref{thm:gnp_fcd}, we first state and prove a deterministic result (see \cref{thm:proc_conv_fcd}) guaranteeing the existence of fractional clique decompositions in a class of hypergraphs satisfying a set of pseudorandomness properties, which we refer to as \emph{clique-regularity}, and then prove that $ G^{(k)}(n, p) $ is clique-regular w.h.p.\ (see \cref{lem:gnp_reg}).
We prove \cref{thm:proc_conv_fcd} by analysing a modification of (the natural generalisation of) a simple weight-distributing algorithm proposed at the end of \cite{mahabaduge2025fractionaltriangledecompositionsrandom}.
Roughly speaking, the idea is as follows.
We start with a uniform weight function on the cliques, which yields an approximate fractional decomposition, since every edge is contained in approximately the same number of cliques.
The process proposed in~\cite{mahabaduge2025fractionaltriangledecompositionsrandom} then iteratively adds or subtracts weight from each edge $ e $ with too little or too much weight, respectively, by adding or subtracting weight uniformly from all cliques containing $ e $; this can be thought of as spreading the \emph{discrepancy} at $ e $ uniformly among its containing cliques.
In fact, we need a slight modification of this naïve algorithm to avoid too much weight being concentrated within the set of edges containing particular vertices.
Specifically, in alternate steps, we instead spread the discrepancy at each vertex uniformly among its containing cliques, where the discrepancy is defined relative to the expected total weight required at the vertex in a fractional clique decomposition.
In hypergraphs, we must further perform an analogous discrepancy-spreading step for sets of $ j $ vertices for each $ 1 \le j \le k $.

Our approach has two main advantages in comparison to the gadgets used by Mahabaduge and Simkin~\cite{mahabaduge2025fractionaltriangledecompositionsrandom}.
Firstly, close to the conjectured threshold, not every edge is contained in a copy of any pinwheel gadget of a constant size; even to obtain a bound $ p \ge n^{-1/2 + o(1)} $, such gadgets would have to be made arbitrarily large, which appears to make the analysis very complicated.
Secondly, such gadgets are tailored to fractional triangle decompositions, and do not have an obvious analogue for larger cliques or hypergraphs.
On the other hand, our approach relies only on the cliques themselves, which exist even at the conjectured sharp threshold.

We proceed to define our algorithm formally in the rest of this section, finishing with the statement of \cref{thm:proc_conv_fcd}.
This is followed by an overview of the proof in \cref{sect:proof_sketch}.
In \cref{section:weightings} we make various further definitions, which allow us to state the required clique-regularity properties, completing the formal statement of \cref{thm:proc_conv_fcd}.
We then use these concepts in \cref{sect:main_proof} to prove \cref{thm:proc_conv_fcd}.
Finally, in \cref{sect:hypergraph_psrand}, we complete the proof of \cref{thm:gnp_fcd}, by showing that w.h.p.\ the random hypergraph satisfies the required properties.

In the rest of this paper, we consider $ k \ge 2 $ and $ r \ge k + 1 $ to be global constants, and consider fractional $ K_r^{(k)} $-decompositions of $ k $-graphs.
Throughout the rest of \cref{section_intro,section:weightings,sect:main_proof}, let $ G $ be a $ k $-graph on vertex set $ V(G) \coloneqq [n] $ equipped with the canonical ordering such that every set $ S \subseteq V(G) $ with $ S \in E(G) $ or $ |S| < k $ is contained in some copy of $ K_r^{(k)} $ in $ G $.

\subsection{The algorithm} \label{section:process}

Given $ a, b \in \mathbb{Z} $, write $ [a, b] \coloneqq \{ x \in \mathbb{Z}: a \le x \le b \} $ and $ [a] \coloneqq [1, a] $.
Write $ N_G^{\mathrm{e}}(S) \coloneqq \{ e \in E(G): S \subseteq e \} $ for the edge neighbourhood of any set $ S \subseteq V(G) $ with size $ 0 \le |S| \le k $, as well as $ K_{r, G}(A) $ for the set of copies of $ K_r $ in $ G $ whose vertex set contains $ A \subseteq V(G) $, or just $ K_r(A) $ if $ G $ is clear from context, and $ K_r(G) \coloneqq K_{r, G}(\emptyset) $.
Say that $ \varphi: K_r(G) \to \mathbb{R} $ is a \emph{$ K_r $-function (on $ G $)} and write $ \varphi(e) \coloneqq \sum_{K \in K_r(e)} \varphi(K) $ for $ e \in E(G) $.
Define the \emph{discrepancy} $ \xi \coloneqq \xi[\varphi] \coloneqq \xi^{(k)}[\varphi]: E(G) \to \mathbb{R} $ by
$$ \xi(e) \coloneqq \varphi(e) - 1 $$
and further the \emph{$ j $-discrepancy} $ \xi^{(j)} \coloneqq \xi^{(j)}[\varphi]: \binom{V(G)}{j} \to \mathbb{R} $ by
$$ \xi^{(j)}(S) \coloneqq \sum_{e \in N_G^{\mathrm{e}}(S)} \xi(e) $$
for each $ S \subseteq V(G) $ with $ |S| = j \in [0, k - 1] $.
This represents (up to a constant factor) the difference between the amount of weight $ \varphi $ assigns to cliques containing a particular $ j $-set $ S $ and the amount of weight it should receive in a fractional $ K_r $-decomposition.
Given $ j \in [0, k] $, define also
$$ E_j \coloneqq E_j(G) \coloneqq \begin{cases} E(G), & j = k;\\ \binom{V(G)}{j}, & j \in [0, k - 1]. \end{cases} $$

In the following, we define a number of operators on the space of $ K_r $-functions.
Our ultimate goal is to start with some suitable $ \varphi $, which approximates a fractional $ K_r $-decomposition, and iteratively map it to a sequence of functions, whose limit is an (exact) fractional $ K_r $-decomposition.
Given $ s \in [k] $, define the \emph{one-step $ s $-distributor} $ \mathcal{R}_s: \mathbb{R}^{K_r(G)} \to \mathbb{R}^{K_r(G)} $ by setting, for any $ K_r $-function $ \varphi $ on $ G $ and $ K \in K_r(G) $,
$$ \mathcal{R}_s[\varphi](K) \coloneqq \varphi(K) - \frac{1}{\binom{r - s}{k - s}} \sum_{S \in \binom{V(K)}{s}} \frac{\xi^{(s)}(S)}{|K_r(S)|}. $$
Throughout the paper, we write $ \mathcal{R}_s \varphi \coloneqq \mathcal{R}_s [\varphi] $ for ease of notation.
To explain the normalising factor, note that, given $ S \in \binom{V(G)}{s} $ and $ K \in K_r(S) $, there are $ \binom{r - s}{k - s} $ edges in $ K $ containing $ S $; in particular, the sum $ \xi^{(s)}(S) $ contains $ \binom{r - s}{k - s} $ terms $ \xi(e) $ with $ e \in E(K) $.
Intuitively, for each $ s $-set $ S $ (or edge in the case $ s = k $), we think of $ \mathcal{R}_s $ as subtracting the discrepancy of $ S $ from every $ r $-clique containing it, divided by the total number of such cliques (taking multiplicities into account).
In this way, the discrepancy at $ S $ is spread among all $ s $-sets (or edges for $ s = k $) which share a clique with $ S $.
Given $ \ell \in \mathbb{N} $, define the \emph{$ (s, \ell) $-distributor} $ \mathcal{P}_{s, \ell} $ for each $ s \in [k] $ inductively, by writing $ \mathcal{P}_{0, \ell}: \mathbb{R}^{K_r(G)} \to \mathbb{R}^{K_r(G)} $ for the identity operator and setting
$$ \mathcal{P}_{s, \ell} \coloneqq (\mathcal{R}_s^{\ell} \mathcal{P}_{s - 1, \ell})^{\ell} \text{ for each } s \in [k]. $$
For reasons which will become clear later (see \cref{sect:proof_sketch}), in order to distribute the edge discrepancies, we first want $ \xi^{(j)} $ to be small for all $ j \in [k - 1] $.
This motivates our iterative definition, in which $ \mathcal{P}_{j, \ell} $ is used to reduce $ \xi^{(j)} $ for each $ j \in [k] $, one at a time.

% To prove \cref{thm:gnp_fcd}, we prove more generally the deterministic existence of fractional clique decompositions in a class of hypergraphs satisfying a set of pseudorandomness properties (see \cref{thm:proc_conv_fcd}), which $ G^{(k)}(n, p) $ satisfies w.h.p.\ (see \cref{lem:gnp_reg}).
We now proceed to state our key deterministic theorem, for which we require some further definitions.
Say that a $ K_r $-function $ \varphi $ is \emph{globally-balanced} if it has the same total weight as a fractional clique decomposition, that is, $ \sum_{K \in K_r(G)} \varphi(K) = \binom{r}{k}^{-1}|E(G)| $, and note that this implies that $ \sum_{e \in E(G)} \xi(e) = 0 $.
Further, say formally that $ \varphi $ is a \emph{fractional $ K_r $-decomposition} if $ \varphi $ is non-negative and $ \varphi(e) = 1 $ (equivalently, $ \xi(e) = 0 $) for every $ e \in E(G) $.
Say that $ \varphi $ is \emph{uniform} if it is a constant function, and observe that every $ k $-graph $ G $ containing a copy of $ K_r $ has a unique globally-balanced uniform $ K_r $-function; specifically, every clique receives weight $ \frac{1}{\gamma} $, writing $ \gamma $ for the average number of cliques containing each edge.
We write $ a \ll b $ to mean that, given any $ b > 0 $, there exists $ a_0 > 0 $ such that, for any $ 0 < a \le a_0 $, the subsequent statement holds; this extends in the obvious way to hierarchies with more variables.

\begin{theorem} \label{thm:proc_conv_fcd}
    Let $ k \ge 2, r \ge k + 1 $, and suppose $ 1/n, \eps \ll 1/C \ll 1/\ell \ll 1/c, 1/r $.
    Let $ G $ be an $ (\ell, c, C, \eps) $-clique-regular $ k $-graph on $ n $ vertices and $ \varphi $ be the globally-balanced uniform $ K_r $-function on $ G $.
    Then the sequence $ (\mathcal{P}_{k, \ell}^t \varphi)_{t \ge 0} $ converges to a fractional $ K_r $-decomposition of $ G $ as $ t \to \infty $.
\end{theorem}

We defer the definition of $ (\ell, c, C, \eps) $-clique-regularity to \cref{section:weightings}.
This is a strong notion of pseudorandomness, which in particular is exhibited w.h.p.\ by the binomial random hypergraph (for the appropriate range of $ p $), but not necessarily by pseudorandom graphs as usually considered in the literature.
We begin now by giving a rough sketch of our proof in \cref{sect:proof_sketch}.

\subsection*{Notation}
Recall that we fix $ k \ge 2 $ and $ r \ge k + 1 $ for the remainder of this paper.
The $ k $-graph $ G $ always has vertex set $ V(G) = [n] $ with the usual ordering, unless otherwise specified.

\section{Proof overview} \label{sect:proof_sketch}

Here we attempt to give some intuition for the key ideas involved in our proof.
We do this by first presenting the naïve (and not entirely correct) idea behind the proof, then explaining how each component of the final proof arises as the solution to a problem encountered in this approach.
For intuition, we work in a random hypergraph $ G^{(k)}(n, p) $ as in \cref{thm:gnp_fcd}, although \cref{thm:proc_conv_fcd} is more general.

\textbf{The idea:}
We start with the unique globally-balanced uniform $ K_r $-function $ \varphi_0 $, which is already an approximate decomposition because, recalling that we write $ \gamma $ for the average number of cliques containing an edge, it is easy to see that $ |\xi(e)| = \left| \frac{|K_r(e)|}{\gamma} - 1 \right| \le n^{-\delta} $ in a random hypergraph $ G $, for some small $ \delta > 0 $.
We now generate a sequence of $ K_r $-functions by setting $ \varphi_{t + 1} \coloneqq \mathcal{R}_k \varphi_t $ for each $ t \ge 0 $, and aim firstly to prove that they converge to a decomposition.
Indeed, we may rewrite the discrepancy $ \xi[\mathcal{R}_{k} \varphi_t](e) $ as a sum of discrepancies $ \xi[\varphi_t](f) $ of edges $ f $ for which there is a clique in $ G $ containing $ e \cup f $, weighted according to $ \frac{1}{|K_r(f)|} $; since all such clique counts are highly concentrated, we may treat these weights as a constant multiplier, which we ignore for the remainder of this summary.
In particular, iterating this $ \ell $ times, we see that $ \xi[\mathcal{R}_{k}^{\ell} \varphi_t](e) $ is a (weighted) sum of discrepancies among edges $ f $ for which there exists a \emph{clique-path} of length $ \ell $ from $ e $ to $ f $, that is, a sequence of cliques $ F_1, \ldots, F_{\ell} $ in $ G $ for which the first contains $ e $, the last contains $ f $, and adjacent cliques $ F_i, F_{i + 1} $ intersect in an edge $ S_i $.
Our first hope is that, for sufficiently large (constant) $ \ell $, the number of such clique-paths should be very well-concentrated for all $ e, f \in E(G) $, so we may write $ \xi[\varphi_{t + \ell}](e) = \sum_{f \in E(G)} (1 \pm n^{-\delta}) w \xi[\varphi_t](f) $ for some constant $ w \in \mathbb{R} $.
Since each $ \varphi_t $ is globally-balanced, the sum over all discrepancies is zero, and we would just be left with the error term, which would be sufficient to obtain $ \norm{\xi[\varphi_{t + \ell}]}_{\infty} \le n^{-\delta} \norm{\xi[\varphi_t]}_{\infty} $.
As such, the process would indeed converge to a $ K_r $-function $ \varphi $ with zero discrepancy, and since $ \varphi_0 $ already has very small discrepancies, it is not hard to further show that the process yields non-negative clique weights, which means that $ \varphi $ would be a decomposition.

\textbf{Problem 1:} Such a strong concentration on the number of clique-paths is too much to hope for, because too many of the clique-paths are \emph{$ b $-pivoting} for some $ b \in [k - 1] $, that is, $ |e \cap S_1 \cap \ldots \cap S_{\ell - 1} \cap f| = b $; thus the number of clique-paths from $ e $ to $ f $ is, for example, disproportionately large for edges $ f $ with $ e \cap f \ne \emptyset $.

\textbf{Solution:} Instead of always applying $ \mathcal{R}_k $, we define instead $ \varphi_{t + 1} \coloneqq \mathcal{R}_k^{\ell} \mathcal{P}_{k - 1, \ell} \varphi_t $ and adopt an inductive approach, assuming for now (the induction hypothesis) that the operator $ \mathcal{P}_{k - 1, \ell} $ significantly reduces the $ (k - 1) $-discrepancy, that is $ \norm{\xi^{(s)}[\mathcal{P}_{k - 1, \ell} \varphi]}_{\infty} \le n^{-\delta} \norm{\xi^{(s)}[\varphi]}_{\infty} $ for every $ s \in [k - 1] $.
We may partition the set of clique-paths of length $ \ell $ into subsets $ \mathcal{F}_b $ depending on the number $ b $ of vertices around which they pivot.
For $ b \in [k - 1] $ and $ \ell $ sufficiently large, we may hope to obtain concentration on the number of clique-paths in $ \mathcal{F}_b $ for all edges $ e, f $ with $ |e \cap f| \ge b$; in other words (again, roughly speaking), we hope to write
$$ \xi[\mathcal{R}_k^{\ell} \varphi](e) = \sum_{b \in [0, k - 1]} C_b \sum_{U \in \binom{e}{b}} \sum_{f \in E(G): U \subseteq f } (1 \pm n^{-\delta}) \xi[\varphi](f) $$
for some multiplicative factors $ C_b $.
Up to an error term, we could then rewrite the discrepancy in terms of $ b $-discrepancies for $ b \in [0, k - 1] $, and thus obtain that
$$ \norm{\xi[\varphi_{t + 1}]}_{\infty} \le O(1) \sum_{b \in [0, k - 1]} \norm{\xi^{(b)}[\mathcal{P}_{k - 1, \ell} \varphi_t]}_{\infty} + n^{-\delta} \norm{\xi[\mathcal{P}_{k - 1, \ell} \varphi_t]}_{\infty} \le O(1) n^{-\delta} \norm{\xi[\varphi_t]}_{\infty}. $$

\textbf{Problem 2:} Such a strong concentration on the number of clique-paths is still too much to hope for, because some clique-paths may require $ e $ and $ f $ to be `at small distance' in $ G $.
For example, clique-paths with $ e \cap S_1 \cap \ldots \cap S_{\ell - 1} \ne \emptyset $ but $ S_{\ell - 1} \cap f = \emptyset $ may only exist between edges $ e, f $ for which $ f $ shares a clique with at least one vertex of $ e $.

\textbf{Solution:} We further redefine $ \varphi_{t + 1} \coloneqq (\mathcal{R}_k^{\ell} \mathcal{P}_{k - 1, \ell})^{\ell} \varphi_t $.
In order to ensure concentration among all pairs of edges $ e, f \in E(G) $, we need to consider only clique-paths in which the \emph{distance} from $ e $ to $ f $ is at least a large constant $ \ell $, where the distance from $ e $ to $ f $ is defined as the length of the shortest \emph{walk} from any vertex of $ e $ to any vertex of $ f $, and a walk of length $ \ell $ in a hypergraph is a sequence of $ \ell $ edges in which adjacent edges have non-empty intersection.
Our next useful observation is that, for clique-paths in $ \mathcal{F}_0 $ in which $ e $ and $ f$ do not intersect, they must at least have distance 1.
This means that applying the operator $ \mathcal{R}_{k}^{\ell} \mathcal{P}_{k - 1, \ell} $ distributes the discrepancy at $ e $ among edges $ f $ sharing clique-paths with $ e $ in which $ e, f $ are distance at least 1 apart (up to error terms).
Thus, applying this operator $ \ell $ times distributes along clique-paths in which $ e, f $ are distance at least $ \ell $ apart, for which we are able to obtain concentration.
The ideas we have discussed so far are in fact sufficient for a proof in the graph case $ k = 2 $, but we encounter one further issue when working with hypergraphs.

\textbf{Problem 3:} In the hypergraph case, there is an analogous issue for $ b $-pivoting clique-paths.
For example, clique-paths with $ |e \cap S_1 \cap \ldots \cap S_{\ell - 1}| = b + 1 $ but $ |S_{\ell - 1} \cap f| = b $ only appear if there also exists an edge intersecting both of $ e \setminus f $ and $ f \setminus e $.
Indeed, in this case there exists at least one vertex $ u \in e \cap S_{\ell - 1} \subseteq V(F_{\ell}) $ with $ u \not \in f $ and at least one vertex $ v \in f \setminus e \subseteq V(F_\ell) $, so $ u $ and $ v $ must belong to at least one edge in the clique $ F_\ell $.

\textbf{Solution:} We exploit a similar concept of distance between $ e \setminus f $ and $ f \setminus e $, considering only walks avoiding $ e \cap f $.
We prove the following Ramsey-type statement for any sufficiently long clique-path $ F $.
Consider starting with $ k $ vertices of $ e $ and repeatedly swapping out vertices to obtain sets $ S_1, \ldots, S_{t - 1}, f $; there must exist some interval $ [a, b] \subseteq [0, t] $ of length at least $ \frac{t}{c} $ upon which every vertex is either constant (i.e. belongs to $ S_a \cap S_b $), or is swapped at least $ c $ times (i.e. the distance in the clique-path between $ S_a \setminus S_b $ and $ S_b \setminus S_a $ is large).
In particular, for any clique-path $ F $ contributing weight to $ (\mathcal{R}_k^\ell \mathcal{P}_{k - 1, \ell})^{\ell} $, there exists some \emph{subpath} $ F' \subseteq F $, corresponding to $ (\mathcal{R}_{k}^\ell \mathcal{P}_{k - 1, \ell})^{i} \mathcal{R}_{k}^\ell $ for some $ i \in [\ell] $, with the property that vertices of $ e \setminus f $ and $ f \setminus e $ are at large distance.
This allows us to show that the number of copies of $ F' $ is concentrated.
Using the concentration of the $ F' $ obtained from all possible $ F $, we may then write the discrepancy $ \xi[\varphi_{t + 1}] $ as a weighted sum of discrepancies $ \xi^{(b)}[\mathcal{P}_{k - 1, \ell} (\mathcal{R}_{k}^\ell \mathcal{P}_{k - 1, \ell})^{\ell - i - 1} \varphi_t] $ for $ b \in [0, k - 1] $, all of which are very small, which turns out to be sufficient.

See the start of \cref{sect:main_proof} for an outline of the structure of the proof itself, as well as the start of \cref{sect:proof_lem_decrease} for more details on the inductive proof of the central lemma.

\section{Clique regularity} \label{section:weightings}

In this section, we assume throughout that $ \varphi $ is an arbitrary $ K_r $-function on $ G $.
It is not hard to see that, for each $ 1 \le s \le j \le k $, the discrepancy $ \xi^{(j)}[\mathcal{R}_s \varphi] $ can be written as a weighted sum of the discrepancies $ \xi^{(j)}[\varphi] $.
The goal of this section is to characterise the weights in this sum, as well as those for $ \mathcal{P}_{j - 1, \ell} $ and $ \mathcal{P}_{j, \ell} $, in terms of sequences of functions depending on the $k$-graph $ G $; we then write our desired pseudorandomness property in terms of these functions.
We start by defining some useful general notation.

\subsection{Notation} \label{section:notation}

Given a finite set $ X $ and function $ \psi: X \to \mathbb{R} $, as well as $ \alpha, \beta: X^2 \to \mathbb{R} $, note that we may equivalently regard $ \psi $ as a real-valued vector $ \mathbf{v} \coloneqq (v_x)_{x \in X} $ and likewise $ \alpha, \beta $ as real-valued matrices $ A, B $, each indexed by elements of $ X $, where $ v_x \coloneqq \psi(x) $, $ A \coloneqq (a_{xy})_{x, y \in X} $ for $ a_{xy} \coloneqq \alpha(x, y) $, and $ B $ is defined analogously for $ \beta $.
We may thus write $ \langle \alpha, \psi \rangle: X \to \mathbb{R} $ for the function corresponding to the vector $ Av $ and $ \alpha \circ \beta: X^2 \to \mathbb{R} $ for the matrix multiplication $ AB $.
Given a scalar $ \lambda \in \mathbb{R} $ write $ \lambda \cdot \psi $ for the function $ x \mapsto \lambda \psi(x) $.
Observe that the identity $ \langle \alpha, \langle \beta, \psi \rangle \rangle \equiv \langle \alpha \circ \beta, \psi \rangle $ is exactly the associativity of matrix-vector multiplication; we make use of this throughout.
Write $ \norm{\psi} \coloneqq \norm{\psi}_{\infty} \coloneqq \max_{S \in X} |\psi(S)| $, and $ \norm{\beta} \coloneqq \max_{S \in X} \sum_{T \in X} |\beta(S, T)| $, which we distinguish from $ \norm{\beta}_{\infty} \coloneqq \max_{S, T \in X} |\beta(S, T)| $.
Given a vector $ \mathbf{x} = (x_1, \ldots, x_{\ell}) \in X^{\ell} $, write $ |\mathbf{x}| \coloneqq \ell $ for its \emph{length}, and given $ I \subseteq [\ell] $ write $ \mathbf{x}_{I} $ for the \emph{subvector} $ (x_i)_{i \in I} $.
Write $ (x)^{\ell} $ for the vector $ (x, \ldots, x) $ of length $ \ell $.
Note that, for convenience, we sometimes index vectors starting from 0 instead of 1; it will be clear when this is the case.

Given a hypergraph $ H $, distinct $ x, y \in V(H) $, $ U \subseteq V(H) $, and $ d \ge 0 $, define a \emph{$ U$-avoiding walk of length $ d $ from $ x $ to $ y $ in $ H $} to be a sequence $ v_0, \ldots, v_d \in V(H) \setminus U $ for which $ x = v_0, y = v_d $, and $ v_i, v_{i - 1} \in e_i $ for some edge $ e_i \in E(H) $ for each $ i \in [d] $; note that we allow the edges $ e_i $ to intersect $ U $, and that both vertices and edges may be repeated.

A \emph{multiset $ X $} is a pair $ (\mathrm{Set}(X), m_X) $ where $ \mathrm{Set}(X) $ is a set and $ m_X \coloneqq m: \mathrm{Set}(X) \to \mathbb{N} $ gives the \emph{multiplicity} of each element.
Say that $ x \in X $ if $ x \in \mathrm{Set}(X) $, and write $ m(x) = 0 $ for any $ x \not \in X $.
Unless otherwise stated, the \emph{size of $ X $} is $ |X| \coloneqq \sum_{x \in \mathrm{Set}(X)} m(x) $.
Given reals $ a_x $ for each $ x \in X $, the \emph{multiset sum} $ \sum_{x \in X} a_x \coloneqq \sum_{x \in \mathrm{Set}(X)} m(x) a_x $.
Given multisets $ A_x $ for each $ x \in X $, the \emph{multiset union} $ \bigcup_{x \in X} A_x $ refers to the multiset $ (\bigcup_{x \in \mathrm{Set}(X)} \mathrm{Set}(A_x), m') $ with $ m'(y) \coloneqq \sum_{x \in X} m_{A_x}(y) = \sum_{x \in \mathrm{Set}(X)} m_X(x) m_{A_x}(y) $.
Given multisets $ X, Y $, the \emph{multiset product} $ X \times Y $ is the multiset with $ \mathrm{Set}(X \times Y) = \mathrm{Set}(X) \times \mathrm{Set}(Y) $ and $ m_{X \times Y}(x, y) = m_X(x) \cdot m_Y(y) $.
For $ i \in \mathbb{N}_0 $, we also write $ X^i \coloneqq X \times \cdots \times X $, where the product consists of $ i $ copies of $ X $.

Given an interval $ I = [a, b] \subseteq \mathbb{Z} $, write $ \mathrm{Len}(I) \coloneqq b - a $ for the \emph{length}.
In general, given an ordered set $ S = \{ x_1, \ldots, x_j \} $ of size $ j $ and a subset $ U \subseteq S $, write $ \iota(U, S) \coloneqq \{ i \in [j]: x_i \in U \} $ to represent the location of $ U $ within the ordering on $ S $.
We omit floor and ceiling notation when it does not affect the argument.

\subsection{Weighted clique-paths} \label{section:clique_paths_def}

As discussed, for $ s \in [k] $, the one-step $ s $-distributor spreads the discrepancy of an $ s $-set among cliques containing it; as such, multiple stages of our process can be thought of as spreading the discrepancy of a given $ s $-set among \emph{clique-paths}, formed by a sequence of cliques, in which adjacent cliques have intersections of specified sizes between $1$ and $ k $.
We now make some definitions to formalise this idea.

\emph{Clique-paths, concatenation, and subpaths:}
Given $ j \in [k] $ and $ \ell \ge 0 $, we define a \emph{$ j $-clique-path $ \mathbf{F} $ of length $ \ell $} to be a pair $ \mathbf{F} \coloneqq ((F_i)_{i \in [\ell]}, \mathbf{S}) $ consisting of sequences $ (F_i)_{i \in [\ell]} $ of copies of $ K_r $ and $ \mathbf{S} \coloneqq (S_i)_{i \in [0, \ell]} $ of pairwise \emph{distinct} ordered sets of size $ j $, respectively, such that $ V((F_1 \cup \cdots \cup F_i) \cap F_{i + 1}) \subseteq V(F_i \cap F_{i + 1}) \subseteq S_i \subseteq V(F_{i + 1}) $ for each $ i \in [0, \ell] $, writing for convenience $ F_0 $ and $ F_{\ell + 1} $ to denote the cliques on vertex sets $ S_0 $ and $ S_{\ell} $, respectively.
We also require that the orderings of $ S_i $ and $ S_{i'} $ agree on the intersection $ S_i \cap S_{i'} $ whenever this is non-empty for $ i, i' \in [0, \ell] $.
Write $ F \coloneqq F_1 \cup \cdots \cup F_{\ell} \cup F_{\ell + 1} $ and $ \mathbf{s} \coloneqq \mathbf{s}(F) \coloneqq (|V(F_i \cap F_{i + 1})|)_{i \in [\ell]} $, and call the sets $ S_i $ the \emph{root sets} of $ \mathbf{F} $.
We identify the clique-path with the pair $ (F, \mathbf{S}) $, observing that this information uniquely determines $ (F_i)_{i \in [\ell]} $, since $ F_i $ is the unique $ r $-clique in $ F $ with $ S_{i - 1} \subseteq V(F_i) $.
Intuitively, one can imagine clique-paths as being constructed iteratively as follows.
Start with an ordered set $ S_0 $ and take a clique $ F_1 $ containing $ S_0 $, then choose a set $ U \subseteq V(F_1) $ of size $ s_1 $, extend $ U $ to a set $ S_1 $ of size $ j $ by adding $ j - s_1 $ new vertices, and endow $ S_1 $ with an ordering (compatible with the ordering of $ S_0 $ on $ S_0 \cap S_1 $).
We then take another clique $ F_2 $ containing $ S_1 $, and iterate this process in the obvious way; see \cref{lem:process_graphs_char} for further motivation of this definition.
Note that we allow $ \ell $ to be zero, and in this case $ \mathbf{F} $ consists only of an (ordered) $ j $-set $ S_0 $.

Regard two $ j $-clique-paths $ (F, \mathbf{S}) $ and $ (F', \mathbf{S}') $ as \emph{isomorphic} if there exists a hypergraph isomorphism $ \psi $ from $ F $ to $ F' $ which maps $ S_i $ to $ S'_i $ in the unique order-preserving way for each $ i \in [0, \ell] $; henceforth, we consider clique-paths only up to isomorphism.
Given $ j $-clique-paths $ \mathbf{F}^1 = (F^1, \mathbf{S}^1) $ of length $ \ell_1 $ and $ \mathbf{F}^2 = (F^2, \mathbf{S}^2) $ of length $ \ell_2 $, define the \emph{concatenation} $ \mathbf{F}^1 \bullet \mathbf{F}^2 $ to be the unique $ j $-clique-path $ (F', \mathbf{S}') $ obtained by gluing in the obvious way.
Specifically, take $ F' $ to be a hypergraph formed by the union of $ F^1 $ and $ F^2 $, in which we identify $ S^1_{\ell_1} $ with $ S^2_0 $ in the unique order-preserving way, but regard the vertex sets as otherwise disjoint, and set $ \mathbf{S}' \coloneqq (S^1_0, \ldots, S^1_{\ell_1}, S^2_1, \ldots, S^2_{\ell_2}) $.
It is easy to see that $ \bullet $ is associative, which we will use throughout.
Given sets $ \mathcal{F}_1, \mathcal{F}_2 $ of $ j $-clique-paths, write $ \mathcal{F}_1 \bullet \mathcal{F}_2 \coloneqq \bigcup_{\mathbf{F}^1 \in \mathcal{F}_1, \mathbf{F}^2 \in \mathcal{F}_2} \mathbf{F}^1 \bullet \mathbf{F}^2 $.
Given $ 0 \le a < b \le \ell $, define the \emph{subpath} $ \mathbf{F}_{[a, b]} $ to be the $ j $-clique-path $ ((F_i)_{i \in [a + 1, b]}, \mathbf{S}_{[a, b]}) $.

Given $ 1 \le s \le j \le k $, write $ \mathcal{F}_{s}^{(j)} $ for the set of all possible $ j $-clique-paths $ \mathbf{F} = ((F_1), (S_0, S_1)) $ of length $ 1 $ for which $ |V(F_1) \cap S_1| = s $.
Observe that each such $ F $ is the union of a copy of $ K_r $ with a set of $ j - s $ isolated vertices, but that multiple (non-isomorphic) such $ \mathbf{F} $ are possible, depending on the intersection $ S_0 \cap S_1 $ and the orderings on $ S_0 $ and $ S_1 $.
Given $ \ell \ge 0 $ and $ \mathbf{s} \in [j]^{\ell} $, define further $ \mathcal{F}_{\mathbf{s}}^{(j)} \coloneqq \mathcal{F}_{s_1}^{(j)} \bullet \ldots \bullet \mathcal{F}_{s_{\ell}}^{(j)} $.
Note that this is exactly the set of clique-paths of length $ \ell $ with $ |V(F_i \cap F_{i + 1})| = s_i $ for each $ i \in [\ell] $, that is, $ \mathbf{s}(F) = \mathbf{s} $.
Intuitively, this definition corresponds directly to the iterative construction above.
We remark that again, $ \ell $ may be zero, in which case $ \mathbf{s} = () $ is the empty sequence and $ \mathcal{F}_{\mathbf{s}}^{(j)} $ is a singleton containing the unique $ j $-clique-path of length zero.

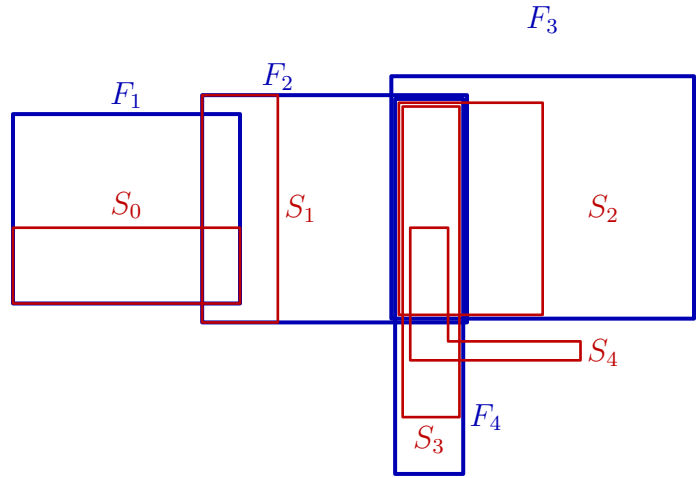
\begin{figure}[!ht]
    \centering

    \begin{tikzpicture}[
        x=1cm,y=1cm,
        Fline/.style={draw=blue!70!black, line width=1.5pt, line cap=round, line join=round},
        Sline/.style={draw=red!75!black,  line width=1pt, line cap=round, line join=round},
        Flab/.style={blue!70!black, font=\large},
        Slab/.style={red!75!black,  font=\large}
    ]
        
        % -------------------------------------------------
        % Blue regions F_i
        % -------------------------------------------------
        
        % F_1
        \draw[Fline] (0,0) rectangle (3,2.5);
        % F_2
        \draw[Fline] (2.5,-0.25) rectangle (6,2.75);
        % F_3
        \draw[Fline] (5,-0.2) rectangle (9,3);
        % F_4
        \draw[Fline] (5.05,-2.25) rectangle (5.95,2.7);
        
        % -------------------------------------------------
        % Red strips / paths s_i
        % -------------------------------------------------
        
        % s_0
        \draw[Sline] (0,0) rectangle (3,1);
        % s_1
        \draw[Sline] (2.5, -0.25) rectangle (3.5, 2.75);
        % s_2
        \draw[Sline] (5.1,-0.15) rectangle (7, 2.65);
        % s_3
        \draw[Sline] (5.15,-1.5) rectangle (5.9,2.6);
        %s_4
        \draw[Sline]
            (5.25,1) -- (5.75,1) -- (5.75,-0.5) -- (7.5,-0.5) -- (7.5,-0.75) -- (5.25,-0.75) -- cycle;
        
        % -------------------------------------------------
        % Labels
        % -------------------------------------------------
        
        \node[Flab] at (1.5,2.75) {$F_1$};
        \node[Flab] at (3.5,3) {$F_2$};
        \node[Flab] at (7,3.75) {$F_3$};
        \node[Flab] at (6.25,-1.5) {$F_4$};
        
        \node[Slab] at (1.5,1.3) {$S_0$};
        \node[Slab] at (3.8,1.25) {$S_1$};
        \node[Slab] at (7.8,1.25) {$S_2$};
        \node[Slab] at (5.5,-1.8) {$S_3$};
        \node[Slab] at (7.8,-0.65) {$S_4$};
        
    \end{tikzpicture}
    
    \caption{A $ j $-clique-path $ (F, \mathbf{S}) $; the blue $ F_i $ are $ r $-cliques and the red $ S_i $ are sets of size $ j $.
    Each clique $ F_{i + 1} $ may contain any set of up to $ j $ vertices of the previous clique $ F_i $, but these must be contained in the root set $ S_i $, which in turn is contained in $ F_{i + 1} $.}
    \label{fig:clique_path}
\end{figure}

\emph{Semi-copies:}
Given a $ j $-clique-path $ \mathbf{F} = (F, \mathbf{S}) $ of length $ \ell $ and sets $ S', T' \in E_j(G) $, define a \emph{homomorphism from $ \mathbf{F} $ to $ G $ rooted at $ S' $ and $ T' $} to be a (not necessarily injective) homomorphism $ \Phi $ from $ F $ to $ G $ mapping the sets $ S_0 $ and $ S_{\ell} $ to $ S' $ and $ T' $, respectively, such that $ \Phi \restriction_{S_i} $ is (injective and) order-preserving for each $ i \in [0, \ell] $, with respect to the canonical ordering on $ V(G) = [n] $.
Consider rooted homomorphisms $ \Phi_1, \Phi_2 $ to be equivalent if they differ only by permutations of the vertex sets $ V(F_i) \setminus (S_{i - 1} \cup S_i) $ for each $ i \in [\ell] $, and call each equivalence class a \emph{semi-copy of $ \mathbf{F} $ in $ G $ rooted at $ S' $ and $ T' $}; note that homomorphisms are equivalent if and only if $ \Phi_2 = \Phi_1 \psi $ for some clique-path isomorphism $ \psi $ (as defined above) from $ \mathbf{F} $ to itself.
Identify a semi-copy $ \Phi $ with the pair $ ((\Phi(F_i))_{i \in [\ell]}, (\Phi(S_i))_{i \in [0, \ell]}) $, noting that this is well-defined since $ \mathbf{S} $ is preserved by clique-path automorphisms, and that it uniquely determines $ \Phi $, and write $ X_G(S', T', \mathbf{F}) $ for the set of all semi-copies of $ \mathbf{F} $ in $ G $ rooted at $ S' $ and $ T' $.

We remark that, somewhat unusually, we are interested in copies which are potentially non-injective and partially labelled.
As mentioned, we use semi-copies of clique-paths to represent the way that discrepancies are redistributed by our process; since cliques are chosen one at a time, the resulting walk need not be injective.
We regard cliques in $ G $ as unlabelled, and thus want to consider unlabelled copies of clique-paths to avoid duplication.
However, we consider the sets $ S_i $ as ordered to ensure that concatenation and subpaths are well-defined; otherwise, there may be multiple non-isomorphic ways to glue together two clique-paths.
It is nonetheless useful to be able to ignore these orderings sometimes, and group together clique-paths which differ only in the orderings; this motivates the next definition.

\emph{Semi-isomorphisms:}
Say that two $ j $-clique-paths $ \mathbf{F} $ and $ \mathbf{F}' $ are \emph{semi-isomorphic} if there exists a hypergraph isomorphism $ \psi $ from $ \mathbf{F} $ to $ \mathbf{F}' $ which maps $ S_0 $ to $ S'_0 $ in the unique order-preserving way; unlike in an isomorphism, we allow changes to the ordering of $ S_{i} $ for each $ i \in [\ell] $.
This is clearly an equivalence relation on the set of $ j $-clique-paths, and we call each equivalence class $ \mathcal{F} $ a \emph{semi-ordered $ j $-clique-path}, noting that we may identify $ \mathcal{F} $ with the pair $ (F, \overline{\mathbf{S}}) $, where only $ \overline{S}_0 $ is equipped with an ordering and $ \overline{S}_i $ is considered unordered for $ i \in [\ell] $.
Given a $ j$-clique-path~$ \mathbf{F} $, write $ \mathcal{SO}(\mathbf{F}) $ for the semi-ordered $ j $-clique-path corresponding to the set of all $ j $-clique-paths which are semi-isomorphic to $ \mathbf{F} $.
We may equivalently think of $ \mathcal{SO}(\mathbf{F}) $ as simply taking $ \mathbf{F} $ and ignoring the orderings of all root sets but the first.

\emph{Weights:}
Given a semi-copy $ \Phi $ in $ G $ of a $ j $-clique-path $ \mathbf{F} = (F, \mathbf{S}) $ of length $ \ell $ with $ \mathbf{s} \coloneqq \mathbf{s}(F) $, define its \emph{weight} to be
\begin{equation} \label{eqn:clique_path_weight_def}
    w(\Phi) \coloneqq \prod_{i = 1}^{\ell} \frac{\hat{w}_{s_i}}{|K_r(\Phi(V(F_i \cap F_{i + 1})))|}, \quad \text{where } \hat{w}_{s} \coloneqq -\frac{1}{\binom{r - s}{j - s}} \text{ for each } s \in [j].
\end{equation}
Note that in the case $ \ell = 0 $ we treat the empty product $ w(\Phi) $ as being $ 1 $.
We may then define the \emph{weight of $ \mathbf{F} $ in $ G $} to be the function $ \Xi[\mathbf{F}]: E_j^2 \to \mathbb{R} $ given by
$$ \Xi[\mathbf{F}](S, T) \coloneqq \sum_{\Phi \in X_G(S, T, \mathbf{F})} w(\Phi), $$
and given a (multi-)set $ \mathcal{F} $ of $ j $-clique-paths, write $ \Xi[\mathcal{F}] \coloneqq \sum_{\mathbf{F} \in \mathcal{F}} \Xi[\mathbf{F}] $, summing with multiplicity.

The function $ \Xi[\mathcal{F}] $ assigns to each pair $ (S, T) $ of $ j $-sets (or edges in the case $ j = k $) a weighted sum over all (semi-)copies of any clique-path $ \mathbf{F} \in \mathcal{F} $ rooted at $ S $ and $ T $.
This will be used to represent the effect that the discrepancy $ \xi^{(j)}[\varphi](T) $ has on the discrepancy $ \xi^{(j)}[\mathcal{R}_s \varphi](S) $; see \cref{lem:process_graphs_char} and in particular \cref{fact:process_graphs_onestep} for details of this characterisation.
Given a semi-ordered $ j $-clique-path $ \mathcal{F} = (F, \overline{\mathbf{S}}) $ of length $ \ell \ge 0 $, define $ \iota(\mathcal{F}) \coloneqq \iota(\overline{S}_0 \cap \overline{S}_{\ell}, \overline{S}_0) $, recalling the definition of $ \iota $ from \cref{section:notation}, observing that $ \Xi[\mathcal{F}](S, T) = 0 $ for any $ S, T \in E_j $ which do not satisfy $ \iota(\mathcal{F}) \subseteq \iota(S \cap T, S) $, by the definition of a semi-copy.

\emph{Concentration:}
Our most significant pseudorandomness property will be the concentration of $ \Xi[\mathcal{F}] $ (around its expectation) for certain semi-ordered clique-paths $ \mathcal{F} $.
However, in a random hypergraph, it is too much to expect this for all clique-paths, so we make the following definition to characterise those clique-paths for which we can expect concentration.
Given $ b, c \in \mathbb{N}_0 $, say that a $ j $-clique-path $ (F, \mathbf{S}) $ of length $ \ell $ is \emph{$ (b, c) $-centred} if $ |S_0 \cap S_{\ell}| = b $ and, for all $ x \in S_0 \setminus S_{\ell} $ and $ y \in S_{\ell} \setminus S_0 $, every $ (S_0 \cap S_{\ell}) $-avoiding walk from $ x $ to $ y $ in $ F $ has length at least $ c $.
Observe that this property is preserved by semi-isomorphism, and we thus extend the definition naturally to semi-ordered clique-paths.
Intuitively, being centred ensures that $ S_0 \setminus S_{\ell} $ and $ S_{\ell} \setminus S_0 $ have sufficient distance in the hypergraph $ F $ to avoid trivial obstructions to the number of rooted copies of $ F $ being well-concentrated; for example, this avoids any $ x \in S_0 \setminus S_{\ell} $ and $ y \in S_{\ell} \setminus S_0 $ sharing an edge in the hypergraph $ F $.

In order to retrospectively motivate the preceding definitions, we now state the key lemma which uses weights of clique-paths to characterise our distributor process.
The correspondence is intuitively clear; we defer the (straightforward) proof to \cref{sect:proc_graph_char_proof}.

\begin{lemma} \label{lem:process_graphs_char}
    Let $ \ell \ge 2 $.
    For every $ j \in [k] $, there exists a multiset $ \mathcal{S}_{j, \ell} $ of size (with multiplicity) at most $ 2^{(2 \ell)^j} $, consisting of elements from $ [j - 1]^{\le (2 \ell)^j} $, such that the following holds for any $ i \in [0, \ell] $ and $ K_r $-function $ \varphi $ on $ G $.
    \begin{enumerate}[\rm (\roman*)]
        \item $ \xi^{(j)}[\mathcal{P}_{j - 1, \ell} \varphi] = \left\langle \sum_{\mathbf{s} \in \mathcal{S}_{j, \ell}} \Xi[\mathcal{F}^{(j)}_{\mathbf{s}}], \xi^{(j)}[\varphi] \right\rangle $. \label{cond:proc_graph_char_1}
    \end{enumerate}
    Writing $ \hat{\mathbf{s}} \coloneqq (j)^{\ell} $, for each $ \mathbf{s} = (\mathbf{s}_1, \ldots, \mathbf{s}_i) \in (\mathcal{S}_{j, \ell})^i $, define
    $$ \mathcal{G}_\ell^{(j)}(\mathbf{s}) \coloneqq \mathcal{F}^{(j)}_{\hat{\mathbf{s}}} \bullet \mathcal{F}^{(j)}_{\mathbf{s}_1} \bullet \mathcal{F}^{(j)}_{\hat{\mathbf{s}}} \bullet \mathcal{F}^{(j)}_{\mathbf{s}_2} \bullet \ldots \bullet \mathcal{F}^{(j)}_{\hat{\mathbf{s}}} \bullet \mathcal{F}^{(j)}_{\mathbf{s}_i} \quad \text{and} \quad \mathcal{G}_{\ell, i}^{(j)} \coloneqq \bigcup_{\mathbf{s} \in (\mathcal{S}_{j, \ell})^i} \mathcal{G}_\ell^{(j)}(\mathbf{s}), $$
    regarding $ \mathcal{G}_{\ell, i}^{(j)} $ as a multiset, with multiplicities according to those in $ (\mathcal{S}_{j, \ell})^i $.
    Then each $ j $-clique-path in $ \mathcal{G}_{\ell, i}^{(j)} $ has length at most $ \ell^{3k} $, and furthermore
    \begin{enumerate}[\rm (\roman*)] \setcounter{enumi}{1}
        \item $ \xi^{(j)}[(\mathcal{R}_j^{\ell} \mathcal{P}_{j - 1, \ell})^i \varphi] = \langle \Xi[\mathcal{G}_{\ell, i}^{(j)}], \xi^{(j)}[\varphi] \rangle $. \label{cond:proc_graph_char_2}
    \end{enumerate}
\end{lemma}

We may now proceed to state our pseudorandomness property.

\subsection{Regularity conditions} \label{sect:prop_reg}

Let $ n, \ell, c \in \mathbb{N} $, $ C, \eps > 0 $, and $ G $ be a $ k $-graph on $ n $ vertices.
For each $ j \in [0, k] $, define
\begin{equation} \label{eqn:gamma_d_defns}
    \gamma_j \coloneqq \gamma_j(G) \coloneqq \frac{|K_r(G)| \binom{r}{j}}{|E_j|} \quad \text{and} \quad d_j \coloneqq d_j(G) \coloneqq \frac{|E(G)| \binom{k}{j}}{|E_j|}.
\end{equation}
Note that $ \gamma_j $ and $ d_j $ represent the average number of cliques and edges, respectively, containing a given element of $ E_j $.
Say that $ G $ is \emph{$ (\ell, c, C, \eps) $-clique-regular} if, for all $ 0 \le b < j \le k $ and $ (b, c) $-centred semi-ordered $ j $-clique-path $ \mathcal{F} $ of length at most $ \ell^{3k} $, there exist $ \alpha_{\mathcal{F}} \in \mathbb{R} $ with $ |\alpha_{\mathcal{F}}| \le C $ and $ \zeta_{\mathcal{F}}: E_j^2 \to \mathbb{R} $ with $ \norm{\zeta_{\mathcal{F}}} \le \eps $ such that the following hold.

\begin{enumerate}[(R1)]
    \item $ |N_G^{\mathrm{e}}(S)| = (1 \pm \eps) d_{k - 1} $ for every $ S \in E_{k - 1} $; \label{cond:creg1}
    \item $ |K_r(e)| = (1 \pm \eps) \gamma_k $ for every $ e \in E(G) $; \label{cond:creg2}
    \item $ \Xi[\mathcal{F}](S, T) = \alpha_{\mathcal{F}} \frac{\gamma_j}{\gamma_b} + \zeta_{\mathcal{F}}(S, T) $ for all $ S, T \in E_j $ with $ \iota(\mathcal{F}) \subseteq \iota(S \cap T, S) $. \label{cond:creg3}
\end{enumerate}

This completes the statement of \cref{thm:proc_conv_fcd}.
In words, \ref{cond:creg1} says that the codegrees of $ G $ are well-concentrated, and \ref{cond:creg2} tells us that the number of cliques containing any edge is well-concentrated.
On a high level, \ref{cond:creg3} says that the (weighted) clique-path count from $ S $ to $ T $ is well-concentrated for any pair of $ j $-sets $ S, T $ (or edges if $ j = k $), provided that we consider only those clique-paths which do not have some degenerate structure trivially blocking their concentration.
We now proceed to prove \cref{thm:proc_conv_fcd}.

\section{Convergence in clique-regular hypergraphs} \label{sect:main_proof}

In this section, we prove \cref{thm:proc_conv_fcd}; our proof revolves around the following key lemma.

\begin{lemma} \label{lem:norm_decrease}
    Suppose $ 1/n \ll \eps \leq 1/C_2 \ll 1/C_1 \ll 1/\ell \ll 1/c, 1/r $.
    Let $ G $ be an $ (\ell, c, C_1, \eps) $-clique-regular $ k $-graph on $ n $ vertices and $ \varphi $ be a globally-balanced $ K_r $-function on $ G $.
    Then
    $$ \norm{\xi^{(j)}[\mathcal{P}_{j, \ell} \varphi]} \le C_2 \eps \norm{\xi^{(j)}[\varphi]} $$
    for every $ j \in [k] $.
\end{lemma}

To prove \cref{lem:norm_decrease}, we first use \cref{lem:process_graphs_char} to rewrite the desired discrepancy $ \xi^{(j)}[\mathcal{P}_{j, \ell} \varphi] $ in terms of a weighted sum over the discrepancies $ \xi^{(j)}[\varphi] $, according to the weights of possible clique-paths.
We then use our Ramsey-type statement (\cref{lem:ramsey_full}) to obtain, in each case, an interval over which the clique-path is $ (b, c) $-centred.
This allows us to rewrite the new discrepancy in terms of weighted sums of old discrepancies, in each of which the weights are concentrated around a particular value, using \ref{cond:creg3}.
By the globally-balanced assumption and inductive hypothesis, this implies a cancellation effect in the discrepancies, from which we deduce the desired decrease.

The rest of this section is structured as follows.
Firstly, in \cref{sect:proof_simple_obs}, we present some general observations, then in \cref{sect:proc_graph_char_proof}, we prove \cref{lem:process_graphs_char}, relating our process to weights of clique-paths.
We next prove our \cref{lem:ramsey_full} in \cref{sect:proof_concentrating}, which allows us to find centred subpaths in clique-paths, and combine this with \cref{lem:process_graphs_char} in \cref{sect:char_cent_decomp} to obtain \cref{lem:centered_decomp}, a more useful characterisation of the discrepancies in our process in terms of sets of clique-paths with the same centred subpath.
This gives us the tools we need to prove \cref{lem:norm_decrease} by induction on $ j $ in \cref{sect:proof_lem_decrease}.
Finally, in \cref{sect:proof_final}, we use \cref{lem:norm_decrease} to conclude the proof of \cref{thm:proc_conv_fcd}.

\subsection{Simple observations} \label{sect:proof_simple_obs}

We start this section by providing some naïve bounds and straightforward observations which will be useful in the main proof.
Firstly, it is easy to see that our process preserves the property of being globally-balanced.

\begin{fact} \label{fact:global_balance}
    Let $ \varphi $ be a globally-balanced $ K_r $-function on $ G $.
    Then $ \mathcal{R}_s \varphi $ is also globally-balanced for every $ s \in [k] $.
\end{fact}

\begin{proof}
    Recalling the definition of $ \mathcal{R}_s $ and rearranging the sums, we see that
    \begin{align*}
        \sum_{K \in K_r(G)} \mathcal{R}_s \varphi(K) &= \sum_{K \in K_r(G)} \left( \varphi(K) - \frac{1}{\binom{r - s}{k - s}} \sum_{S \in \binom{V(K)}{s}} \frac{\xi^{(s)}(S)}{|K_r(S)|}\right) \\
            &= \sum_{K \in K_r(G)} \varphi(K) - \frac{1}{\binom{r - s}{k - s}} \sum_{S \in E_s} \xi^{(s)}(S) \frac{|K_r(S)|}{|K_r(S)|}.
            % &= \binom{r}{k}^{-1} |E(G)| - 0,
    \end{align*}
    Noting that
    $$ \sum_{S \in E_s} \xi^{(s)}[\varphi](S) = \sum_{S \in E_s} \sum_{e \in N_G^{\mathrm{e}}(S)} \xi[\varphi](e) = \binom{k}{s} \sum_{e \in E(G)} \xi[\varphi](e) = 0 $$
    for any globally-balanced $ K_r $-function $ \varphi $, it follows that
    $$ \sum_{K \in K_r(G)} \mathcal{R}_s \varphi(K) = \sum_{K \in K_r(G)} \varphi(K) = \frac{|E(G)|}{\binom{r}{k}}, $$
    as required.
\end{proof}

Throughout our proof, we will make frequent use of the following commutativity statement for the operator $ \Xi $ with respect to $ \bullet $, the concatenation of clique-paths defined in \cref{section:clique_paths_def}, and $ \circ $, the matrix multiplication defined in \cref{section:notation}.

\begin{fact} \label{fact:xi_comp_commute}
    Let $ j \in [k] $ and $ \mathbf{F}^1, \mathbf{F}^2 $ be $ j $-clique-paths.
    Then
    $ \Xi[\mathbf{F}^1 \bullet \mathbf{F}^2] = \Xi[\mathbf{F}^1] \circ \Xi[\mathbf{F}^2]. $
\end{fact}

Given $ S, T \in E_j $, the function on the left hand side is a weighted sum over all copies of the concatenation $ \mathbf{F}^1 \bullet \mathbf{F}^2 $ rooted at $ S $ and $ T $, whereas the right hand side is a weighted sum over all pairs of copies of $ \mathbf{F}^1 $ rooted at $ S $ and $ U $ and $ \mathbf{F}^2 $ rooted at $ U $ and $ T $, for any $ U \in E_j $.
It is intuitively clear that there is a weight-preserving bijection between these two sets, since clique-paths fix the orderings of root sets, so there is a unique way to concatenate the two paths.

\begin{proof}[Proof of \cref{fact:xi_comp_commute}]
    Consider $ \mathbf{F}^1 = (F^1, \mathbf{S}^1) $ and $ \mathbf{F}^2 = (F^2, \mathbf{S}^2) $ of lengths $ \ell_1 $ and $ \ell_2 $, respectively, and write $ \mathbf{F}^1 \bullet \mathbf{F}^2 = (F', \mathbf{S}') $.
    Given $ S, T \in E_j $, let $ A \coloneqq X_G(S, T, \mathbf{F}^1 \bullet \mathbf{F}^2) $ and $ B \coloneqq \bigcup_{U \in E_j} X_G(S, U, \mathbf{F}^1) \times X_G(U, T, \mathbf{F}^2)$; by definition, it suffices to show that there exists a weight-preserving bijection $ \Psi: A \to B $, that is, such that $ w(\Phi) = w(\Phi_1) w(\Phi_2) $ whenever $ \Psi(\Phi) = (\Phi_1, \Phi_2) $.
    To do this, recall that by definition $ V(F') = V(F^1) \cup V(F^2) $, where the two vertex sets are regarded as intersecting exactly in the (ordered) set $ S^1_{\ell_1} = S^2_0 $.
    Given $ \Phi \in A $, define $ \Phi_1 \coloneqq \Phi \restriction_{V(F^1)} $ and $ \Phi_2 \coloneqq \Phi \restriction_{V(F^2)} $ and take $ U \coloneqq \Phi(S^1_{\ell_1}) $.
    It is thus clear that $ \Phi_1 $ and $ \Phi_2 $ are homomorphisms from $ F^1 $ and $ F^2 $ to $ G $, respectively, preserving the orderings on the root sets.
    As such, $ \Psi(\Phi) \coloneqq (\Phi_1, \Phi_2) \in B $.
    If $ \psi $ is an automorphism of $ \mathbf{F}^1 \bullet \mathbf{F}^2 $ then clearly $ \psi \restriction_{V(F^1)} $ and $ \psi \restriction_{V(F^2)} $ are hypergraph automorphisms of $ F^1 $ and $ F^2 $, respectively, preserving the root sets and their orderings; in particular, the function $ \Psi $ is well-defined.

    Suppose $ \Psi(\Phi) = (\Phi_1, \Phi_2) = (\Phi'_1, \Phi'_2) = \Psi(\Phi') $, then $ \Phi_1, \Phi'_1$ and $ \Phi_2, \Phi'_2 $ differ at most by permutations of vertices not belonging to any root set, so $ \Phi $ and $ \Phi' $ also differ at most by some automorphism of $ \mathbf{F}^1 \bullet \mathbf{F}^2 $; in particular, $ \Psi $ is injective.
    Furthermore, given any $ (\Phi_1, \Phi_2) \in B $, we may define a semi-copy $ \Phi \in A $ of $ \mathbf{F}^1 \bullet \mathbf{F}^2 $ by taking $ \Phi(v) \coloneqq \Phi_q(v) $ for each $ q \in [2] $ and $ v \in V(F^q) $.
    Note that this is well-defined, as $ \Phi_1 $ and $ \Phi_2 $ both map the intersection $ S^2_0 $ to some $ U \in E_j $ in the unique order-preserving way, and $ \Psi(\Phi) = (\Phi_1, \Phi_2) $; in particular, $ \Psi $ is surjective.
    It is easy to see that $ w(\Phi) = w(\Phi_1) w(\Phi_2) $, so $ \Psi $ is the required weight-preserving bijection.
\end{proof}

Our next observation gives general estimates for converting between statements about edge and clique counts and discrepancy functions on sets of different sizes.
Given $ j \in [k] $ and a set $ S \subseteq V(G) $ with $ |S| \le j $ write $ E_j(S) \coloneqq \{ T \in E_j: S \subseteq T \} $.
Recall the definitions of $ d_j, \gamma_j $ in \eqref{eqn:gamma_d_defns}.

\begin{fact} \label{fact:uniformity_disc}
    Let $ 0 \le i \le j \le k $ and $ \varphi $ be a $ K_r $-function on $ G $.
    Then
    \begin{enumerate}[\rm (\roman*)]
        \item $ \xi^{(i)}(S) = \binom{k - i}{j - i}^{-1} \sum_{T \in E_j(S)} \xi^{(j)}(T) $ for every $ S \in E_i $. \label{cond:unif_disc_3}
    \end{enumerate}
    If additionally $ G $ satisfies {\rm \ref{cond:creg1}} and {\rm \ref{cond:creg2}} for some $ \eps \in (0, 1) $, then also
    \begin{enumerate}[\rm (\roman*)] \setcounter{enumi}{1}
        \item $ |E_j(S)| = (1 \pm \eps) \binom{j}{i} \frac{|E_j|}{|E_i|} $ for every $ S \in E_i $, so in particular $ |N_G^{\mathrm{e}}(S)| = (1 \pm \eps) d_i $; \label{cond:unif_disc_1}
        \item $ |K_r(S)| = (1 \pm 3 \eps) \gamma_j $ for every $ S \in E_j $ with $ j \in [k] $; \label{cond:unif_disc_2}
        \item $ \norm{\xi^{(i)}[\varphi]} \le (1 + \eps) \frac{d_i}{d_j} \norm{\xi^{(j)}[\varphi]} $. \label{cond:unif_disc_4}
    \end{enumerate}
\end{fact}

\begin{proof}
    For \ref{cond:unif_disc_3}, given $ S \in E_i $, we compute
    $$ \xi^{(i)}(S) = \sum_{e \in N_G^{\mathrm{e}}(S)} \xi(e) = \sum_{e \in N_G^{\mathrm{e}}(S)} \xi(e) \sum_{T \in \binom{e}{j}: S \subseteq T} \binom{k - i}{j - i}^{-1} = \binom{k - i}{j - i}^{-1} \sum_{T \in E_j(S)} \xi^{(j)}(T). $$ % = \binom{k - i}{j - i}^{-1} \sum_{T \in E_j(S)} \sum_{e \in N_G^{\mathrm{e}}(T)} \xi(e) 
    For \ref{cond:unif_disc_1}, given $ S \in E_i $, observe that, if $ j \in [k - 1] $, then
    \begin{equation} \label{eqn:unif_disc_1}
        |E_j(S)| = \binom{n - i}{j - i} = \binom{j}{i} \frac{\binom{n}{j}}{\binom{n}{i}} = \binom{j}{i} \frac{|E_j|}{|E_i|}.
    \end{equation}
    The case $ i = j = k $ is trivial so assume instead $ i < j = k $, then by \ref{cond:creg1}, we have
    \begin{align*}
        |E_k(S)| = |N_G^{\mathrm{e}}(S)| &= \binom{k - i}{k - 1 - i}^{-1} \sum_{T \in E_{k - 1}(S)} |N_G^{\mathrm{e}}(T)|\\
            &= \binom{k - i}{k - 1 - i}^{-1} \binom{k - 1}{i} \frac{|E_{k - 1}|}{|E_i|} \cdot (1 \pm \eps) d_{k - 1}\\
            &= (1 \pm \eps) \frac{\binom{k}{k - 1} \binom{k - 1}{i}}{\binom{k - i}{k - 1 - i}} \frac{|E(G)|}{|E_i|}\\
            &= (1 \pm \eps) \binom{k}{i} \frac{|E_k|}{|E_i|}\\
            &= (1 \pm \eps) d_i.
    \end{align*}
    For \ref{cond:unif_disc_2}, given $ S \in E_j $, using \ref{cond:unif_disc_1} and \ref{cond:creg2} in the second equality, we have
    \begin{align*}
        |K_r(S)| = \binom{r - j}{k - j}^{-1} \sum_{T \in E_{k}(S)} |K_r(T)| &= (1 \pm \eps)^2 \binom{r - j}{k - j}^{-1} \binom{k}{j} \frac{|E_{k}|}{|E_j|} \gamma_{k}\\
            &= (1 \pm 3 \eps) \frac{\binom{k}{j} \binom{r}{k}}{\binom{r - j}{k - j}} \frac{|K_r(G)|}{|E_j|}\\
            &= (1 \pm 3 \eps) \binom{r}{j} \frac{|K_r(G)|}{|E_j|}\\
            &= (1 \pm 3 \eps) \gamma_j.
    \end{align*}
    Using \ref{cond:unif_disc_3} and \ref{cond:unif_disc_1}, we deduce that
    $$ \norm{\xi^{(i)}[\varphi]} \le \binom{k - i}{j - i}^{-1} (1 + \eps) \binom{j}{i} \frac{|E_j|}{|E_i|} \norm{\xi^{(j)}[\varphi]} = (1 + \eps)  \frac{d_i}{d_j} \norm{\xi^{(j)}[\varphi]}, $$
    as required for \ref{cond:unif_disc_4}.
\end{proof}

Next we present a collection of naïve bounds on various relevant quantities, such as the magnitude of changes to a $ K_r $-function and its discrepancy under the distributor process.
We remark that \cref{fact:norm_delta_naive}~\ref{cond:norm_naive_5} is essentially a weaker version of \cref{lem:norm_decrease}, but the former will be required for the proof of the latter.

\begin{fact} \label{fact:norm_delta_naive}
    Suppose $ 1/n \ll 1/C \ll 1/\ell, 1/r $.
    Let $ G $ satisfy {\rm \ref{cond:creg1}} and {\rm \ref{cond:creg2}} for some $ \eps \in (0, \frac{1}{6}) $.
    Then, for any $ j \in [k] $, $ i \in [0, \ell] $, $ \mathbf{s} \in (\mathcal{S}_{j, \ell})^i $, $ m \ge 0 $, $ j $-clique-path $ \mathbf{F} $ of length $ m $, and $ K_r $-function $ \varphi $ on $ G $, the following hold.
    
    \begin{enumerate}[\rm (\roman*)]
        \item $ |\mathcal{G}_{\ell}^{(j)}(\mathbf{s})| \le C $; \label{cond:norm_naive_1}
        \item $ \norm{\Xi[\mathbf{F}]} \le C^m $; \label{cond:norm_naive_3}
        % \item $ \norm{\xi^{(j)}[\mathcal{P}_{j - 1, \ell} \varphi]} \le C \norm{\xi^{(j)}[\varphi]} $; \label{cond:norm_naive_4}
        \item $ \norm{\xi^{(j)}[(\mathcal{R}_{j}^{\ell} \mathcal{P}_{j - 1, \ell})^i \varphi]} \le C \norm{\xi^{(j)}[\varphi]} $; \label{cond:norm_naive_5}
        \item $ \norm{(\mathcal{R}_{j}^{\ell} \mathcal{P}_{j - 1, \ell})^i \varphi - \varphi} \le \frac{C}{\gamma_j} \norm{\xi^{(j)}[\varphi]} $. \label{cond:norm_naive_6}
    \end{enumerate}
\end{fact}

\begin{proof}
    Introduce a new constant $ C' > 0 $ satisfying $ 1/C \ll 1/C' \ll 1/\ell, 1/r $.
    
    For \ref{cond:norm_naive_1}, recall the definition of $ \mathcal{F}_s^{(j)} $ from \cref{section:clique_paths_def} and note for any $ s \in [j] $ that clearly
    \begin{equation} \label{eqn:norm_naive_-1}
        |\mathcal{F}_s^{(j)}| \le C',
    \end{equation}
    since this is a bound on the number of $ j $-clique-paths of length 1.
    Recalling \cref{lem:process_graphs_char}, since $ |\mathbf{s}_q| \le (2 \ell)^j \le C' $ for each $ q \in [i] $, we see by definition that $ |\mathcal{G}_{\ell}^{(j)}(\mathbf{s})| \le (C')^{i (\ell + C')} \le C $.
    
    For \ref{cond:norm_naive_3}, recall from \cref{section:clique_paths_def} that any $j$-clique-path $ \mathbf{F} $ of length $ m \ge 0 $ may be written as a concatenation $ \mathbf{F} = \mathbf{F}^1 \bullet \ldots \bullet \mathbf{F}^m $ of $j$-clique-paths $ \mathbf{F}^i $ each of length 1.
    By \cref{fact:xi_comp_commute}, this means that $ \norm{\Xi[\mathbf{F}]} = \norm{\Xi[\mathbf{F}^1] \circ \cdots \circ \Xi[\mathbf{F}^m]} \le \norm{\Xi[\mathbf{F}^1]} \cdots \norm{\Xi[\mathbf{F}^m]} $, so it suffices to show that any clique-path $ \mathbf{F}' $ of length 1 has $ \norm{\Xi[\mathbf{F}']} \le C $.
    Indeed, recall that
    $$ \sum_{T \in E_j} |\Xi[\mathbf{F}'](S, T)| \le \sum_{T \in E_j} \sum_{\Phi \in X_G(S, T, \mathbf{F}')} |w(\Phi)|. $$
    By \cref{fact:uniformity_disc}~\ref{cond:unif_disc_2}, letting $ s \coloneqq s_1(\mathbf{F}') $, we see that $ |w(\Phi)| \le \frac{2}{\gamma_s} $ for any $ T \in E_j $ and $ \Phi \in X_G(S, T, \mathbf{F}') $.
    We will now show that $ \sum_{T \in E_j} |X_G(S, T, \mathbf{F}')| \le (C')^3 \gamma_s $, which clearly suffices.
    Indeed, this sum represents the total number of semi-copies of $ \mathbf{F}' $ rooted at $ S $ and any $ T \in E_j $, each of which consists of a copy of $ K_r $ containing $ S $, as well as $ j - s $ vertices, which form an edge with some $ s $ vertices of the clique in the case $ j = k $. 
    By \cref{fact:uniformity_disc}~\ref{cond:unif_disc_2}, the total number of cliques containing $ S $ is at most $ 2 \gamma_j $.
    There are at most $ C' $ choices for a set $ U $ of $ s $ vertices in the clique.
    Then by \cref{fact:uniformity_disc}~\ref{cond:unif_disc_1}, there are at most $ C' \frac{d_s}{d_j} $ choices for $ T $.
    Hence in total the sum is at most $ 2 (C')^2 \frac{d_s \gamma_j}{d_j} \le (C')^3 \gamma_s $, meaning that
    \begin{equation} \label{eqn:norm_naive_0}
        \norm{\Xi[\mathbf{F}]} \le (C')^{4m},
    \end{equation}
    which suffices for \ref{cond:norm_naive_3}.

    By \cref{lem:process_graphs_char}~\ref{cond:proc_graph_char_2}, we have $ \xi^{(j)}[(\mathcal{R}_{j}^{\ell} \mathcal{P}_{j - 1, \ell})^i \varphi] = \langle \Xi[\mathcal{G}_{\ell, i}^{(j)}], \xi^{(j)}[\varphi] \rangle $.
    For each $ \mathbf{s} \in \mathcal{S}_{j, \ell} $, recall that $ \mathcal{F}_{\mathbf{s}}^{(j)} = \mathcal{F}_{s_1}^{(j)} \bullet \ldots \bullet \mathcal{F}_{s_t}^{(j)} $ for $ t \coloneqq |\mathbf{s}| \le (2 \ell)^j $, so by~\eqref{eqn:norm_naive_-1} we have $ |\mathcal{F}_{\mathbf{s}}^{(j)}| \le (C')^{(2 \ell)^j} $, and similarly note that $ |\mathcal{F}_{\hat{\mathbf{s}}}| \le (C')^{\ell} $.
    As such, we see that $ |\mathcal{G}_{\ell, i}^{(j)}| \le 2^{(2\ell)^j i} \cdot (C')^{(2 \ell)^j i} \cdot (C')^{\ell i} \le 2^{C'} $.
    Also, by \cref{lem:process_graphs_char}, every $j $-clique-path in $ \mathcal{G}_{\ell, i}^{(j)} $ has length at most $ \ell^{3k} $.
    As such, we conclude by \eqref{eqn:norm_naive_0} that $ \norm{\Xi[\mathcal{G}_{\ell, i}^{(j)}]} \le 2^{C'} \cdot (C')^{4 \ell^{3k}} \le C $, from which \ref{cond:norm_naive_5} follows.

    For \ref{cond:norm_naive_6}, note that
    \begin{equation} \label{eqn:norm_naive_1}
        \norm{\mathcal{R}_s \varphi - \varphi} \le \frac{1}{\binom{r - s}{k -s}}\binom{r}{s} \frac{2 \norm{\xi^{(s)}[\varphi]}}{\gamma_s} \le \frac{C'}{\gamma_j} \norm{\xi^{(j)}[\varphi]}
    \end{equation}
    for any $ s \in [j] $, where the first inequality follows from the definition of $ \mathcal{R}_s $ using \cref{fact:uniformity_disc}~\ref{cond:unif_disc_2}, and the second inequality uses \cref{fact:uniformity_disc}~\ref{cond:unif_disc_4}.
    Observe, by the definition of $ \xi $, that $ \xi[\psi] - \xi[\varphi] = \psi - \varphi $ as functions $ E(G) \to \mathbb{R} $ for any $ K_r $-functions $ \varphi, \psi $, which means in particular that $ \xi^{(j)}[\psi] =\xi^{(j)}[\varphi] + \psi - \varphi $ as functions $ E_j \to \mathbb{R} $.
    Since $ |N_G^{\mathrm{e}}(S)| \le 2 d_j $ for every $ S \in E_j $ by \cref{fact:uniformity_disc}~\ref{cond:unif_disc_1} and $ |K_r(e)| \le 2 \gamma_k $ for every $ e \in E(G) $ by \ref{cond:creg2}, it follows that
    \begin{align*}
        \norm{\xi^{(j)}[\mathcal{R}_s \varphi]} &\le \norm{\xi^{(j)}[\varphi]} + 2 d_j \cdot 2 \gamma_k \cdot \norm{\mathcal{R}_s \varphi - \varphi}\\
            &\le \norm{\xi^{(j)}[\varphi]} + 2 d_j \cdot 2 \gamma_k \cdot \frac{C'}{\gamma_j} \norm{\xi^{(j)}[\varphi]}
            \le (C')^2 \norm{\xi^{(j)}[\varphi]}.
    \end{align*}
    Given $ m \ge 0 $ and a sequence $ \mathbf{s} \in [j]^m $, writing $ \mathcal{R}_{\mathbf{s}} \coloneqq \mathcal{R}_{s_m} \cdots \mathcal{R}_{s_1} $, it follows by induction that
    \begin{equation} \label{eqn:norm_naive_2}
        \norm{\xi^{(j)}[\mathcal{R}_{\mathbf{s}} \varphi]} \le (C')^{2m} \norm{\xi^{(j)}[\varphi]}.
    \end{equation}
    By definition, it is clear that we may write $ \mathcal{P}_{j - 1, \ell} = \mathcal{R}_{\mathbf{s}'} $ for some $ \mathbf{s}' \in [j]^{m'} $, where $ m' \le C' $ by the fact that $ 1/C' \ll 1/\ell, 1/r $.
    As such, we see that $ (\mathcal{R}_{j}^{\ell} \mathcal{P}_{j - 1, \ell})^i = \mathcal{R}_{\mathbf{s}} $ for some $ \mathbf{s} \in [j]^m $, where $ m \le i (C' + \ell) \le 2 C' \ell $.
    Using \eqref{eqn:norm_naive_1} and \eqref{eqn:norm_naive_2} in the second inequality, it follows that
    \begin{align*}
        \norm{(\mathcal{R}_{j}^{\ell} \mathcal{P}_{j - 1, \ell})^i \varphi - \varphi} &\le \sum_{q \in [m]} \norm{\mathcal{R}_{\mathbf{s}_{[q]}} \varphi - \mathcal{R}_{\mathbf{s}_{[q - 1]}} \varphi} \\
        &\le 2 C' \ell \cdot \frac{C'}{\gamma_j} \cdot (C')^{2 m} \norm{\xi^{(j)}[\varphi]}
        \le \frac{C}{\gamma_j} \norm{\xi^{(j)}[\varphi]},
    \end{align*}
    as required.
\end{proof}

Our next observation is somewhat more technical to formulate, but the statement is intuitive: if $ \chi: E_j^2 \to \mathbb{R} $ can be defined by $ \chi(S, T) \coloneqq \hat{\chi}(S) \mathds{1}[J \subseteq \iota(S \cap T, S)] $ for some function $ \hat{\chi}: E_j \to \mathbb{R} $ and set $ J \in \binom{[j]}{b} $, then the maximum size of $ \langle \chi, \xi^{(j)} \rangle $ can be bounded in terms of the discrepancy $ \xi^{(b)} $ on $ b $-sets.

\begin{fact} \label{fact:func_intersect_bound}
    Let $ 0 \le b < j \le k $ and $ \varphi $ be a $ K_r $-function on $ G $.
    Let $ \chi: E_j^2 \to \mathbb{R} $ be defined by $ \chi(S, T) \coloneqq \mathds{1}[J \subseteq \iota(S \cap T, S)] \hat{\chi} $ for some set $ J \in \binom{[j]}{b} $ and $ \hat{\chi} \in \mathbb{R} $.
    Then
    $$ \norm{\langle \chi, \xi^{(j)}[\varphi] \rangle} \le 2^k |\hat{\chi}| \norm{\xi^{(b)}[\varphi]}. $$
\end{fact}

\begin{proof}
    Given $ S \in E_j$, observe that there exists a unique set $ U \subseteq S $ (of size $ b $) with $ \iota(U, S) = J $.
    Hence, using \cref{fact:uniformity_disc}~\ref{cond:unif_disc_3}, we may write
    $$ \langle \chi, \xi^{(j)}[\varphi] \rangle(S)
        = \sum_{T \in E_j} \xi^{(j)}[\varphi](T) \chi(S, T)
        = \sum_{T \in E_j(U)} \xi^{(j)}[\varphi](T) \hat{\chi}
        = \binom{k - b}{j - b} \xi^{(b)}[\varphi](U) \hat{\chi}, $$
    from which the desired bound follows immediately.
\end{proof}

In order to prove non-negativity of our process, it will be helpful to observe that clique-regularity ensures that the (unique) globally-balanced uniform $ K_r $-function already has small initial discrepancies, since the number of cliques containing each edge is well-concentrated.

\begin{fact} \label{fact:uniform_initial_small}
    Let $ G $ satisfy {\rm \ref{cond:creg2}} for some $ \eps > 0 $.
    Let $ \varphi $ be the globally-balanced uniform $ K_r $-function on $ G $.
    Then $ \norm{\xi[\varphi]} \le \eps $.
\end{fact}

\begin{proof}
    Recall firstly that $ \varphi(K) = \frac{1}{\gamma_k} = \frac{|E(G)|}{|K_r(G)| \binom{r}{k}} $ for every $ K \in K_r(G) $.
    Given $ e \in E_k $, using {\rm \ref{cond:creg2}}, we obtain
    \begin{align*}
        \xi[\varphi](e) = \left( \sum_{K \in K_r(e)} \varphi(K) \right) - 1
            = \frac{|K_r(e)|}{\gamma_k} - 1
            = \frac{(1 \pm \eps) \gamma_k}{\gamma_k} - 1
            = \pm \eps,
    \end{align*}
    as required.
\end{proof}

\subsection{Clique-path characterisation} \label{sect:proc_graph_char_proof}

We now prove \cref{lem:process_graphs_char}, which relies on the following observation, allowing us to relate the effect of the distributor process on the discrepancies of a $ K_r $-function to the weight of suitable clique-paths in $ G $.
Throughout this section, let $ \varphi $ be a $ K_r $-function on $ G $.

\begin{fact} \label{fact:process_graphs_onestep}
    Let $ 1 \le s < j \le k $.
    Then
    $ \xi^{(j)}[\mathcal{R}_s \varphi] = \langle \mathbf{1} + \Xi[\mathcal{F}^{(j)}_s], \xi^{(j)}[\varphi] \rangle $
    and
    $ \xi^{(j)}[\mathcal{R}_j \varphi] = \langle \Xi[\mathcal{F}^{(j)}_j], \xi^{(j)}[\varphi] \rangle $.
\end{fact}

\begin{proof}
    For $ S, T \in E_j $, note that $ \bigcup_{\mathbf{F} \in \mathcal{F}^{(j)}_s} X_G(S, T, \mathbf{F}) $ may be regarded as the multiset consisting of all copies $ F' $ of $ K_r $ in $ G $ whose vertex set contains $ S $ and some set $ S \ne U \subseteq T $ of size $ s $.
    The multiplicity of such an $ F' $ is the number of different choices of $ U $, noting that, since we also consider non-injective homomorphisms, it may be the case that $ |V(F') \cap T| > s $.
    As such, we see that
    \begin{equation} \label{eqn:proc_graph_one_eq1}
        \Xi[\mathcal{F}^{(j)}_s](S, T) = - \frac{1}{\binom{r - s}{j - s}} \sum_{U \in \binom{T}{s}} \frac{|K_r(S \cup U)|}{|K_r(U)|} + \mathds{1}[s = j, S = T],
    \end{equation}
    since in the case that $ s = j $ and $ S = T $, the summand $ U = S $ contributes weight exactly $ -1 $ to the sum, and thus cancels with the indicator function.
    On the other hand, for any $ S \in E_j $, we compute
    \begin{align*}
        \xi^{(j)} [\mathcal{R}_s \varphi](S) %&= \sum_{e \in N_G^{\mathrm{e}}(S)} \sum_{K \in K_r(e)} (\mathcal{R}_s \varphi(K) - 1)\\
            &= \sum_{e \in N_G^{\mathrm{e}}(S)} \left( \sum_{K \in K_r(e)} \left( \varphi(K) - \frac{1}{\binom{r - s}{k - s}} \sum_{U \in \binom{V(K)}{s}} \frac{\xi^{(s)}[\varphi](U)}{|K_r(U)|} \right) - 1 \right)\\
            &= \xi^{(j)}[\varphi](S) - \frac{1}{\binom{r - s}{k - s}} \sum_{e \in N_G^{\mathrm{e}}(S)} \sum_{K \in K_r(e)} \sum_{U \in \binom{V(K)}{s}} \frac{\xi^{(s)}[\varphi](U)}{|K_r(U)|}\\
            &= \xi^{(j)}[\varphi](S) - \frac{1}{\binom{r - s}{k - s}} \sum_{K \in K_r(S)} \binom{r - j}{k - j} \sum_{U \in \binom{V(K)}{s}} \frac{\xi^{(s)}[\varphi](U)}{|K_r(U)|},
    \end{align*}
    where the last equality follows from the fact that every $ K \in K_r(S) $ corresponds to exactly $ \binom{r - j}{k - j} $ pairs $ (e, K) $ with $ e \in N_G^{\mathrm{e}}(S) $ and $ K \in K_r(e) $.
    Now by \cref{fact:uniformity_disc}~\ref{cond:unif_disc_3}, we may write $ \xi^{(s)}[\varphi](U) = \binom{k - s}{j - s}^{-1} \sum_{T \in E_j(U)} \xi^{(j)}[\varphi](T) $ for any $ U \in E_s $.
    In order to rearrange the sums, given a triple $ (K, U, T) $, the conditions $ K \in K_r(S) $, $ U \in \binom{V(K)}{s} $ and $ T \in E_j(U) $ are exactly equivalent to the conditions $ T \in E_j $, $ U \in \binom{T}{s} $, and $ K \in K_r(S \cup U) $.
    As such, writing $ \mathbf{1}'(S, T) \coloneqq \mathds{1}[s \ne j, S = T] $, we may rearrange to obtain
    \begin{align*}
            \xi^{(j)} [\mathcal{R}_s \varphi](S) &= \xi^{(j)}[\varphi](S) - \frac{\binom{r - j}{k - j}}{\binom{r - s}{k - s}\binom{k - s}{j - s}} \sum_{T \in E_j} \xi^{(j)}[\varphi](T) \sum_{U \in \binom{T}{s}} \frac{|K_r(S \cup U)|}{|K_r(U)|}\\
            &= \langle \mathbf{1}' + \Xi[\mathcal{F}^{(j)}_s], \xi^{(j)}[\varphi] \rangle(S),
    \end{align*}
    where the last equality follows from \eqref{eqn:proc_graph_one_eq1}, using the identity $ \frac{\binom{r - j}{k - j}}{\binom{r - s}{k - s}\binom{k - s}{j - s}} = \binom{r - s}{j -s}^{-1} $.
\end{proof}

The proof of \cref{lem:process_graphs_char} is now fairly straightforward.

\begin{proof}[Proof of \cref{lem:process_graphs_char}]
    For \ref{cond:proc_graph_char_1}, note by the inductive definition that the operator $ \mathcal{P}_{j - 1, \ell} $ can be written as a composition $ \mathcal{R}_{s_1} \cdots \mathcal{R}_{s_m} $ for some $ m \coloneqq m_j \ge 0 $ and $ \mathbf{s} \in [j - 1]^{m} $, noting that we take $ m_1 \coloneqq 0 $ and consider the empty composition to be the identity operator.
    Since $ m_j = \ell (\ell + m_{j - 1}) $, it is easy to check inductively that $ m_j \le (2 \ell)^{j} $ for $ j \in [2, k] $.
    It follows inductively from \cref{fact:process_graphs_onestep}, using associativity, that
    $$ \xi^{(j)}[\mathcal{R}_{s_1} \cdots \mathcal{R}_{s_m} \varphi] = \langle (\mathbf{1} + \Xi[\mathcal{F}^{(j)}_{s_1}]) \circ \cdots \circ (\mathbf{1} + \Xi[\mathcal{F}^{(j)}_{s_m}]), \xi^{(j)}[\varphi] \rangle, $$
    again regarding the empty composition as the identity operator in the case $ j = 1 $.
    Thus, expanding the product above into a sum with $ 2^m $ elements, since the identity matrix $ \mathbf{1} $ can be ignored in any product, it is clear that there exists a multiset $ \mathcal{S}_{j, \ell} $ of size (with multiplicity) exactly $ 2^m $ such that
    $$ \xi^{(j)}[\mathcal{P}_{j - 1, \ell} \varphi] = \left\langle \sum_{\mathbf{s} \in \mathcal{S}_{j, \ell}} \Xi[\mathcal{F}^{(j)}_{s_1}] \circ \cdots \circ \Xi[\mathcal{F}^{(j)}_{s_{|\mathbf{s}|}}], \xi^{(j)}[\varphi] \right\rangle, $$
    from which \ref{cond:proc_graph_char_1} follows, using \cref{fact:xi_comp_commute}.\COMMENT{Note that one term in the product is just the identity matrix, and this corresponds to the empty sequence $ \mathbf{s} $ and thus the empty $j$-clique-path, both of which are defined.}

    It is clear from the definition that for all $ \mathbf{s} \in (\mathcal{S}_{j, \ell})^i $ and $ \mathbf{F} \in \mathcal{G}_{\ell}^{(j)}(\mathbf{s}) $, the length of $ \mathbf{F} $ is at most $ i (\ell + (2 \ell)^j) \le (2 \ell)^{k + 1} \le \ell^{3k} $, using here that $ \ell, k \ge 2 $.
    By \cref{fact:process_graphs_onestep}, associativity, and \cref{fact:xi_comp_commute}, we see that
    \begin{equation} \label{eqn:proc_graph_char_1}
        \xi^{(j)}[\mathcal{R}_j^{\ell} \varphi] = \left\langle \Xi[\mathcal{F}^{(j)}_j] \circ \cdots \circ \Xi[\mathcal{F}^{(j)}_j], \xi^{(j)}[\varphi] \right\rangle = \left\langle \Xi[\mathcal{F}^{(j)}_{\hat{\mathbf{s}}}], \xi^{(j)}[\varphi] \right\rangle
    \end{equation}
    for any $ K_r $-function $ \varphi$.
    Using \ref{cond:proc_graph_char_1} and associativity, this implies that
    $$ \xi^{(j)}[\mathcal{R}_j^{\ell} \mathcal{P}_{j - 1, \ell} \varphi] = \left\langle \Xi[\mathcal{F}^{(j)}_{\hat{\mathbf{s}}}] \circ \sum_{\mathbf{s} \in \mathcal{S}_{j, \ell}} \Xi[\mathcal{F}^{(j)}_{\mathbf{s}}], \xi^{(j)}[\varphi] \right\rangle, $$
    and thus, by iterating, that
    $$ \xi^{(j)}[(\mathcal{R}_j^{\ell} \mathcal{P}_{j - 1, \ell})^i \varphi] = \left\langle \sum_{\mathbf{s}_1 \in \mathcal{S}_{j, \ell}} \cdots \sum_{\mathbf{s}_i \in \mathcal{S}_{j, \ell}} \Xi[\mathcal{F}^{(j)}_{\hat{\mathbf{s}}}] \circ \Xi[\mathcal{F}^{(j)}_{\mathbf{s}_1}] \circ \cdots \circ \Xi[\mathcal{F}^{(j)}_{\hat{\mathbf{s}}}] \circ \Xi[\mathcal{F}^{(j)}_{\mathbf{s}_i}], \xi^{(j)}[\varphi] \right\rangle, $$
    from which \ref{cond:proc_graph_char_2} follows immediately by \cref{fact:xi_comp_commute}.
\end{proof}

\subsection{Centred subpaths} \label{sect:proof_concentrating}

In this section we prove that all of our clique-paths contain some centred subpath.
Given $ \ell, i \in \mathbb{N} $, $ j \in [k] $, and $ \mathbf{s} \in (\mathcal{S}_{j, \ell})^i $, write $ M(\mathbf{s}, \ell) \coloneqq i\ell + \sum_{q \in [i]} |\mathbf{s}_q| $.

\begin{lemma} \label{lem:ramsey_full}
    Suppose $ 1/i \ll 1/c, 1/k $, and let $ j \in [k] $ and $ \ell \in \mathbb{N} $.
    Given $ \mathbf{s} \in (\mathcal{S}_{j, \ell})^i $ and $ \mathbf{F} = (F, \mathbf{S}) \in \mathcal{G}^{(j)}_{\ell}(\mathbf{s}) $, there exist $ i_1 \coloneqq i_1(\mathbf{F}) \in [0, i - 1], i_2 \coloneqq i_2(\mathbf{F}) \in [i_1 + 1, i] $, and $ b \coloneqq b(\mathbf{F}) \in [0, j - 1] $ such that the $ j $-clique-path $ \mathbf{F}_{[M(\mathbf{s}_{[i_1]}, \ell), m_2]} $ is $ (b, c) $-centred, writing $ m_2 \coloneqq M(\mathbf{s}_{[i_2 - 1]}, \ell) + \ell $.
    Furthermore, we can choose $ i_1, i_2, b$ in such a way that they are fully determined by $ \mathbf{F}_{[0, m_2]} $ (that is, independent of the rest of $ \mathbf{F} $) and invariant under semi-isomorphisms of $ \mathbf{F}_{[0, m_2]} $.
\end{lemma}

This is a Ramsey-type statement: given a sequence of sets, we may find a consecutive subsequence in which every element is either contained in the intersection of the whole subsequence, or not contained in the intersection of any sufficiently long subsequence of the subsequence; to formalise this idea, we make a further definition.
Given $ j \in [0, k] $ and $ m \in \mathbb{N} $, define a \emph{non-repeating $ j $-sequence of length $ m $} to be a sequence $ \mathbf{T} = (T_i)_{i \in [0, m]} $ of sets of size at most $ j $, such that the \emph{vertex index set $ I_v \coloneqq \{ q \in [0, m]: v \in T_q \} $ of $v$ in $ \mathbf{T}$} is an interval for every $ v \in V(\mathbf{T}) \coloneqq \bigcup_{q = 0}^{m} T_q $.
Observe that, for any $j$-clique-path $ \mathbf{F} = (F, \mathbf{S}) $ of length $ m $, the sequence $ \mathbf{S} $ is a non-repeating $ j $-sequence of length $ m $, by definition.
Note also that any subsequence of a non-repeating $ j $-sequence is also a non-repeating $ j $-sequence.
Given $ b, c \in \mathbb{N}_0 $, say that $ \mathbf{T} $ is \emph{$ (b, c) $-spreading} if $ |T_0 \cap T_{m}| = b $ and $ \mathrm{Len}(I_v) \le \frac{m}{c} $ for every $ v \in V(\mathbf{T}) \setminus (T_0 \cap T_{m}) $. 
We immediately make the following observation.

\begin{fact} \label{fact:ramsey_spreading_subint}
    Let $ 0 \le b \le j \le k $, $ m, c \in \mathbb{N} $, and $ \lambda > \frac{1}{c} $.
    Suppose $ \mathbf{T} $ is a $ (b, c) $-spreading non-repeating $ j $-sequence and $ I \subseteq [0, m] $ is an interval with $ \mathrm{Len}(I) \ge \lambda m $.
    Then $ \mathbf{T}_I $ is $ (b, \lfloor \lambda c \rfloor) $-spreading.
\end{fact}

\begin{proof}
    Let $ m' \coloneqq \mathrm{Len}(I) \ge \lambda m $ and write $ I = [\alpha, \beta ] $.
    Note that $ T_0 \cap T_m \subseteq T_{\alpha} \cap T_{\beta} $.
    Since restricting to the interval $ I $ can only decrease the lengths of the intervals $ I_v $ for $ v \in V(\mathbf{T}_I) $, we have $ \mathrm{Len}(I_v) \le \frac{m}{c} \le \frac{m'}{\lambda c} < m' $ for every $ v \in V(\mathbf{T}_I) \setminus (T_0 \cap T_m) $.
    In particular, $ T_{\alpha} \cap T_{\beta} = T_0 \cap T_m $, so it follows that $ \mathbf{T}_I $ is $ (b, \lfloor \lambda c \rfloor) $-spreading.
\end{proof}

We now start by showing that it suffices to prove that a subsequence of $ \mathbf{T}(F) $ is $ (b, 2c) $-spreading.

\begin{lemma} \label{fact:ramsey_spreading_concentrating}
    Let $0 \le b < j \le k$, $m \in \mathbb{N}$, $m' \in [m]$, and $c \in [\lfloor \frac{m'}{2} \rfloor]$. 
    Let $\mathbf{F} = (F, \mathbf{S})$ be a $j$-clique-path of length $m$.
    Suppose there exist indices $0 = i_0 < i_1 < \dots < i_{m'} = m$ such that the subsequence $ \mathbf{T} \coloneqq (S_{i_q})_{q \in [0, m']}$ is $(b, 2c)$-spreading.
    Then $\mathbf{F}$ is $(b, c)$-centred.
\end{lemma}

\begin{proof}
    Note that $ |S_0 \cap S_m| = |T_0 \cap T_{m'}| = b $ by the definition of $ (b, 2c) $-spreading.
    Write $ I_v $ and $ I'_v $ for the vertex index sets of $ v \in V(\mathbf{S}) $ in $ \mathbf{S} $ and $ \mathbf{T} $, respectively, taking $ I'_v $ to be empty in the case that $ v \not \in V(\mathbf{T}) $.
    Observe that, for every $ v \in V(\mathbf{T}) \subseteq V(\mathbf{S}) $, we have $ \{ i_q: q \in I'_v \} = I_v \cap \{ i_q: q \in [0, m'] \} $.
    Now suppose for contradiction that there exist $ d < c $ and a $ (S_0 \cap S_{m}) $-avoiding walk $ v_0, \ldots, v_d \in V(F) $ with $ v_0 \in S_0 \setminus S_{m} $ and $ v_d \in S_{m} \setminus S_0 $.

    Firstly note that, without loss of generality, for every $ i \in [d - 1] $, there exists $ q \in [m - 1] $ such that $ v_i \in V(F_q \cap F_{q + 1}) \subseteq S_q $.
    Indeed, if $ v_i \in V(F_q \setminus (F_{q - 1} \cup F_{q + 1})) $ for some $ q \in [m] $, then $ v_i $ is only contained in edges with other vertices of $ F_q $, so $ v_{i - 1}, v_{i + 1} \in V(F_q) $, but this means that $ v_{i - 1} $ and $ v_{i + 1} $ share an edge, so we may delete $ v_i $ from the walk.

    Secondly, note further, for every $ i \in [0, d - 1] $, that $ \min{I_{v_{i + 1}}} \le \max{I_{v_i}} + 1 $.
    Indeed, if $ q_1 \le q_2 \le q_3 $, then $ V(F_{q_1}) \cap V(F_{q_3}) \subseteq V(F_{q_2}) $.
    This implies that, if $ v_i \in V(F_{\hat{q}}) $ for some $ \hat{q} \ge q + 1 $, then $ v_i \in V(F_{\hat{q}} \cap F_q) \subseteq V(F_{\hat{q}} \cap F_{\hat{q} - 1}) $, and likewise if $ \hat{q} \le q $, then $ v_i \in V(F_{\hat{q}} \cap F_{q + 1}) \subseteq V(F_{\hat{q}} \cap F_{\hat{q} + 1}) $.
    We deduce that $ v_i \in V(F_{\hat{q}}) $ if and only if $ \hat{q} \in I_{v_i} $ or $ \hat{q} - 1 \in I_{v_i} $.
    Since there is an edge in $ F $ containing both $ v_i $ and $ v_{i + 1} $, there must exist $ \hat{q} \in [m] $ such that $ v_i, v_{i + 1} \in F_{\hat{q}} $, which means that $ \min{I_{v_{i + 1}}} \le \hat{q} $ and $ \max{I_{v_i}} \ge \hat{q} - 1 $, so $ \min{I_{v_{i + 1}}} \le \max{I_{v_i}} + 1 $.
    Since $ 0 \in I_{v_0} $ and $ m \in I_{v_d} $, it follows that $ [0, m] \subseteq \bigcup_{i = 0}^d I_{v_i} $, and so in particular $ [0, m'] \subseteq \bigcup_{i = 0}^d I'_{v_i} $.

    On the other hand, since $ v_i \not \in S_0 \cap S_{m} $ for each $ i \in [0, d] $, we know that each $ \mathrm{Len}(I'_{v_i}) \le \frac{m'}{2c} $, and in particular $ |I'_{v_i}| \le \frac{m'}{2c} + 1 $.
    This means that $ m' + 1 \le (d + 1) \left( \frac{m'}{2c} + 1 \right) \le m' $, yielding the required contradiction.
\end{proof}

We now show in general that non-repeating sequences contain large spreading intervals.

\begin{lemma} \label{fact:ramsey_basic}
    Suppose $ 1 / m \ll 1/C \ll 1/c, 1/k $ and let $ j \in [0, k] $.
    For any non-repeating $ j $-sequence $ \mathbf{T} $ of length $ m $, there exist $ i_1 \in [0, m] $, $ i_2 \in [i_1 + \frac{m}{C}, m] $, and $ b \in [0, j] $ such that $ \mathbf{T}_{[i_1, i_2]} $ is $ (b, c) $-spreading.
\end{lemma}

\begin{proof}
    We work by induction on $ j \in [0, k] $, proving the statement for any $ C \ge c^j $ and $ m \ge m_0(j) $, for some integer $ m_0(j) $, which we define inductively.
    
    The case $ j = 0 $ is trivial, taking $ b = 0 $. 
    Now fix $ j \in [k] $, and assume that the statement holds for $ j - 1 $.
    If $ \mathrm{Len}(I_v) \le \frac{m}{c} $ for every $ v \in V(\mathbf{T}) $, then we are done, so assume this is not the case.
    
    Choose $ v $ such that $ I_v = [\hat{i}_1, \hat{i}_2] $ has $ m' \coloneqq \hat{i}_2 - \hat{i}_1 \ge \frac{m}{c} \ge m_0(j - 1) $, by choosing $ m_0(j) $ sufficiently large.
    We obtain a non-repeating $ (j - 1) $-sequence $ \mathbf{T}' $ of length $ m' $ by setting $ T'_i \coloneqq T_{\hat{i}_1 + i} \setminus \{ v \} $ for each $ i \in [0, m'] $.
    Then, by the induction hypothesis with $ \frac{C}{c} \ge c^{j - 1} $ playing the role of $ C $, there exist $ i_1' \in [0, m'] $, $ i_2' \in [i_1' + \frac{c m'}{C}, m'] $, and $ b' \in [0, j - 1] $ such that $ \mathbf{T}'_{[i_1', i_2']} $ is $ (b', c) $-spreading.
    Set $ i_1 \coloneqq \hat{i}_1 + i_1' $, $ i_2 \coloneqq \hat{i}_1 + i_2' $, and $ b \coloneqq b' + 1 $, so $ i_2 - i_1 \ge \frac{c m'}{C} \ge \frac{m}{C} $.
    It thus follows that $ \mathbf{T}_{[i_1, i_2]} $ is $ (b, c) $-spreading, as required for the inductive step; this completes the proof.
\end{proof}

We may now combine \cref{fact:ramsey_spreading_concentrating,fact:ramsey_basic} to prove \cref{lem:ramsey_full}.

\begin{proof}[Proof of \cref{lem:ramsey_full}]
    Suppose further $ 1/i \ll 1/C \ll 1/c, 1/k $.
    Define a sequence $ 0 = a_0 < a_1 < \cdots < a_{2i} = M(\mathbf{s}, \ell) $ by taking $ a_{2q} \coloneqq M(\mathbf{s}_{[q]}, \ell) $ and $ a_{2q + 1} \coloneqq M(\mathbf{s}_{[q]}, \ell) + \ell $ for each $ q \in [0, i - 1] $.
    Let $ \mathbf{T} \coloneqq (S_{a_q})_{q \in [0, 2i]} $.

    We may apply \cref{fact:ramsey_basic} to $ \mathbf{T} $ to find $ i'_1 \in [0, 2i], i'_2 \in [i'_1 + \frac{2i}{C}, 2i] $, and $ b \in [0, j] $ such that $ \mathbf{T}_{[i'_1, i'_2]} $ is $ (b, 3c) $-spreading.
    Note in fact that $ b \in [0, j - 1] $; indeed, we cannot have $ b = j $, since this would imply that $ S_{a_{i'_1}} = S_{a_{i'_2}} $, contradicting the pairwise distinctness of the sets $ (S_q) $.
    To ensure a subsequence of the required form, we may then choose $ i_1, i_2 \in [0, i] $ such that $ i'_1 \le 2 i_1 < 2 i_2 - 1 \le i'_2 $ and $ 2i_2 - 1 - 2i_1 \ge i'_2 - i'_1 - 2 \ge \frac{2}{3} (i'_2 - i'_1) $.
    By \cref{fact:ramsey_spreading_subint}, the subsequence $ \mathbf{T}_{[2 i_1, 2i_2 - 1]} $ is $ (b, 2c) $-spreading.
    We may therefore apply \cref{fact:ramsey_spreading_concentrating}, with $ (\mathbf{F}_{[M(\mathbf{s}_{[i_1]}, \ell), M(\mathbf{s}_{[i_2 - 1]}, \ell) + \ell]}, \mathbf{T}_{[2 i_1, 2i_2 - 1]}) $ playing the role of $ (\mathbf{F}, \mathbf{T}) $, to see that $ \mathbf{F}_{[M(\mathbf{s}_{[i_1]}, \ell), M(\mathbf{s}_{[i_2 - 1]}, \ell) + \ell]} $ is $ (b, c) $-centred, as required.
    
    To see the second statement, note that clearly whether or not $ \mathbf{F}_J $ is $ (b, c) $-centred is fully determined by $ \mathbf{F}_J $ (that is, independent of the rest of $ \mathbf{F} $), so by choosing $ (i_2, i_1) $ to be lexicographically minimal, we ensure that our choice depends only upon $ \mathbf{F}_{[0, m_2]} $.
    Furthermore, the sizes of any intersections of the root sets in $ \mathbf{S} $ are invariant under semi-isomorphism, so clearly whether or not $ \mathbf{F}_J $ is $ (b, c) $-centred is also independent of this.
\end{proof}

\subsection{Centred clique-path characterisation} \label{sect:char_cent_decomp}

In this section, we use \cref{lem:ramsey_full} to partition the multiset $ \mathcal{G}_{\ell, \ell}^{(j)} $ of clique-paths from \cref{lem:process_graphs_char} according to their centred subpaths, which will be useful in the proof of \cref{lem:norm_decrease}.

\begin{lemma} \label{lem:centered_decomp}
    Suppose $ 1/\ell \ll 1/c, 1/k $ and let $ j \in [k] $.
    Then there exists a partition $ \mathcal{B}^*_{j, \ell} $ of the multiset $ \mathcal{G}_{\ell, \ell}^{(j)} $ and, for each $ \mathcal{B} \in \mathcal{B}^*_{j, \ell} $, integers $ b(\mathcal{B}) \in [0, j - 1] $, $ i_2(\mathcal{B}) \in [\ell] $, and $ m(\mathcal{B}) \le 2^{\ell^{3j}} $, a $ j $-clique-path $ \hat{\mathbf{F}}_1(\mathcal{B}) $, and a $ (b(\mathcal{B}), c) $-centred semi-ordered $ j $-clique-path $ \hat{\mathcal{F}}_2(\mathcal{B})$, each of length at most $ \ell^{3k} $, such that
    $$ \Xi[\mathcal{B}] = m(\mathcal{B}) \cdot \Xi[\hat{\mathbf{F}}_1(\mathcal{B})] \circ \Xi[\hat{\mathcal{F}}_2(\mathcal{B})] \circ \Xi[\hat{\mathcal{F}}_3(j, \ell, i_2(\mathcal{B}))], $$
    where
    $$ \hat{\mathcal{F}}_3(j, \ell, i_2) \coloneqq \bigcup_{\mathbf{s}' \in \mathcal{S}_{j, \ell}} \mathcal{F}_{\mathbf{s}'}^{(j)} \bullet \mathcal{G}_{\ell, \ell - i_2}^{(j)},$$
    noting that the union is with multiplicity.
\end{lemma}

\begin{proof}
    Let $ \mathbf{s}, \mathbf{s}' \in (\mathcal{S}_{j, \ell})^{\ell} $, and $ \mathbf{F} \in \mathcal{G}_{\ell}^{(j)}(\mathbf{s}) $ and $ \mathbf{F}' \in \mathcal{G}_{\ell}^{(j)}(\mathbf{s}')$ be two $ j $-clique-paths in $ \mathcal{G}_{\ell, \ell}^{(j)} $.
    We define an equivalence relation $ \sim $ on $ \mathrm{Set}(\mathcal{G}_{\ell, \ell}^{(j)}) $ by saying that $ \mathbf{F} \sim \mathbf{F}' $ if and only if the following holds.
    Firstly, we require that $ i_1 \coloneqq i_1(\mathbf{F}) = i_1(\mathbf{F}') $, $ i_2 \coloneqq i_2(\mathbf{F}) = i_2(\mathbf{F}') $, $ b \coloneqq b(\mathbf{F}) = b(\mathbf{F}') $, and $ \mathbf{s}_{[i_2 - 1]} = \mathbf{s}'_{[i_2 - 1]} $.
    Secondly, recalling the definition of $ M(\mathbf{s}, \ell) $ from \cref{sect:proof_concentrating} and writing $ M_1 \coloneqq M(\mathbf{s}_{[i_1]}, \ell) $ and $ M_2 \coloneqq M(\mathbf{s}_{[i_2 - 1]}, \ell) + \ell $, we require that $ \mathbf{F}_{[0, M_1]} = \mathbf{F}'_{[0, M_1]} $ (that is, they are isomorphic) and that $ \mathbf{F}_{[M_1, M_2]} $ and $ \mathbf{F}'_{[M_1, M_2]} $ are semi-isomorphic.
    In other words, we essentially split each $ \mathbf{F} = \hat{\mathbf{F}}_1 \bullet \hat{\mathbf{F}}_2 \bullet \hat{\mathbf{F}}_3 $, where $ \hat{\mathbf{F}}_2 $ is $ (b, c) $-centred, and partition $ \mathcal{G}_{\ell, \ell}^{(j)} $ according to the isomorphism classes of $ \hat{\mathbf{F}}_1 = (\hat{F}_1, \hat{\mathbf{S}}_1) $ and $ \hat{\mathbf{F}}_2 = (\hat{F}_2, \hat{\mathbf{S}}_2) $, treating sets in $ \hat{\mathbf{S}}_1 $ as ordered but ignoring the orderings of all but the first set in $ \hat{\mathbf{S}}_2 $.
    
    Define now $ \mathcal{B}^*_{j, \ell} \coloneqq \mathcal{G}_{\ell, \ell}^{(j)} / \sim $ to be the set of equivalence classes, treating each $ \mathcal{B} \in \mathcal{B}^*_{j, \ell} $ as a multiset, with multiplicities matching those in $ \mathcal{G}_{\ell, \ell}^{(j)} $.
    Given $ \mathcal{B} \in \mathcal{B}^* $, take an arbitrary clique-path $ \mathbf{F} \in \mathrm{Set}(\mathcal{B}) $ and define $ \hat{\mathbf{F}}_1(\mathcal{B}) \coloneqq \mathbf{F}_{[0, M_1]} $ and $ \hat{\mathcal{F}}_2(\mathcal{B}) \coloneqq \mathcal{SO}(\mathbf{F}_{[M_1, M_2]}) $, noting that these are well-defined by the definition of $ \sim $.
    Recall from the definitions in \cref{section:clique_paths_def} that $ \hat{\mathbf{F}}_1 $ is a $ j $-clique-path (with ordered root sets) and $ \hat{\mathcal{F}}_2 $ is a semi-ordered $ j $-clique-path (a clique-path in which only the first root set is ordered).
    By the choice of $ M_1 $ and $ M_2 $, every $ \mathbf{F}' \in \hat{\mathcal{F}}_2(\mathcal{B}) $ is $ (b, c) $-centred.

    \begin{claim}
        We have
        $$ \mathrm{Set}(\mathcal{B}) = \hat{\mathbf{F}}_1(\mathcal{B}) \bullet \hat{\mathcal{F}}_2(\mathcal{B}) \bullet \bigcup_{\mathbf{s}' \in \mathrm{Set}(\mathcal{S}_{j, \ell})} \mathcal{F}_{\mathbf{s}'}^{(j)} \bullet \mathrm{Set}(\mathcal{G}_{\ell, \ell - i_2}^{(j)}) $$
    as sets, noting that the union over $ \mathcal{S}_{j, \ell} $ is considered without multiplicity.
    \end{claim}

    In other words, $ \mathrm{Set}(\mathcal{B}) $ consists of all clique-paths in $ \mathcal{G}_{\ell, \ell}^{(j)} $ with some fixed initial segment, the first part of which has a fixed ordering.

    \begin{claimproof}
        For every $ \mathbf{F} \in \mathrm{Set}(\mathcal{B}) $, note that there exists (a unique) $ \mathbf{F}' \in \hat{\mathcal{F}}_2(\mathcal{B}) $, corresponding to a choice of orderings of the root sets, so that $ \hat{\mathbf{F}}_1(\mathcal{B}) \bullet \mathbf{F}' = \mathbf{F}_{[0, M_2]} $.
        Furthermore, by the definition of $ \mathcal{G}_{\ell}^{(j)}(\mathbf{s}) $, the clique-path $ \mathbf{F}_{[M_2, M(\mathbf{s}, \ell)]} $ is an element of $ \mathcal{F}_{\mathbf{s}_{i_2}}^{(j)} \bullet \mathrm{Set}(\mathcal{G}_{\ell, \ell - i_2}^{(j)}) $.
        This proves the inclusion of the left-hand side in the right-hand side.

        Now let $ \mathbf{F} \in \mathrm{Set}(\mathcal{B}) $ and $ \mathbf{F}' $ be an element of the right-hand side with $ \mathbf{F}' \in \mathcal{G}_{\ell}^{(j)}(\mathbf{s}') $ for some $ \mathbf{s}' \in (\mathcal{S}_{j, \ell})^{\ell} $.
        By definition, we have that $ \mathbf{F}_{[0, M_1]} = \mathbf{F}'_{[0, M_1]} $ and that $ \mathbf{F}_{[M_1, M_2]} $ and $ \mathbf{F}'_{[M_1, M_2]} $ are semi-isomorphic, so in particular $ \mathbf{F}_{[0, M_2]} $ and $ \mathbf{F}'_{[0, M_2]} $ are semi-isomorphic.
        It is thus clear that $ \mathbf{s}_{[i_2 - 1]} = \mathbf{s}'_{[i_2 - 1]} $, and we recall from \cref{lem:ramsey_full} that this also implies $ i_1(\mathbf{F}) = i_1(\mathbf{F}') $, $ i_2(\mathbf{F}) = i_2(\mathbf{F}') $, and $ b(\mathbf{F}) = b(\mathbf{F}') $.
        As such, $ \mathbf{F} \sim \mathbf{F}' $, and so $ \mathbf{F}' \in \mathrm{Set}(\mathcal{B}) $, as required.
    \end{claimproof}

    Now define $ m(\mathcal{B}) \le 2^{(2\ell)^j (i_2 - 1)} \le 2^{\ell^{3j}} $ to be the multiplicity of $ \mathbf{s}_{[i_2 - 1]} $ in $ (\mathcal{S}_{j, \ell})^{i_2 - 1} $, noting that $ \mathbf{s}_{[i_2 - 1]} $ is uniquely determined by the class $ \mathcal{B} $ (independent of the choice of $ \mathbf{F}$), and let $ \mathcal{B}' $ be the multiset with $ \mathrm{Set}(\mathcal{B}') \coloneqq \hat{\mathbf{F}}_1(\mathcal{B}) \bullet \hat{\mathcal{F}}_2(\mathcal{B})$ in which every element has multiplicity $ m(\mathcal{B}) $.

    \begin{claim}
        We have
        $$ \mathcal{B} = \mathcal{B}' \bullet \bigcup_{\mathbf{s}' \in \mathcal{S}_{j, \ell}} \mathcal{F}_{\mathbf{s}'}^{(j)} \bullet \mathcal{G}_{\ell, \ell - i_2}^{(j)} $$
        as multisets, noting that the union over $ \mathcal{S}_{j, \ell} $ is now considered with multiplicity.
    \end{claim}

    \begin{claimproof}
        Fix some $ \mathbf{F} \in \mathrm{Set}(\mathcal{B}) $ and observe that there is a unique $ \mathbf{s} \in \mathrm{Set}((\mathcal{S}_{j, \ell})^{\ell}) $ for which $ \mathbf{F} \in \mathcal{G}_{\ell}^{(j)}(\mathbf{s}) $, since this is determined by the sequence $ \mathbf{s}(\mathbf{F}) $.
        By the definition of $ \mathcal{G}_{\ell, \ell}^{(j)} $, the multiplicity of $ \mathbf{F} $ in $ \mathcal{B} $ is exactly the multiplicity of $ \mathbf{s} $ in $ (\mathcal{S}_{j, \ell})^{\ell} $.
        This is simply the product of $ m(\mathcal{B}) $ with the multiplicity of $ \mathbf{s}_{[i_2, \ell]} $ in $ (\mathcal{S}_{j, \ell})^{\ell - i_2 + 1} $ (noting that the latter may vary for distinct $ \mathbf{F} \in \mathrm{Set}(\mathcal{B}) $).
        By the definition of the multiset union, this is equal to the multiplicity of $ \mathbf{F} $ on the right-hand side.
    \end{claimproof}

    It now follows by \cref{fact:xi_comp_commute} that
    $$ \Xi[\mathcal{B}] = \sum_{\mathbf{F} \in \mathcal{B}} \Xi[\mathbf{F}] = m(\mathcal{B}) \cdot \Xi[\hat{\mathbf{F}}_1(\mathcal{B})] \circ \Xi[\hat{\mathcal{F}}_2(\mathcal{B})] \circ \Xi[\hat{\mathcal{F}}_3(j, \ell, i_2(\mathcal{B}))], $$
    as required.
\end{proof}

We now have all the tools we need to prove our main lemma.

\subsection{Proof of key lemma} \label{sect:proof_lem_decrease}

We now proceed to prove \cref{lem:norm_decrease} by induction on $ j \in [0, k] $, using the inductive definition of $ \mathcal{P}_{j, \ell} $.
Intuitively, the idea is that the operator $ \mathcal{R}_j^{\ell} $ replaces the discrepancy $ \xi^{(j)} $ by a suitably weighted sum over discrepancies $ \xi^{(s)} $ for different $ s \in [0, j - 1] $.
The inductive hypothesis says that applying $ \mathcal{P}_{j - 1, \ell} $ significantly reduces all discrepancies $ \xi^{(s)} $ for $ s \in [0, j - 1] $, so we are able to conclude that $ \mathcal{P}_{j, \ell} $ further reduces the discrepancy $ \xi^{(j)} $.

Let us start by very briefly and roughly summarising the inductive step, ignoring many details for now.
For simplicity, let $ C $ be a suitably large constant for the purpose of this summary.
We use the partition $ \mathcal{B}^* $ in \cref{lem:centered_decomp} and analyse the effect of $ \hat{\mathcal{F}}_3 $, $ \hat{\mathcal{F}}_2 $, and $ \hat{\mathbf{F}}_1 $ one at a time.
Firstly, we write $ \langle \Xi[\mathcal{B}], \xi^{(j)}[\varphi] \rangle = \langle \Xi[\hat{\mathbf{F}}_1] \circ \Xi[\hat{\mathcal{F}}_2], \xi^{(j)}[\widehat{\psi}] \rangle $ for some $ K_r $-function $ \psi $ with $ \xi^{(j)}[\widehat{\psi}] = \langle \Xi[\hat{\mathcal{F}}_3], \xi^{(j)}[\varphi] \rangle $.
We then use the induction hypothesis to show that $ \norm{\xi^{(b)}[\widehat{\psi}]} \le \eps \norm{\xi^{(j)}[\varphi]} \frac{d_b}{d_j} $ for any $ b \in [0, j - 1] $, recalling from \cref{fact:uniformity_disc} that, for a general $ K_r $-function, we expect $ \norm{\xi^{(b)}[\varphi]} $ to be (at most) of the order of magnitude $ \norm{\xi^{(j)}[\varphi]} \frac{d_b}{d_j} $, because each set in $ E_b $ is contained in $ \Theta(\frac{d_b}{d_j}) $ sets in $ E_j $.
Next, we use the concentration in \ref{cond:creg3} for the $ (b, c) $-centred semi-ordered $ j $-clique-path $ \hat{\mathcal{F}}_2 $ to show that $ \norm{\langle \Xi[\hat{\mathcal{F}}_2], \xi^{(j)}[\widehat{\psi}] \rangle} \le C \norm{\xi^{(b)}[\widehat{\psi}]} \frac{d_j}{d_b} \le C \eps \norm{\xi^{(j)}[\varphi]} $.
Finally, we use the fact that $ \norm{\Xi[\hat{\mathbf{F}}_1]} \le C $ and $ |\mathcal{B}^*| \le C $ to deduce the desired bound.
We now proceed with the full proof of \cref{lem:norm_decrease}.

\begin{proof}[Proof of \cref{lem:norm_decrease}]
    Introduce a new constant $ C_3 > 0 $ with $ 1/C_2 \ll 1/C_3 \ll 1/C_1 $.
    We prove by induction on $ j \in [0, k] $ that
    \begin{equation} \label{eqn:norm_dec_ih}
        \norm{\xi^{(j)}[\mathcal{P}_{j, \ell} \varphi]} \le C_3^{j + 1} \eps \norm{\xi^{(j)}[\varphi]}
    \end{equation}
    for any globally-balanced $ K_r $-function $ \varphi $ on $ G $.
    Note that this clearly implies the statement.
    
    Note that the base case $ j = 0 $ is immediate from the globally-balanced property of \cref{fact:global_balance}, since $ \norm{\xi^{(0)}[\varphi]} = \left| \sum_{e \in E(G)} \xi[\varphi](e) \right| = 0 $.
    We now proceed with the inductive step, so assume that the statement holds up to $ j - 1 $, for some $ j \in [k] $.
    By \cref{lem:process_graphs_char}~\ref{cond:proc_graph_char_2} with $ \ell $ playing the role of $ i $ and \cref{lem:centered_decomp}, we may write
    \begin{equation} \label{eqn:norm_dec_-2}
        \xi^{(j)}[\mathcal{P}_{j, \ell} \varphi]
            = \xi^{(j)}[(\mathcal{R}_j^{\ell} \mathcal{P}_{j - 1, \ell})^{\ell} \varphi]
            = \sum_{\mathbf{F} \in \mathcal{G}_{\ell, \ell}^{(j)}} \langle \Xi[\mathbf{F}], \xi^{(j)}[\varphi] \rangle
            = \sum_{\mathcal{B} \in \mathcal{B}^*_{j, \ell}} \langle \Xi[\mathcal{B}], \xi^{(j)}[\varphi] \rangle.
    \end{equation}
    Since $ |\mathcal{G}_{\ell, \ell}^{(j)}| $ can easily be bounded in terms of $ \ell $, we now fix $ \mathcal{B} \in \mathcal{B}^*_{j, \ell} $, and aim to bound $ \norm{\langle \Xi[\mathcal{B}], \xi^{(j)}[\varphi] \rangle} $.
    Write $ b \coloneqq b(\mathcal{B}) $, $ i_2 \coloneqq i_2(\mathcal{B}) $, $ \hat{\mathbf{F}}_1 \coloneqq \hat{\mathbf{F}}_1(\mathcal{B}) $, and $ \hat{\mathcal{F}}_2 \coloneqq \hat{\mathcal{F}}_2(\mathcal{B}) $, as in \cref{lem:centered_decomp}.
    To obtain the desired bound, we consider $ \hat{\mathbf{F}}_1 $, $ \hat{\mathcal{F}}_2 $, and $ \hat{\mathcal{F}}_3 $ separately.
    
    Define first $ \psi \coloneqq (\mathcal{R}_{j}^{\ell} \mathcal{P}_{j - 1, \ell})^{\ell - i_2} \varphi $ and $ \widehat{\psi} \coloneqq \mathcal{P}_{j - 1, \ell} \psi $.
    Using \cref{lem:process_graphs_char}~\ref{cond:proc_graph_char_1}~and~\ref{cond:proc_graph_char_2} for the first two equalities, respectively, as well as associativity and \cref{fact:xi_comp_commute} for the final equality, we see that
    \begin{align*}
        \xi^{(j)}[\widehat{\psi}]
            % = \xi^{(j)}[\mathcal{P}_{j - 1, \ell} (\mathcal{R}_{j}^{\ell} \mathcal{P}_{j - 1, \ell})^{\ell - i_2} \varphi]
            &= \left\langle \sum_{\mathbf{s}' \in \mathcal{S}_{j, \ell}} \Xi[\mathcal{F}_{\mathbf{s}'}^{(j)}], \xi^{(j)}[\psi] \right\rangle\\
            &= \left\langle \sum_{\mathbf{s}' \in \mathcal{S}_{j, \ell}} \Xi[\mathcal{F}_{\mathbf{s}'}^{(j)}], \langle \Xi[\mathcal{G}_{\ell, \ell - i_2}^{(j)}], \xi^{(j)}[\varphi] \rangle \right\rangle
            = \langle \Xi[\hat{\mathcal{F}}_3(j, \ell, i_2)], \xi^{(j)}[\varphi] \rangle.
    \end{align*}
    This means, by \cref{lem:centered_decomp} and associativity, that
    \begin{align}
        \langle \Xi[\mathcal{B}], \xi^{(j)}[\varphi] \rangle
            &= \langle m(\mathcal{B}) \cdot \Xi[\hat{\mathbf{F}}_1] \circ \Xi[\hat{\mathcal{F}}_2] \circ \Xi[\hat{\mathcal{F}}_3(j, \ell, i_2)], \xi^{(j)}[\varphi] \rangle \nonumber \\
            &= m(\mathcal{B}) \cdot \left\langle\Xi[\hat{\mathbf{F}}_1] \circ \Xi[\hat{\mathcal{F}}_2], \xi^{(j)}[\widehat{\psi}] \right\rangle. \label{eqn:norm_dec_-1}
    \end{align}
    We now start by using the inductive hypothesis to prove the following bounds on the discrepancies of $ \widehat{\psi} $; recall \eqref{eqn:gamma_d_defns}.

    \begin{claim} \label{claim:norm_dec_1}
        We have
        \begin{enumerate}[\rm (\roman*)]
            \item $ \norm{\xi^{(s)}[\widehat{\psi}]} \le C_3^{j + \frac{1}{2}} \eps \frac{d_s}{d_j} \norm{\xi^{(j)}[\varphi]} $ for every $ s \in [0, j - 1] $; \label{cond:norm_dec_claim_1}
            \item $ \norm{\xi^{(j)}[\widehat{\psi}]} \le C_3 \norm{\xi^{(j)}[\varphi]} $. \label{cond:norm_dec_claim_2}
        \end{enumerate}
    \end{claim}

    \begin{claimproof}
        We may naïvely bound the effect of $ (\mathcal{R}_{j}^{\ell} \mathcal{P}_{j - 1, \ell})^{\ell - i_2} $ on the $ j $-discrepancy using \cref{fact:norm_delta_naive}~\ref{cond:norm_naive_5}, obtaining
        \begin{equation} \label{eqn:norm_dec_1}
            \norm{\xi^{(j)}[\psi]} \le C_1 \norm{\xi^{(j)}[\varphi]},
        \end{equation}
        from which we immediately deduce a bound on the $ (j - 1) $-discrepancy by \cref{fact:uniformity_disc}~\ref{cond:unif_disc_4}, namely that
        $$ \norm{\xi^{(j - 1)}[\psi]} \le 2 \frac{d_{j - 1}}{d_j} \norm{\xi^{(j)}[\psi]} \le 2 C_1 \frac{d_{j - 1}}{d_j} \norm{\xi^{(j)}[\varphi]}. $$
        Since $ \psi $ is globally-balanced by \cref{fact:global_balance}, we may apply the induction hypothesis \eqref{eqn:norm_dec_ih} with $ (\psi, j - 1) $ playing the role of $ (\varphi, j) $ to see that
        $$ \norm{\xi^{(j - 1)}[\widehat{\psi}]} \le C_3^j \eps \norm{\xi^{(j - 1)}[\psi]} \le 2 C_1 C_3^{j} \eps \frac{d_{j - 1}}{d_j} \norm{\xi^{(j)}[\varphi]}. $$
        Again, using \cref{fact:uniformity_disc}~\ref{cond:unif_disc_4}, we immediately deduce an analogous bound for any $ s \in [j - 1] $, specifically that
        \begin{equation*} \label{eqn:norm_dec_2}
            \norm{\xi^{(s)}[\widehat{\psi}]} \le 2 \frac{d_s}{d_{j - 1}} \norm{\xi^{(j - 1)}[\widehat{\psi}]} \le 4 C_1 C_3^{j} \eps \frac{d_s}{d_j} \norm{\xi^{(j)}[\varphi]},
        \end{equation*}
        which suffices for \ref{cond:norm_dec_claim_1}.
        Note that, since $ \widehat{\psi} $ is also globally-balanced by \cref{fact:global_balance}, this bound holds trivially for $ s = 0 $.
        For the $ j $-discrepancy, we again use the naïve bounds from \cref{fact:norm_delta_naive}~\ref{cond:norm_naive_5} (with $ \ell $ playing the role of $ i $) and \eqref{eqn:norm_dec_1} to see that
        \begin{equation*} \label{eqn:norm_dec_2.5}
            \norm{\xi^{(j)}[\widehat{\psi}]} \le C_1 \norm{\xi^{(j)}[\psi]} \le C_1^2 \norm{\xi^{(j)}[\varphi]},
        \end{equation*}
        which suffices for \ref{cond:norm_dec_claim_2}.
    \end{claimproof}

    We next analyse the effect of $ \Xi[\hat{\mathcal{F}}_2] $ using \ref{cond:creg3}.
    
    \begin{claim} \label{claim:norm_dec_2}
        We have $ \norm{\langle \Xi[\hat{\mathcal{F}}_2], \xi^{(j)}[\widehat{\psi}] \rangle} \le C_3^{j + \frac{2}{3}} \eps \norm{\xi^{(j)}[\varphi]}. $
    \end{claim}

    \begin{claimproof}
        Define $ \chi: E_j^2 \to \mathbb{R} $ by $ \chi(S, T) \coloneqq \mathds{1}[\iota(\hat{\mathcal{F}}_2) \subseteq \iota(S \cap T, S)] \hat{\chi} $, where $ \hat{\chi} \coloneqq \alpha_{\hat{\mathcal{F}}_2} \frac{\gamma_j}{\gamma_b} $, recalling the definitions of $ \iota(\mathcal{F}) $ in \cref{section:clique_paths_def} and $ \alpha_{\mathcal{F}} $ in \cref{sect:prop_reg}, and that the length of $ \hat{\mathcal{F}}_2 $ is at most $ \ell^{3k} $.
        Note, by \ref{cond:creg2}, \cref{fact:uniformity_disc}~\ref{cond:unif_disc_2}, and the bound on $ |\alpha_{\hat{\mathcal{F}}_2}| $ in the definition of clique-regularity, that
        \begin{equation} \label{eqn:norm_dec_3}
            |\hat{\chi}| \le |\alpha_{\hat{\mathcal{F}}_2}| \cdot \frac{\gamma_j}{\gamma_b} \le C_1^2 \frac{d_j}{d_b}.
        \end{equation}
        By \ref{cond:creg3}, recalling from \cref{section:clique_paths_def} that $ \Xi[\mathcal{F}](S, T) = 0 $ whenever $ \iota(\mathcal{F}) \not \subseteq \iota(S \cap T, S) $, we may write $ \Xi[\hat{\mathcal{F}}_2] = \chi + \zeta_{\hat{\mathcal{F}}_2} $ for some $ \zeta_{\hat{\mathcal{F}}_2} $ with $ \norm{\zeta_{\hat{\mathcal{F}}_2}} \le \eps $, which means that
        \begin{equation} \label{eqn:norm_dec_3.5}
            \langle \Xi[\hat{\mathcal{F}}_2], \xi^{(j)}[\widehat{\psi}] \rangle = \langle \chi, \xi^{(j)}[\widehat{\psi}] \rangle + \langle \zeta_{\hat{\mathcal{F}}_2}, \xi^{(j)}[\widehat{\psi}] \rangle.
        \end{equation}
        Using \cref{fact:func_intersect_bound}, followed by \cref{claim:norm_dec_1}~\ref{cond:norm_dec_claim_1} and \eqref{eqn:norm_dec_3}, we obtain that
        \begin{equation} \label{eqn:norm_dec_4}
            \norm{\langle \chi, \xi^{(j)}[\widehat{\psi}] \rangle} \le 2^k |\hat{\chi}| \norm{\xi^{(b)}[\widehat{\psi}]} \le 2^k \cdot C_1^2 \frac{d_j}{d_b} \cdot C_3^{j + \frac{1}{2}} \eps \frac{d_b}{d_j} \norm{\xi^{(j)}[\varphi]} \le C_1^3 C_3^{j + \frac{1}{2}} \eps \norm{\xi^{(j)}[\varphi]}.
        \end{equation}
        Using \cref{claim:norm_dec_1}~\ref{cond:norm_dec_claim_2}, we see that
        \begin{equation} \label{eqn:norm_dec_5}
            \norm{\langle \zeta_{\hat{\mathcal{F}}_2}, \xi^{(j)}[\widehat{\psi}] \rangle} \le \norm{\zeta_{\hat{\mathcal{F}}_2}} \norm{\xi^{(j)}[\widehat{\psi}]} \le C_3 \eps \norm{\xi^{(j)}[\varphi]}.
        \end{equation}
        Plugging \eqref{eqn:norm_dec_4} and \eqref{eqn:norm_dec_5} into \eqref{eqn:norm_dec_3.5} now yields the desired bound.
    \end{claimproof}

    Finally, we complete the proof of \eqref{eqn:norm_dec_ih}.
    First observe that $ \norm{\Xi[\hat{\mathbf{F}}_1]} \le C_1^{\ell^{3k}} $ by \cref{fact:norm_delta_naive}~\ref{cond:norm_naive_3}, and recall the bound $ m(\mathcal{B}) \le 2^{\ell^{3j}} $ from \cref{lem:centered_decomp}.
    Hence, plugging this estimate and \cref{claim:norm_dec_2} into \eqref{eqn:norm_dec_-1}, we conclude that
    \begin{align*}
        \norm{\langle \Xi[\mathcal{B}], \xi^{(j)}[\varphi] \rangle}
            &\le m(\mathcal{B}) \cdot \norm{\Xi[\hat{\mathbf{F}}_1]} \cdot \norm{\langle \Xi[\hat{\mathcal{F}}_2], \xi^{(j)}[\widehat{\psi}] \rangle}\\
            &\le 2^{\ell^{3j}} \cdot C_1^{\ell^{3k}} \cdot C_3^{j + \frac{2}{3}} \eps \norm{\xi^{(j)}[\varphi]}
            \le C_3^{j + \frac{3}{4}} \eps \norm{\xi^{(j)}[\varphi]}.
    \end{align*}
    Note further (counting with multiplicity) that
    $$ |\mathcal{B}^*_{j, \ell}| \le |\mathcal{G}_{\ell, \ell}^{(j)}| \le |\mathcal{S}_{j, \ell}|^{\ell} \cdot \max_{\mathbf{s} \in (\mathcal{S}_{j, \ell})^{\ell}} |\mathcal{G}_{\ell}^{(j)}(\mathbf{s})| \le 2^{(2 \ell)^j \ell} \cdot C_1 \le C_1^2, $$
    using the bounds from \cref{lem:process_graphs_char} and \cref{fact:norm_delta_naive}~\ref{cond:norm_naive_1} for $ |\mathcal{S}_{j, \ell}| $ and $ |\mathcal{G}_{\ell}^{(j)}(\mathbf{s})| $, respectively.
    Plugging these bounds into \eqref{eqn:norm_dec_-2}, we deduce that
    $$ \norm{\xi^{(j)}[\mathcal{P}_{j, \ell} \varphi]} \le |\mathcal{B}^*_{j, \ell}| \cdot \norm{\langle \Xi[\mathcal{B}], \xi^{(j)}[\varphi] \rangle} \le C_1^2 \cdot C_3^{j + \frac{3}{4}} \eps \norm{\xi^{(j)}[\varphi]} \le C_3^{j + 1} \eps \norm{\xi^{(j)}[\varphi]}, $$
    as required for \eqref{eqn:norm_dec_ih}, completing the inductive step, and thus the proof.
\end{proof}

\subsection{Proof of main theorem} \label{sect:proof_final}

We are now ready to deduce \cref{thm:proc_conv_fcd} from \cref{lem:norm_decrease}.

\begin{proof}[Proof of \cref{thm:proc_conv_fcd}]
    Suppose $ 1/n, \eps \ll 1/C' \ll 1/C $.
    Let $ \varphi_t \coloneqq \mathcal{P}_{k, \ell}^t \varphi $ and regard $ (\varphi_t)_{t \ge 0} $ as a sequence in the complete metric space $ \mathbb{R}^{K_r(G)} $, equipped with the $ \ell^{\infty} $-norm.
    Note that $ \varphi_t $ is globally-balanced for all $ t \ge 0 $ by \cref{fact:global_balance}, and $ \norm{\xi[\varphi_0]} \le \eps $ by \cref{fact:uniform_initial_small}.
    Hence, by \cref{lem:norm_decrease}, we have $ \norm{\xi[\varphi_t]} \le (C' \eps)^t \eps $ for all $ t \ge 0 $, which means, by \cref{fact:norm_delta_naive}~\ref{cond:norm_naive_6} with $ (k, \ell, C') $ playing the role of $ (j, i, C) $, that
    $$ \norm{\varphi_{t + 1} - \varphi_t} \le \frac{C'}{\gamma_k} \norm{\xi[\varphi_t]} \le \frac{(C' \eps)^{t + 1}}{\gamma_k}. $$
    Hence, we see that
    \begin{align*}
        \sum_{t = 0}^{\infty} \norm{\varphi_{t + 1} - \varphi_t} &\le \frac{C' \eps}{\gamma_k} \sum_{t = 0}^{\infty} (C' \eps)^t
            \le \frac{C' \eps}{\gamma_k (1 - C' \eps)}
            \le \frac{2C' \eps}{\gamma_k}
            \le \frac{1}{\gamma_k}
            < \infty.
    \end{align*}
    In particular, the sequence $ (\varphi_t)_{t \ge 0} $ is Cauchy, so converges to some $ K_r $-function $ \varphi_{\infty} $ and $ \varphi_{\infty}(K) \ge \varphi_0(K) - \frac{1}{\gamma_k} \ge 0 $ for every $ K \in K_r(G) $.
    Also, since $ \norm{\xi[\varphi_t]} \to 0 $, and the map $ \varphi \mapsto \norm{\xi[\varphi]} $ is clearly continuous, we have that $ \xi [\varphi_{\infty}] \equiv 0 $.
    In other words, $ \varphi_{\infty} $ is a fractional $ K_r $-decomposition of $ G $, as required.
\end{proof}

This completes the proof of our main deterministic result.
The remaining section is dedicated to proving that typical random hypergraphs exhibit the required pseudorandomness properties.

\section{Random hypergraphs are clique-regular} \label{sect:hypergraph_psrand}

In this section we prove that the clique-regularity properties \ref{cond:creg1}--\ref{cond:creg3} required for \cref{thm:proc_conv_fcd} are indeed satisfied w.h.p.\ by a sufficiently dense random hypergraph.
Specifically, we prove the following theorem.

\begin{theorem} \label{lem:gnp_reg}
    Let $ k \ge 2 $, $ r \ge k + 1 $, and suppose $ 1/n \ll \delta, 1/C \ll 1/\ell, 1/c \ll 1/r, \eps $.
    Suppose $ G \sim G^{(k)}(n, p) $ for $ p \ge n^{-\frac{r - k}{\binom{r}{k} - 1} + \eps} $.
    Then $ G $ is $ (\ell, c, C, n^{-\delta}) $-clique-regular with probability at least $ 1 - e^{-n^{\delta}} $.
\end{theorem}

We remark that in the proof of \cref{lem:gnp_reg} we only use that $ 1/n \ll \delta \ll 1/c \ll 1/r, \eps $ and $ 1/n \ll 1/C \ll 1/\ell, 1/r $, but the stronger assumption above simplifies the statement and suffices for our purposes.
The bulk of the work for the proof of \cref{lem:gnp_reg} lies in proving that the number of rooted semi-copies of suitable clique-paths is appropriately concentrated, in order to deduce \ref{cond:creg3}.
The key step here is to use $ (b, c) $-centredness to bound a suitably defined density parameter, which allows us to deduce the required concentration by a standard application of the Kim-Vu polynomial concentration inequality.
Since the copies of clique-paths which we count are not necessarily injective, we must first take some extra care to show that non-injective copies contribute in a negligible way.

Before proceeding with the proof, we note that \cref{thm:proc_conv_fcd,lem:gnp_reg} are sufficient to prove \cref{thm:gnp_fcd}.

\begin{proof}[Proof of \cref{thm:gnp_fcd}]
    Given $ k, r, \eps $, choose $ 1/n \ll \delta, 1/C \ll 1/\ell \ll 1/c \ll 1/r, \eps $, and note that this is compatible with the hierarchies of both \cref{thm:proc_conv_fcd,lem:gnp_reg}.
    Suppose $ p \ge n^{-\frac{r - k}{\binom{r}{k} - 1} + \eps} $.
    Then w.h.p.\ $ G $ is $ (\ell, c, C, n^{-\delta}) $-clique-regular by \cref{lem:gnp_reg}, so there exists a fractional $ K_r $-decomposition of $ G $ by \cref{thm:proc_conv_fcd}.
\end{proof}

In order to prove \cref{lem:gnp_reg}, we first state the main concentration inequality we require in \cref{sect:kimvu_poly}, and immediately use it to prove some simple regularity properties in \cref{sect:conc_edge_clique}, from which in particular \ref{cond:creg1} and \ref{cond:creg2} follow.
In \cref{sect:cliquepath_conc,sect:clique_path_root_conc} we prove various technical statements about the concentration of rooted semi-copies of clique-paths in $ G^{(k)}(n, p) $, which we use to deduce \ref{cond:creg3} in \cref{sect:clique_reg_hyper}.

\subsection{Polynomial concentration} \label{sect:kimvu_poly}

We begin by stating a useful version of the Kim-Vu polynomial concentration inequality~\cite{KimVu2000}.
Let $ G $ be a $ k $-graph and write $ V_k \coloneqq \binom{V(G)}{k} $.
Given $ s \in \mathbb{N} $, say that $ w: \binom{V_k}{s} \to \mathbb{R}_{\ge 0} $ is an \emph{$ s $-weighting on $ G $}, that is, $ w $ assigns a weight to any $ s $-set of $ k $-sets in $ V(G) $.
Given $ \mathcal{T} \subseteq \binom{E(G)}{s} $ write $ w(\mathcal{T}) \coloneqq \sum_{S \in \mathcal{T}} w(S) $.
Write $ w(G) \coloneqq w(\binom{E(G)}{s}) $, and given $ A \subseteq V_k $, let $ E_A(G) \coloneqq \{ S \in \binom{V_k}{s}: A \subseteq S, S \setminus A \subseteq E(G) \} $, that is, the set of $ s $-sets $S$ of $ k $-sets of vertices containing all $ k $-sets in $ A $, and such that all other $ k $-sets in $ S $ are in fact edges of $ G $.
Write $ w_A(G) \coloneqq w(E_A(G)) $.

\begin{lemma}[Corollary 4.1.3 of \cite{KimVu2000}] \label{lem:kim_vu}
    Suppose $ 1/n \ll \eps \ll \delta, s $ and let $ p \in (0, 1) $.
    Suppose $ G \sim G^{(k)}(n, p) $ and $ w $ is an $ s $-weighting on $ G $ with $ \mathbb{E}[w_A(G)] \le n^{-\delta} \mathbb{E}[w(G)] $ for every $ \emptyset \ne A \subseteq V_k $.
    Then
    $$ \mathbb{P}[|w(G) - \mathbb{E}[w(G)]| \ge n^{-\eps} \mathbb{E}[w(G)]] \le e^{-n^{\eps}}. $$
\end{lemma}

This will be central to all bounds we prove for the random hypergraph; observe that we work with constant-size subhypergraphs and seek to bound random variables with polynomial expectation.
Note that given a multiset $ \mathcal{S} $ with $ \mathrm{Set}(\mathcal{S}) \subseteq \binom{V_k}{s} $, the multiplicity $ m_{\mathcal{S}} $ (as defined in \cref{section:notation}) is an $ s $-weighting on $ G $, which can be thought of as a generalised indicator function for $ \mathcal{S} $.

\subsection{Concentration of edges and cliques} \label{sect:conc_edge_clique}

We start by proving some simple concentration results, from which in particular \ref{cond:creg1} and \ref{cond:creg2} follow.
Given $ n \in \mathbb{N} $, $ p \in (0, 1) $, and $ G \sim G^{(k)}(n, p) $, define
$$ \hat{d}_j \coloneqq \hat{d}_j(n, p) \coloneqq p \binom{n - j}{k - j} \quad \text{and} \quad \hat{\gamma}_s \coloneqq \hat{\gamma}_s(n, p) \coloneqq p^{\binom{r}{k} - \binom{s}{k}} \binom{n - s}{r - s} $$
for every $ j \in [0, k - 1]$ and $ s \in [0, k + 1] $, as well as $ \hat{d}_k \coloneqq 1 $.
Note that $ \hat{d}_j $ and $ \hat{\gamma}_j $ represent the expected number of edges and cliques containing a given $ j $-set, respectively.
The first property follows from a standard Chernoff bound (or a straightforward application of \cref{lem:kim_vu}); we remark that the extra $ n^{\eps} $ factor in the probability $ p $ is not needed here.

\begin{fact} \label{fact:creg1}
    Suppose $ 1/n \ll \delta \ll 1/k $, let $ p \ge n^{-\frac{r - k}{\binom{r}{k} - 1}} $, and suppose $ G \sim G^{(k)}(n, p) $.
    Then the following holds with probability at least $ 1 - e^{-n^{\delta}} $.
    For each $ j \in [0, k - 1] $ and $ S \in E_j $ we have $ |N_G^{\mathrm{e}}(S)| = (1 \pm n^{-2\delta}) \hat{d}_j $.
    In particular, $ G $ satisfies \ref{cond:creg1}, with $ n^{-\delta} $ playing the role of $ \eps $.
\end{fact}

% \begin{proof}
%     Observe that if $ j = k $ the statement is trivial, so consider $ j \in [k - 1] $.
%     By a standard Chernoff bound (or \cref{lem:kim_vu}), we obtain $ |E(G)| = (1 \pm n^{-2 \delta}) \binom{n}{k} $, with probability at least $ 1 - e^{-n^{2\delta}} $.
%     Let $ \mathcal{S} $ be the set of edges in $ K_n^{(k)} $ containing $ S $, then $ |\mathcal{S}| = \binom{n - j}{k - j} $.
%     Therefore, by a similar Chernoff bound, we obtain that
%     $$ |\mathcal{S} \cap E(G)| = (1 \pm n^{-2 \delta}) p \binom{n - j}{k - j} = (1 \pm n^{-2\delta}) p \binom{n}{k} \frac{\binom{k}{j}}{\binom{n}{j}} = (1 \pm n^{-\delta}) \frac{|E(G)| \binom{k}{j}}{|E_j|}, $$
%     with probability at least $ 1 - e^{-n^{2 \delta}} $.
%     We conclude by taking a union bound over the at most $ n^j $ possible $ S $.
% \end{proof}

The second property is a stronger version of \ref{cond:creg2}, and also follows from a standard application of \cref{lem:kim_vu}.
We remark that the extra $ n^{\eps} $ factor in the probability is only required for the case $ j = k $ and the additional statement for $ |S| \ge k + 1 $.

\begin{fact} \label{fact:creg2}
    Suppose $ 1/n \ll \delta \ll \eps, 1/r $, let $ p \ge n^{-\frac{r - k}{\binom{r}{k} - 1} + \eps} $, and suppose $ G \sim G^{(k)}(n, p) $.
    Then the following hold with probability at least $ 1 - e^{-n^{\delta}} $.
    \begin{enumerate}[\rm (\roman*)]
        \item $ |K_r(S)| = (1 \pm n^{-2\delta}) \hat{\gamma}_j $ for all $ j \in [0, k] $ and $ S \in E_j $; \label{cond:fact_creg2_1}
        \item $ |K_r(S)| \le n^{-\delta} \hat{\gamma}_k $ for any set $ S \subseteq V(G) $ with $ |S| \ge k + 1 $. \label{cond:fact_creg2_2}
    \end{enumerate}
    In particular, $ G $ satisfies \ref{cond:creg2}, with $ n^{-\delta} $ playing the role of $ \eps $.
\end{fact}

The proof of \ref{cond:fact_creg2_1} is standard, so we provide only a proof of \ref{cond:fact_creg2_2}.

\begin{proof}[Proof of \ref{cond:fact_creg2_2}]
    Firstly note that it suffices to prove \ref{cond:fact_creg2_2} for sets $ S $ of size exactly $ k + 1 $.
    Note also that if $ r = k + 1 $, then $ |K_{r, G}(S)| \le 1 $ deterministically for any set $ S \in \binom{V(G)}{k + 1} $, which suffices; we assume henceforth that $ r \ge k + 2 $.
    Suppose $ \delta \ll \eta \ll \eps $ and write $ p' \coloneqq n^{-\frac{r - k-1}{\binom{r}{k} - \binom{k+1}{k}} + \eta} $.
    Given a $k$-graph $ H $ and a set $ S \in \binom{V(H)}{k + 1} $, write $ K'_H(S) $ for the number of copies of $ K_r^{(k)} $ containing $ S $ in the $k$-graph with edge set $ E(H) \cup \binom{S}{k} $; whenever $ |K_{r, H}(S)| > 0 $, we must have $ \binom{S}{k} \subseteq E(H) $ and thus $ K'_H(S) = |K_{r, H}(S)| $.
    If $ p \ge p' $, then \ref{cond:fact_creg2_2} follows by a standard application of \cref{lem:kim_vu}, noting that $ n^{\eta} \le \mathbb{E}[K'_G(S)] \le n^{-2 \delta} \hat{\gamma}_k(n, p) $ in this case.
    If instead $ p \le p' $, note that $ \hat{\gamma}_k(n, p) \ge n^{\eps} $, which means that $ n^{-\delta} \hat{\gamma}_k(n, p) \ge n^{\eps - \delta} \ge n^{\eps / 2} $.
    Let $ G' \sim G^{(k)}(n, p') $ and observe, by a standard coupling argument, that $ \mathbb{P}[G \in \mathcal{A}] \le \mathbb{P}[G' \in \mathcal{A}] $ for any increasing event $ \mathcal{A} $.
    In particular, we obtain 
    $$ \mathbb{P}[K'_G(S) \ge n^{-\delta} \hat{\gamma}_k] \le \mathbb{P}[K'_G(S) \ge n^{\eps / 2}] \le \mathbb{P}[K'_{G'}(S) \ge n^{\eps / 2}]. $$
    Noting that $ n^{\eta} \le \mathbb{E}[K'_{G'}(S)] \le n^{\eps / 3} $, it follows from a standard application of \cref{lem:kim_vu} that $ \mathbb{P}[K'_{G'}(S) \ge n^{\eps / 2}] \le e^{-n^{2 \delta}} $, which suffices.
\end{proof}

\subsection{Counting clique-paths} \label{sect:cliquepath_conc}

We now introduce some definitions and intermediate results, which will help in the proof of~\ref{cond:creg3}.
Let $ G $ be a $ k $-graph on vertex set $ [n] $, with the canonical ordering, and let $ j \in [k] $.
In this section we consider semi-ordered $ j $-clique-paths $ \mathcal{F} $, recalling that $ \mathcal{F} $ is an equivalence class of clique-paths under semi-isomorphism, consisting of all possible clique-paths with a given underlying hypergraph $ F $ and root sets $ \mathbf{S} $, but any possible orderings on the sets $ S_1, \ldots, S_{\ell} $.
In particular, recall that we identify $ \mathcal{F} $ with the pair $ (F, \overline{\mathbf{S}}) $, consisting of the hypergraph $ F $ and sets $ (\overline{S}_i)_{i \in [0, \ell]} $ for which only $ \overline{S}_0 $ is equipped with an ordering.
Given $ S \in E_j $ and $ J \subseteq [j] $, let $ \mathcal{T}_J(S) $ be the set of $ T \in E_j $ with $ \iota(S \cap T, S) = J $ and $ \hat{\mathcal{T}}_J(S) $ be those with $ J \subseteq \iota(S \cap T, S) $; we immediately make the following observation.

\begin{fact} \label{fact:sets_intersect_conc}
    Suppose $ 1/n \ll \delta \ll 1/k $, let $ p \ge n^{-\frac{r - k}{\binom{r}{k} - 1}} $, and suppose $ G \sim G^{(k)}(n, p) $.
    Then the following holds with probability at least $ 1 - e^{-n^{\delta}} $.
    For each $ j \in [k] $ and $ J \subseteq [j] $, writing $ t_J \coloneqq \binom{k - |J|}{j - |J|} \frac{\hat{d}_{|J|}}{\hat{d}_j} $, we have $ |\mathcal{T}_J(S)| = (1 \pm n^{-\delta}) t_J $ and $ |\hat{\mathcal{T}}_J(S)| = (1 \pm n^{-\delta}) t_J $ for every $ S \in E_j $.
\end{fact}

We remark that the statement only depends on the structure of $ G $ in the case $ j = k $.

\begin{proof}
    It clearly suffices to prove the statement for a (fixed) $ k $-graph $ G $ satisfying the conclusion of \cref{fact:creg1}, with $ 2 \delta $ playing the role of $ \delta $.
    Given $ S \in E_j $, there exists a unique set $ U \subseteq S $ with $ \iota(U, S) = J $.
    By \cref{fact:creg1}, the $k$-graph $ G $ satisfies \ref{cond:creg1} with $ n^{-2\delta} $ playing the role of $ \eps $, so by \cref{fact:uniformity_disc}~\ref{cond:unif_disc_1}, we see that $ |E_j(U)| = (1 \pm n^{-2\delta}) \binom{j}{|J|} \frac{|E_j|}{|E_{|J|}|} $.
    By definition, we have $ |\hat{\mathcal{T}}_J(S)| = |E_j(U)| $, so take $ t_J \coloneqq \binom{k - |J|}{j - |J|} \frac{\hat{d}_{|J|}}{\hat{d}_j} = (1 \pm n^{-2 \delta}) \binom{j}{|J|} \frac{|E_j|}{|E_{|J|}|} $, using \ref{cond:creg1} and recalling \eqref{eqn:gamma_d_defns}.
    Similarly, given $ U \subsetneq U' \subseteq S$, we have $ |E_j(U')| \le 2^{j + 1} \frac{|E_j|}{|E_{|J| + 1}|} \le n^{-2 \delta} t_J $.
    In particular,
    $$ |\mathcal{T}_J(S)| = |E_j(U)| - \left| \bigcup_{U \subsetneq U' \subseteq S} E_j(U') \right| = (1 \pm 3 n^{-2\delta}) t_J - 2^j n^{-2 \delta} t_J, $$
    which suffices to complete the proof.
\end{proof}

Given $ S, T \in E_j $, define a \emph{copy of $ \mathcal{F} $ in $ G $ rooted at $ S $ and $ T $} to be an \emph{injective} semi-copy~$ \Phi $ of any $ \mathbf{F} \in \mathcal{F} $ in $ G $ rooted at $ S $ and $ T $.
Note that copies of $ \mathcal{F} $ correspond to embeddings of $ F $ in $ G $ in the usual sense, in which the image of the first root set $ \overline{S}_0 $ has a prescribed ordering, and vertices are treated as unlabelled except for those in the root sets.
Write $ Z_{\mathcal{F}}(S, T) $ for the number of copies of $ \mathcal{F} $ in~$ G $ rooted at $ S $ and $ T $, and set $ Z_{\mathcal{F}}(S) \coloneqq \sum_{T \in E_j} Z_{\mathcal{F}}(S, T) $.
Note that $ Z_{\mathcal{F}}(S, T) = 0 $ whenever $ T \not \in \mathcal{T}_{\iota(\mathcal{F})}(S) $.
Given $ p \in (0, 1) $, $ G \sim G^{(k)}(n, p) $, and $ S, T \in \binom{V(G)}{j} $ with $ \iota(S \cap T, S) = \iota(\mathcal{F}) $, we define $ \mu_{\mathcal{F}}(n, p) \coloneqq \mathbb{E}[Z_{\mathcal{F}}(S, T) \mid S, T \in E_j] $; note that this is independent of the choice of $ S \in E_j $ and $ T \in \mathcal{T}_{\iota(\mathcal{F})}(S) $, and that the conditioning ensures that $ S, T \in E(G) $ in the case $ j = k $.

Given $ S \in E_j $, write
$$ Y_{\mathcal{F}}(S, T) \coloneqq \sum_{\mathbf{F} \in \mathcal{F}} |X_G(S, T, \mathbf{F})|, \quad \text{and set} \quad Y_{\mathcal{F}}(S) \coloneqq \sum_{T \in E_j} Y_{\mathcal{F}}(S, T) $$
to be the total number of semi-copies of any $ \mathbf{F} \in \mathcal{F} $ in $ G $ rooted at $ S $ and any $ T \in E_j $.
Recall that $ s_i = s_i(F) = |V(F_i \cap F_{i + 1})| $ for $ i \in [\ell] $ and write $ s_0 \coloneqq j $.
We first prove that, considering clique-paths rooted (only) at the first root set, most semi-copies (as counted by $ Y_{\mathcal{F}}(S) $) are in fact injective (as counted by $ Z_{\mathcal{F}}(S) $).

\begin{lemma} \label{fact:clique_path_noninj_trivial}
    Suppose $ 1/n \ll \delta, 1/C \ll 1 / M \ll \eps, 1/r $, let $ p \ge n^{-\frac{r - k}{\binom{r}{k} - 1} + \eps} $, and suppose $ G \sim G^{(k)}(n, p) $.
    Then the following holds with probability at least $ 1 - e^{-n^{\delta}} $.
    For all $ j \in [k] $, semi-ordered $ j $-clique-paths $ \mathcal{F} $ of length $ m \le M $, and $ S \in E_j $, we have
    $$ Y_{\mathcal{F}}(S) = (1 \pm n^{-\delta}) t_{\iota(\mathcal{F})} \mu_{\mathcal{F}}(n, p), $$
    as well as
    $$ \mu_{\mathcal{F}}(n, p) \le C \frac{\hat{d}_{s_{m}}}{\hat{d}_{|\iota(\mathcal{F})|}} \prod_{i = 0}^{m - 1} \hat{\gamma}_{s_i}. $$
\end{lemma}

\begin{proof}
    Suppose $ 1/C \ll 1/C' \ll 1/M $.
    A semi-copy of some $ \mathbf{F} = ((F_i)_{i \in [m]}, \mathbf{S}) \in \mathcal{F} $ in $ G $ rooted at $ S $ and some $ T \in E_j $ corresponds to a (not necessarily injective) homomorphism $ \Phi $ from $ F $ to $ G $ mapping $ S_0 $ to $ S $ in the unique order-preserving way, counted up to clique-path automorphisms of $ \mathbf{F} $.
    A copy of $ \mathcal{F} $ corresponds to an injective such homomorphism $ \Phi $.
    Let $ \mathcal{G} $ be the set of $k$-graphs on $ [n] $ satisfying the conclusions of \cref{fact:creg1,fact:creg2,fact:sets_intersect_conc} with $ 3 \delta $ playing the role of $ \delta $.

    \begin{claim} \label{claim:noninjtriv_1}
        There exist constants $ 1 \le C_{\mathcal{F}, i} \le C' $ for each $ i \in [0, m] $, depending only on $ \mathcal{F} $, such that, for any (fixed) $ k $-graph $ G \in \mathcal{G} $, we have
        $$ Z_{\mathcal{F}}(S) = (1 \pm n^{-2 \delta}) E, \quad \text{where} \quad E \coloneqq C_{\mathcal{F}, m} \frac{\hat{d}_{s_{m}}}{\hat{d}_j} \prod_{i = 0}^{m - 1} C_{\mathcal{F}, i} \hat{\gamma}_{s_i} $$
        for every $ S \in E_j $.
    \end{claim}

    \begin{claimproof}
        Fix $ S \in E_j $ and, to estimate $ Z_{\mathcal{F}}(S) $, consider iteratively constructing an injective homomorphism by choosing the images of the cliques $ F_i $ one at a time.
        Given the image of all vertices in $ F_0 \cup \cdots \cup F_{i} $, there are $ (1 \pm n^{-3 \delta}) C_{\mathcal{F}, i} \hat{\gamma}_{s_i} $ choices for the images of the remaining vertices in $ F_{i + 1} $, for a suitable constant $ 1 \le C_{\mathcal{F}, i} \le C' $, for each $ i \in [0, m - 1] $, by \cref{fact:creg2}~\ref{cond:fact_creg2_1}.
        If $ s_i \le k - 1 $, then by \cref{fact:creg2}~\ref{cond:fact_creg2_1} at most $ 2jrm \cdot C_{\mathcal{F}, i} \cdot \hat{\gamma}_{s_i + 1} \le n^{-3 \delta} \hat{\gamma}_{s_i} $ such choices involve mapping any vertex of $ V(F_{i+1} \setminus F_i) $ to the image of a vertex in $ V(F_0 \cup \cdots \cup F_i) $.
        If instead $ s_i = k $, then we similarly get at most $ j rm \cdot n^{-3\delta} \hat{\gamma}_k $ by \cref{fact:creg2}~\ref{cond:fact_creg2_2}.
        This leaves only the images of $ S_{m} \setminus V(F_{m}) $ to choose, for which there are $ (1 \pm n^{-3 \delta}) C_{\mathcal{F}, m} \frac{\hat{d}_{s_{m}}}{\hat{d}_j} $ choices, for a suitable constant $ 1 \le C_{\mathcal{F}, m} \le C'$, by \cref{fact:uniformity_disc}~\ref{cond:unif_disc_1}.
        If $ s_m < j $, then by \cref{fact:uniformity_disc}~\ref{cond:unif_disc_1} at most $ j rm \cdot C_{\mathcal{F}, m} \cdot \frac{\hat{d}_{s_{m} + 1}}{\hat{d}_j} \le n^{-3 \delta} \frac{\hat{d}_{s_{m}}}{\hat{d}_j} $ such choices involve vertices from the image of $ V(F_0 \cup \cdots \cup F_{m}) $, and if $ s_m = j $ then trivially there are none.
        Hence the desired bound follows by taking the product of our estimates.
    \end{claimproof}

    Fix $ S \in \binom{V(G)}{j} $.
    Observe, since $ Z_{\mathcal{F}}(S) \le n^{Cm} $ deterministically and $ \mathbb{P}[S \in E_j] \ge p \ge n^{-1} $, using \cref{fact:creg1,fact:creg2,fact:sets_intersect_conc,claim:noninjtriv_1}, that
    \begin{align*}
        \mathbb{E}[Z_{\mathcal{F}}(S) \mid S \in E_j] &= \mathbb{E}[Z_{\mathcal{F}}(S) \mathds{1}[G \in \mathcal{G}] \mid S \in E_j] \pm n^{Cm + 1} \mathbb{P}[G \not \in \mathcal{G}]\\
            &= (1 \pm n^{-2 \delta}) E \pm n^{Cm + 1} \cdot 3 e^{-n^{\delta}}
            = (1 \pm 2 n^{-2 \delta}) E,
    \end{align*}
    where we use that clearly $ E \ge 1 $ in the final equality.
    On the other hand, note by definition and \cref{fact:sets_intersect_conc} that
    \begin{align*}
        \mathbb{E}[Z_{\mathcal{F}}(S) \mid S \in E_j] &= \mathbb{E} \left[ \sum_{T \in \mathcal{T}_{\iota(\mathcal{F})}(S)} Z_{\mathcal{F}}(S, T) \ \bigg| \ S \in E_j \right]\\
        &= \sum_{T \in \binom{V(G)}{j}} \mathbb{P}[T \in \mathcal{T}_{\iota(\mathcal{F})}(S) \mid S \in E_j] \mathbb{E}[Z_{\mathcal{F}}(S, T) \mid S, T \in E_j]\\
        &= \mathbb{E}[|\mathcal{T}_{\iota(\mathcal{F})}(S)| \mid S \in E_j] \mu_{\mathcal{F}}(n, p)
        = (1 \pm n^{-2\delta}) t_{\iota(\mathcal{F})} \mu_{\mathcal{F}}(n, p).
    \end{align*}
    In particular
    \begin{equation} \label{eqn:noninjtriv_1}
        E = (1 \pm 5n^{-2\delta}) t_{\iota(\mathcal{F})} \mu_{\mathcal{F}}(n, p).
    \end{equation}
    By \eqref{eqn:noninjtriv_1}, the definition of $ t_{\iota(\mathcal{F})} $ in \cref{fact:sets_intersect_conc}, and the upper bound on $ C_{\mathcal{F}, i} $ in \cref{claim:noninjtriv_1}, it follows that
    $$ \mu_{\mathcal{F}}(n, p) \le \frac{2E}{t_{\iota(\mathcal{F})}}
        \le \frac{2 \hat{d}_j}{\hat{d}_{|\iota(\mathcal{F})|}} \cdot C_{\mathcal{F}, m} \frac{\hat{d}_{s_{m}}}{\hat{d}_j} \prod_{i = 0}^{m - 1} C_{\mathcal{F}, i} \hat{\gamma}_{s_i}
        \le C \frac{\hat{d}_{s_{m}}}{\hat{d}_{|\iota(\mathcal{F})|}} \prod_{i = 0}^{m - 1} \hat{\gamma}_{s_i}, $$
    as required for the second statement of the lemma.

    \begin{claim} \label{claim:noninjtriv_2}
        Let $ G $ be a (fixed) $k$-graph in $ \mathcal{G} $.
        Then
        $$ Y_{\mathcal{F}}(S) - Z_{\mathcal{F}}(S) \le n^{-2\delta} E $$
        for every $ S \in E_j $.
    \end{claim}

    \begin{claimproof}
        Fix $ S \in E_j $ and note that $ Y_{\mathcal{F}}(S) - Z_{\mathcal{F}}(S) $ is exactly the number of non-injective homomorphisms $ \Phi $.
        We use the same iterative construction as for \cref{claim:noninjtriv_1}, noting that for any non-injective $ \Phi $, there exist $ i \in [0, m] $ and $ v \in V(F_{i + 1} \setminus (F_0 \cup \cdots \cup F_i)) $ with $ \Phi(v) \in \Phi(V(F_0 \cup \cdots \cup F_i)) $.
        If $ i \in [0, m - 1] $, this means the total number of choices for the images of the remaining vertices in $ F_{i + 1} $ is at most  $ C \hat{\gamma}_s \le C n^{-3\delta} \hat{\gamma}_{s_i} $ for some $ s_i < s \le r $ by \cref{fact:creg2}, writing $ \hat{\gamma}_s \coloneqq n^{-3\delta} \hat{\gamma}_k $ for any $ s \ge k + 1 $, instead of $ (1 \pm n^{-3 \delta}) C_{\mathcal{F}, i} \hat{\gamma}_{s_i} $.
        In the case $ i = m $, this means the total number of choices for the images of $ S_m \setminus V(F_m) $ is at most $ C \frac{\hat{d}_s}{\hat{d}_j} \le C n^{-3 \delta} \frac{\hat{d}_{s_{m}}}{\hat{d}_j} $ for some $ s_m < s \le j $ by \cref{fact:creg1} and the same argument as before, instead of $ (1 \pm n^{-3 \delta}) C_{\mathcal{F}, m} \frac{\hat{d}_{s_{m}}}{\hat{d}_j} $.
        Note also that, for each $ i' \ne i $, the number of choices for the images of $ V(F_{i' + 1} \setminus (F_0 \cup \cdots \cup F_{i'})) $ remains at most $ C \hat{\gamma}_{s_{i'}} $ in the case $ i' \ne m $ by \cref{fact:creg2}, and $ C \frac{\hat{d}_{s_{m}}}{\hat{d}_j} $ in the case $ i' = m$ by \cref{fact:creg1}.
        Summing over all possible choices of $ i $ and $ v $, the desired bound follows.
    \end{claimproof}

    The first statement now follows from \cref{claim:noninjtriv_1,claim:noninjtriv_2}, as well as \eqref{eqn:noninjtriv_1}, completing the proof.
\end{proof}

\Cref{fact:clique_path_noninj_trivial} lets us restrict to injective homomorphisms, for which we now prove concentration.

\subsection{Concentration of rooted copies of clique-paths} \label{sect:clique_path_root_conc}

In this section, we prove concentration of the number of (injective) copies of $ (b, c) $-centred clique-paths rooted at both ends.
We will require the following observations, which are easy to check.

\begin{fact} \label{fact:binom_est}
    Let $ x,y,r,k $ be integers with $ 0 < k < r $, $ 0 < x \le r $, and $ 0 \le y \le \min\{ k, x - 1 \} $.
    Then
    $$ \frac{\binom{x}{k} - \mathds{1}[y = k]}{x-y} \le \frac{\binom{r}{k} - 1}{r-k}. $$
\end{fact}

\begin{fact} \label{fact:fraction_sum}
    Let $ n \in \mathbb{N} $ and $ c, x_1, \ldots, x_n, y_1, \ldots, y_n > 0 $, and suppose that $ \frac{x_i}{y_i} \le c $ for every $ i \in [n] $.
    Then $ \frac{\sum_{i = 1}^n x_i}{\sum_{i = 1}^n y_i} \le c $.
\end{fact}

% \begin{proof}
%     We have $ x_i \le c y_i $ for all $ i \in [n] $, so summing these $ n $ equations yields $ \sum_{i = 1}^{n} x_i \le c \sum_{i = 1}^{n} y_i $.
%     We now divide both sides by $ \sum_{i = 1}^n y_i $ to obtain the result.
% \end{proof}

We also need a definition of the maximum average density of a rooted hypergraph.
Let $ F $ be a hypergraph rooted at some set $ R \subsetneq V(F) $.
Given $ R \subsetneq B \subseteq V(F) $, define $ d(F, B) \coloneqq \frac{e(F[B]) - e(F[R])}{|B| - |R|} $ and write
$$ d^*(F) \coloneqq \max_{R \subsetneq B \subseteq V(F)} d(F, B). $$
Our concentration result revolves around the following deterministic density estimate.
We remark that the argument here is partially inspired by arguments used in the analysis of the hypergraph removal process by Joos~and~Kühn~\cite{jooshypergraphremoval}, originally based on ideas of Bohman, Frieze, and Lubetzky~\cite{bohmantriangleremoval}.

\begin{lemma} \label{lem:clique_path_density}
    Suppose $ 1/c \ll \eta \ll 1/r $ and let $ 0 \le b < j \le k $.
    Let $ \mathcal{F} = (F, \overline{\mathbf{S}}) $ be a $ (b, c)$-centred semi-ordered $ j $-clique-path of length $ m \in \mathbb{N} $ and regard $ F $ as rooted at $ R \coloneqq \overline{S}_0 \cup \overline{S}_{m} $.
    Then
    $$ d^*(F) \le \frac{\binom{r}{k} - 1}{r - k} + \eta. $$
\end{lemma}

\begin{proof}
    Fix $ R \subsetneq B \subseteq V(F) $ and consider two cases.
    Assume first that there is no $ (\overline{S}_0 \cap \overline{S}_{m}) $-avoiding walk between $ \overline{S}_0 \setminus \overline{S}_{m} $ and $ \overline{S}_{m} \setminus \overline{S}_0 $ in $ F[B] $, which in particular means we may write $ B = B_1 \cup B_2 $ for sets $ B_1, B_2 $ with $ \overline{S}_0 \setminus \overline{S}_{m} \subseteq B_1 $, $ \overline{S}_{m} \setminus \overline{S}_0 \subseteq B_2 $, and $ B_1 \cap B_2 = \overline{S}_0 \cap \overline{S}_{m} $, such that every $ e \in E(F[B]) $ has either $ e \subseteq B_1 $ or $ e \subseteq B_2 $.
    Our aim is to count the vertices and edges of $ V(F_i) \cap B_1 $ one at a time for each $ i = 1, \ldots, m $, and those of $ V(F_i) \cap B_2 $ in reverse order, for $ i = m, \ldots, 1 $.
    Since the intersection $ B_1 \cap B_2 = \overline{S}_0 \cap \overline{S}_{m} $ does not contribute any vertices or edges to the density, we avoid double counting.
    
    For each $ i \in [m] $ and $ q \in [2] $, we have $ r_i^{(q)} \coloneqq |V(F_i) \cap B_q| \le r $, as well as $ s_i^{(1)} \coloneqq |V(F_i) \cap V(F_{i - 1}) \cap B_1| \le \min \{ k, r_i^{(1)} \} $ and $ s_i^{(2)} \coloneqq |V(F_i) \cap V(F_{i + 1})\cap B_2| \le \min\{ k, r_i^{(2)} \} $.
    Using the facts $ V(F_i \cap (F_0 \cup \cdots \cup F_{i - 1})) \subseteq V(F_i \cap F_{i - 1}) $ and $ V(F_i \cap (F_{i + 1} \cup \cdots \cup F_{m+1})) \subseteq V(F_i \cap F_{i + 1}) $, we may then estimate
    $$ |B| - |R| = |B_1 \setminus \overline{S}_0| + |B_2 \setminus \overline{S}_{m}| = \sum_{i = 1}^{m} r_i^{(1)} - s_i^{(1)} + \sum_{i = 1}^{m} r_{m + 1 - i}^{(2)} - s_{m + 1 - i}^{(2)}. $$
    Similarly, since every edge of $ F $ is contained in $ B_1 $ or $ B_2 $, we have
    \begin{align*}
        e(F[B]) - e(F[R]) &= e(F[B_1]) - e(F[\overline{S}_0]) + e(F[B_2]) - e(F[\overline{S}_{m}])\\
            &= \sum_{i = 1}^m \binom{r_i^{(1)}}{k} - \mathds{1}[s_i^{(1)} = k] + \binom{r_i^{(2)}}{k} - \mathds{1}[s_i^{(2)} = k].
    \end{align*}
    In this case, using \cref{fact:fraction_sum}, it is therefore sufficient to prove that
    $$ \frac{\binom{r_i^{(q)}}{k} - \mathds{1}[s_i^{(q)} = k]}{r_i^{(q)} - s_i^{(q)}} \le \frac{\binom{r}{k} - 1}{r - k} $$
    for every $ i \in [m] $ with $ r_i^{(q)} \ne s_i^{(q)} $.
    Indeed, this follows directly from \cref{fact:binom_est}, using the facts that $ 0 \le r_i^{(q)} \le r $ and $ 0 \le s_i^{(q)} \le \min \{ k, r_i^{(q)} \} $.

    For the second case, assume that there is an $ (\overline{S}_0 \cap \overline{S}_{m}) $-avoiding walk between $ \overline{S}_0 \setminus \overline{S}_{m} $ and $ \overline{S}_{m} \setminus \overline{S}_0 $ in $ F[B] $.
    By taking a minimal such walk, we may assume without loss of generality that all vertices are distinct.
    Since $ F $ is $ (b, c) $-centred, this walk must have length at least $ c $, so we may assume that $ |B \setminus R| \ge c - 2k \ge \frac{c}{2} $.
    We use a similar argument to the first case, but define instead $ r_i \coloneqq |V(F_i) \cap B| \le r $ and $ s'_i \coloneqq |V(F_i \cap F_{i - 1}) \cap B| \le \min \{ k, r_i \} $ for $ i \in [m + 1] $.
    We then similarly obtain
    $$ \frac{c}{2} \le |B| - |R| = \left( \sum_{i = 1}^{m} r_i - s'_i \right) - (s'_{m + 1} - b) $$
    and
    $$ e(F[B]) - e(F[R]) = \left( \sum_{i = 1}^{m} \binom{r_i}{k} - \mathds{1}[s'_i = k] \right) - \mathds{1}[s'_{m + 1} = k]. $$
    Observe that
    $$ d(F, B) \le \frac{\sum_{i = 1}^m \binom{r_i}{k} - \mathds{1}[s'_i = k]}{\sum_{i = 1}^{m} r_i - s'_i} \beta, \quad \text{for} \quad \beta \coloneqq \frac{\sum_{i = 1}^{m} r_i - s'_i}{\left( \sum_{i = 1}^{m} r_i - s'_i \right) - (s'_{m + 1} - b)} \le 1 + \frac{2k}{c} \le 1 + \eta^2. $$
    By \cref{fact:binom_est}, we obtain
    $$ \frac{\binom{r_i}{k} - \mathds{1}[s'_i = k]}{r_i - s'_i} \le \frac{\binom{r}{k} - 1}{r - k} $$
    for every $ i \in [m] $ with $ r_i \ne s'_i $, so it follows, using again \cref{fact:fraction_sum}, that
    $$ d(F, B) \le (1 + \eta^2) \frac{\binom{r}{k} - 1}{r - k} \le \frac{\binom{r}{k} - 1}{r - k} + \eta, $$
    as required.
\end{proof}

The proof of concentration of $ Z_{\mathcal{F}}(S, T) $ is now a standard corollary of \cref{lem:kim_vu}.

\begin{lemma} \label{lem:counting_extensions}
    Suppose $ 1/n \ll \delta, 1/c, 1/M \ll \eps, 1/r $, let $ 0 \le b < j \le k $ and $ p \ge n^{-\frac{r - k}{\binom{r}{k} - 1} + \eps} $, and suppose $ G \sim G^{(k)}(n, p) $.
    Then the following holds with probability at least $ 1 - e^{-n^{\delta}} $.
    Let $ \mathcal{F} = (F, \overline{\mathbf{S}}) $ be a $ (b, c)$-centred semi-ordered $ j $-clique-path of length $ m \le M $.
    For all $ S \in E_j $ and $ T \in \mathcal{T}_{\iota(\mathcal{F})}(S) $, we have
    $$ Z_{\mathcal{F}}(S, T) = (1 \pm n^{-\delta}) \mu_{\mathcal{F}}(n, p). $$
\end{lemma}

\begin{proof}
    Introduce new constants $ C, \eta $ satisfying $ 1/n \ll 1/C \ll \delta, 1/c, 1/M \ll \eta \ll \eps, 1/r $.
    Fix $ S, T \in \binom{V(G)}{j} $ with $ \iota(S \cap T, S) = \iota(\mathcal{F}) $ and aim to apply \cref{lem:kim_vu} to the $ s $-weighting given by the multiplicity $ m_{\mathcal{S}} $, where $ \mathcal{S} $ is the collection of all edge sets $ E(\Phi(F)) \setminus E(\Phi(F)[S \cup T]) $ of any copy $ \Phi $ of $ \mathcal{F} $ in $ K_n^{(k)} $ rooted at $ S $ and $ T $; any edge sets associated to multiple copies are counted with multiplicity.
    Note that, in the case $ S, T \in E_j $ (which is only a non-trivial condition if $ j = k$), clearly
    $ Z_{\mathcal{F}}(S, T) = w(G) $
    so, by the definition of $ \mu_{\mathcal{F}} $, we have
    $$ \mathbb{E}[w(G)] = \mathbb{E}[w(G) \mid S, T \in E_j] = \mathbb{E}[Z_{\mathcal{F}}(S, T) \mid S, T \in E_j] = \mu_{\mathcal{F}}(n, p). $$
    We now fix $ \emptyset \ne A \subseteq E(K_n^{(k)}) $ and aim to prove that
    \begin{equation} \label{eqn:extensions_2}
        \mathbb{E}[w_A(G)] \le n^{-\eta} \mu_{\mathcal{F}}(n, p).
    \end{equation}

    Fix an injective function $ \Psi: \bigcup A \to V(F) $, let $ A' \coloneqq \text{Image}(\Psi) $, write $ R \coloneqq \overline{S}_0 \cup \overline{S}_{m} $, and let $ B \coloneqq R \cup A' $.
    To bound $ \mathbb{E}[w_A(G)] $, we seek to bound the expected number of copies of $ \mathcal{F} $ in $ G $ with $ B $ mapped to $ S \cup T \cup \bigcup A $; this will suffice, since there are at most $ C $ choices for $ \Psi $.
    In particular, we have
    $$ \mu_{\mathcal{F}}(n, p) \ge \frac{1}{C} p^{e(F) - e(F[R])} n^{v(F) - |R|} \quad \text{and} \quad \mathbb{E}[w_A(G)] \le C p^{e(F) - e(F[B])} n^{v(F) - |B|}, $$
    since $ |A| \le e(F[B]) - e(F[R]) $.
    As such, it is enough to prove that
    \begin{equation} \label{eqn:extensions_0}
        p^{e(F) - e(F[B])} n^{v(F) - |B|} \le n^{-2 \eta} p^{e(F) - e(F[R])} n^{v(F) - |R|}.
    \end{equation}
    
    Write $ v \coloneqq |B| - |R| $ and $ d \coloneqq d(F, B) $, and note that $ dv = e(F[B]) - e(F[R]) \ge 1 $ since $ A \ne \emptyset $.
    Write also $ \alpha \coloneqq \frac{\binom{r}{k} - 1}{r - k} \ge 1 $ and note that $ p \ge n^{-\frac{1}{\alpha} + \eps} $.
    By \cref{lem:clique_path_density}, we see that
    $$ \frac{1}{d} \ge \frac{1}{\alpha + \eta} \ge \frac{1}{\alpha} - 2 \eta, \quad \text{so} \quad \frac{1}{d} - \frac{1}{\alpha} + \eps \ge 2 \eta. $$
    It follows that
    $$ p^{e(F[B]) - e(F[R])} n^{|B| - |R|} = p^{dv} n^v \ge n^{dv (\frac{1}{d} -\frac{1}{\alpha} + \eps)} \ge n^{2 \eta dv} \ge n^{2 \eta}, $$
    which rearranges to give \eqref{eqn:extensions_0}.
    This is sufficient for \eqref{eqn:extensions_2}, so the result follows immediately by \cref{lem:kim_vu} and a union bound over all $ \mathcal{F}, S, T $.
\end{proof}

We now have all of the key results we need to deduce our pseudorandomness properties.

\subsection{Proof of clique regularity} \label{sect:clique_reg_hyper}

In this section, we complete the proof of \cref{lem:gnp_reg}.
We have already shown \ref{cond:creg1} and \ref{cond:creg2} in \cref{fact:creg1,fact:creg2}, respectively, so we aim to prove \ref{cond:creg3} using \cref{lem:counting_extensions}.
We remark that \ref{cond:creg3} makes full use of the extra $ n^{\eps} $ factor in $ p $.
While it is conceivable that some variant of \ref{cond:creg3}, in which the lengths of the relevant clique-paths are allowed to grow with $ n $, could be true at (or some constant factor above) the conjectured threshold, our proof of \cref{claim:creg3} relies heavily on $ \ell $ being constant, as well as the strong concentration of the number of cliques containing each edge given by \ref{cond:creg2}.

\begin{lemma} \label{claim:creg3}
    Suppose $ 1/n \ll \delta, 1/C \ll 1/\ell, 1/c \ll 1/r, \eps $.
    Let $ p \ge n^{-\frac{r - k}{\binom{r}{k} - 1} + \eps} $ and suppose $ G \sim G^{(k)}(n, p) $.
    Then the following holds with probability at least $ 1 - e^{-n^{\delta}} $.
    For any $ 0 \le b < j \le k $ and $ (b, c) $-centred semi-ordered $ j $-clique-path $ \mathcal{F} $ of length $ m \le \ell^{3k} $, there exist $ \alpha_{\mathcal{F}} \in \mathbb{R} $ with $ |\alpha_{\mathcal{F}}| \le C $ and $ \zeta_{\mathcal{F}}: E_j^2 \to \mathbb{R} $ with $ \norm{\zeta_{\mathcal{F}}} \le n^{-\delta} $ such that
    $$ \Xi[\mathcal{F}](S, T) = \alpha_{\mathcal{F}} \frac{\gamma_j}{\gamma_b} + \zeta_{\mathcal{F}}(S, T) $$
    for all $ S, T \in E_j $ with $ \iota(\mathcal{F}) \subseteq \iota(S \cap T, S) $.
    In particular, $ G $ satisfies \ref{cond:creg3}, with $ n^{-\delta} $ playing the role of $ \eps $.
\end{lemma}

\begin{proof}
    Suppose $ 1/C \ll 1/C' \ll 1/\ell, 1/c $.
    By \cref{fact:creg2,fact:sets_intersect_conc,fact:clique_path_noninj_trivial,lem:counting_extensions}, it suffices to prove that the statement holds for any (fixed) $ k $-graph $ G $ satisfying the conclusions of \cref{fact:creg2,fact:sets_intersect_conc,fact:clique_path_noninj_trivial,lem:counting_extensions}, with $ 3 \delta $ playing the role of $ \delta $ in each case, and $ \ell^{3k} $ playing the role of $ M $ in \cref{fact:clique_path_noninj_trivial,lem:counting_extensions}.
    We now fix such a $k$-graph $ G $.
    
    Let $ w $ be defined as in \eqref{eqn:clique_path_weight_def}.
    By \cref{fact:creg2}~\ref{cond:fact_creg2_1}, we have $ |K_r(U)| = (1 \pm n^{-3 \delta}) \hat{\gamma}_s $ for all $ s \in [k] $ and $ U \in E_s $.
    It is thus easy to see that
    $$ w(\Phi) = (1 \pm n^{-2 \delta}) \overline{w}, \quad \text{where} \quad \overline{w} \coloneqq \prod_{i = 1}^{m} \frac{\hat{w}_{s_i}}{\hat{\gamma}_{s_i}}, $$
    for any $\mathbf{F} = (F, \mathbf{S}) \in \mathcal{F}$ and $ \Phi \in X_G(S, T, \mathbf{F}) $, where $ \mathbf{s} \coloneqq \mathbf{s}(F) $, noting that $ \mathbf{s} $ is invariant under semi-isomorphism.
    It follows that
    \begin{equation} \label{eqn:creg3_pf_1}
        \Xi[\mathcal{F}](S, T) = (1 \pm n^{-2 \delta}) Y_{\mathcal{F}}(S, T) \overline{w}
    \end{equation}
    for all $ S, T \in E_j $ with $ \iota(\mathcal{F}) \subseteq \iota(S \cap T, S) $.

    Define
    \begin{equation} \label{eqn:creg3_pf_2}
        \alpha_{\mathcal{F}} \coloneqq \frac{\hat{\gamma}_b}{\hat{\gamma}_j} \mu_{\mathcal{F}}(n, p) \overline{w} = (1 \pm n^{-2 \delta}) \frac{\gamma_b}{\gamma_j} \mu_{\mathcal{F}}(n, p) \overline{w},
    \end{equation}
    using \cref{fact:creg2}~\ref{cond:fact_creg2_1}.
    By the bound on $ \mu_{\mathcal{F}}(n, p) $ from \cref{fact:clique_path_noninj_trivial} with $ (\ell^{3k}, C') $ playing the role of $ (M, C) $, we see that
    \begin{equation} \label{eqn:creg3_pf_2.5}
        |\alpha_{\mathcal{F}}| \le \frac{\hat{\gamma}_b}{\hat{\gamma}_j} \cdot C' \frac{\hat{d}_{s_m}}{\hat{d}_b} \prod_{i = 0}^{m - 1} \hat{\gamma}_{s_i} \cdot \prod_{i = 1}^{m} \frac{|\hat{w}_{s_i}|}{\hat{\gamma}_{s_i}}
        \le C' \frac{\hat{\gamma}_b}{\hat{\gamma}_j} \frac{\hat{d}_{s_m}}{\hat{d}_b} \frac{\hat{\gamma}_{s_0}}{\hat{\gamma}_{s_m}} \prod_{i = 1}^{m} |\hat{w}_{s_i}|
        \le C' \frac{\hat{d}_{s_m}}{\hat{d}_b} \frac{\hat{\gamma}_b}{\hat{\gamma}_{s_m}} \le C,
    \end{equation}
    using also in the penultimate inequality the facts that $ s_0 = j $ and $ |\hat{w}_s| \le 1 $ for every $ s \in [k] $, by definition.
    Define further
    $$ \zeta_{\mathcal{F}}(S, T) \coloneqq \Xi[\mathcal{F}](S, T) - \alpha_\mathcal{F} \frac{\gamma_j}{\gamma_{b}}, $$
    for any $ S, T \in E_j $ with $ \iota(\mathcal{F}) \subseteq \iota(S \cap T, S) $, taking $ \zeta_{\mathcal{F}}(S, T) \coloneqq 0 $ otherwise.
    It remains only to prove that $ \norm{\zeta_{\mathcal{F}}} \le n^{-\delta} $.
    
    Indeed, by \eqref{eqn:creg3_pf_1} and \eqref{eqn:creg3_pf_2}, we may write
    \begin{align*}
        \frac{|\zeta_{\mathcal{F}}(S, T)|}{|\overline{w}|} &= Y_{\mathcal{F}}(S, T) - \mu_{\mathcal{F}}(n, p) \pm n^{-2 \delta}(Y_{\mathcal{F}}(S, T) + \mu_{\mathcal{F}}(n, p)) \nonumber \\
            &\le |Z_{\mathcal{F}}(S, T) - \mu_{\mathcal{F}}(n, p)| + |Y_{\mathcal{F}}(S, T) - Z_{\mathcal{F}}(S, T)| + n^{-2 \delta}(Y_{\mathcal{F}}(S, T) + \mu_{\mathcal{F}}(n, p)), \label{eqn:creg3_pf_3}
    \end{align*}
    for any $ S \in E_j $ and $ T \in \hat{\mathcal{T}}_{\iota(\mathcal{F})}(S) $.
    In particular, for any $ S \in E_j $, we compute
    $$ \sum_{T \in E_j} |\zeta_{\mathcal{F}}(S, T)| \le (A + B + D + E) |\overline{w}|, $$
    $$ \quad \text{where} \quad A \coloneqq \sum_{T \in \hat{\mathcal{T}}_{\iota(\mathcal{F})}(S)} (Y_{\mathcal{F}}(S, T) - Z_{\mathcal{F}}(S, T)),
    \quad
    B \coloneqq \sum_{T \in \hat{\mathcal{T}}_{\iota(\mathcal{F})}(S) \setminus \mathcal{T}_{\iota(\mathcal{F})}(S)} |Z_{\mathcal{F}}(S, T) - \mu_{\mathcal{F}}(n, p)|, $$
    $$ D \coloneqq \sum_{T \in \mathcal{T}_{\iota(\mathcal{F})}(S)} |Z_{\mathcal{F}}(S, T) - \mu_{\mathcal{F}}(n, p)|,
    \quad \text{and} \quad
    E \coloneqq n^{-2 \delta} \sum_{T \in \hat{\mathcal{T}}_{\iota(\mathcal{F})}(S)} (Y_{\mathcal{F}}(S, T) + \mu_{\mathcal{F}}(n, p)). $$
    We now seek to bound each of these quantities.
    
    Recall that $ Z_{\mathcal{F}}(S, T) = 0 $ whenever $ T \not \in \mathcal{T}_{\iota(\mathcal{F})}(S) $, and by \cref{fact:sets_intersect_conc}, we have $ |\hat{\mathcal{T}}_{\iota(\mathcal{F})}(S) \setminus \mathcal{T}_{\iota(\mathcal{F})}(S)| \le n^{-2 \delta} t_{\iota(\mathcal{F})} $.
    It follows that $ B \le 2 n^{-2 \delta} t_{\iota(\mathcal{F})} \mu_{\mathcal{F}}(n, p) $.
    By \cref{lem:counting_extensions}, we have
    \begin{equation*} \label{eqn:creg3_pf_4}
        Z_{\mathcal{F}}(S, T) = (1 \pm n^{-2 \delta}) \mu_{\mathcal{F}}(n, p)
    \end{equation*}
    for all $ S \in E_j $ and $ T \in \mathcal{T}_{\iota(\mathcal{F})}(S) $.
    It follows that
    $$ D \le |\mathcal{T}_{\iota(\mathcal{F})}(S)| \cdot n^{-2 \delta} \mu_{\mathcal{F}}(n, p) \le 2 n^{-2 \delta} t_{\iota(\mathcal{F})} \mu_{\mathcal{F}}(n, p). $$
    Using \cref{lem:counting_extensions} in the second equality, \cref{fact:sets_intersect_conc} in the third, and \cref{fact:clique_path_noninj_trivial} in the inequality, we also obtain
    \begin{align*}
        A = \sum_{T \in \hat{\mathcal{T}}_{\iota(\mathcal{F})}(S)} (Y_{\mathcal{F}}(S, T) - Z_{\mathcal{F}}(S, T)) &= Y_{\mathcal{F}}(S) - Z_{\mathcal{F}}(S) \\
            &= Y_{\mathcal{F}}(S) - |\mathcal{T}_{\iota(\mathcal{F})}(S)| \cdot (1 \pm n^{-2 \delta}) \mu_{\mathcal{F}}(n, p) \nonumber \\
            &= Y_{\mathcal{F}}(S) - (1 \pm 3 n^{-2 \delta}) t_{\iota(\mathcal{F})} \mu_{\mathcal{F}}(n, p) \nonumber \\
            &\le 4 n^{-2 \delta} t_{\iota(\mathcal{F})} \mu_{\mathcal{F}}(n, p), \label{eqn:creg3_pf_5}
    \end{align*}
    for every $ S \in E_j $.
    Finally, by \cref{fact:sets_intersect_conc,fact:clique_path_noninj_trivial}, we have 
    $$ E = n^{-2 \delta} (Y_{\mathcal{F}}(S) + |\hat{\mathcal{T}}_{\iota(\mathcal{F})}(S)| \mu_{\mathcal{F}}(n, p)) \le 4 n^{-2 \delta} t_{\iota(\mathcal{F})} \mu_{\mathcal{F}}(n, p). $$
    Using \eqref{eqn:creg3_pf_2}, \eqref{eqn:creg3_pf_2.5}, and the definition of $ t_J $ in \cref{fact:sets_intersect_conc} in the second inequality, we conclude that
    $$ \norm{\zeta_{\mathcal{F}}} \le 12 n^{-2 \delta} t_{\iota(\mathcal{F})} \mu_{\mathcal{F}}(n, p) |\overline{w}| \le 12 n^{-2 \delta} \cdot 2^k \frac{\hat{d}_b}{\hat{d}_j} \cdot C \frac{\hat{\gamma}_j}{\hat{\gamma}_b} \le n^{-\delta}, $$
    as required.
\end{proof}

This completes the proof of \cref{lem:gnp_reg}, which now follows immediately from \cref{fact:creg1,fact:creg2,claim:creg3}.

\section{Concluding remarks}

We believe that the analogue of \cref{conj:triangle_threshold} holds in general for fractional clique decompositions in random hypergraphs.
Write
$$ p^*_{k, r}(n) \coloneqq c_{k, r} \left( \frac{\log{n}}{n^{r - k}} \right)^{\frac{1}{\binom{r}{k} - 1}}
\quad \text{where} \quad
c_{k, r} \coloneqq \left(
\left(k-\frac{r-k}{\binom{r}{k}-1}\right)(r-k)!
\right)^{\frac{1}{\binom{r}{k}-1}} $$
is the unique constant such that $ p^*_{k, r} $ is the sharp threshold function for the property that every edge present in $ G^{(k)}(n, p) $ is contained in a copy of $ K_r^{(k)} $.

\begin{conjecture} \label{conj:general_threshold}
    Let $ k \ge 2 $, $ r \ge k + 1 $, and $ \eps > 0 $, and suppose $ p \ge (1 + \eps) p^*_{k, r}(n) $.
    Then w.h.p.\ $ G^{(k)}(n, p) $ admits a fractional $ K_r^{(k)} $-decomposition.
\end{conjecture}

We wonder whether a suitable variant of the process defined in this paper would also converge to a fractional clique decomposition at (or a constant factor above) the conjectured threshold.
The natural modification would be to replace $ \mathcal{P}_{k, \ell} $ with a new operator $ \mathcal{P}_{k, \ell, L} \coloneqq (\mathcal{R}_k^L \mathcal{P}_{k - 1, \ell})^{\ell} $, where $ \ell $ is the same as in \cref{thm:proc_conv_fcd} but $ L = \Theta(\frac{\log{n}}{\log{\log{n}}}) $.
This choice of $ L $ ensures that the expected number of the associated clique-paths between any pair of edges is still polynomial, so there is some hope of achieving the required concentration.

There are, however, several additional difficulties in this regime.
Firstly, since the number of cliques containing each edge is no longer well-concentrated, we cannot ignore the error terms in the weight functions $ w(\Phi) $; indeed, we expect that the weight of any given semi-copy of a clique-path will depend on local variations in $ G $.
However, we may still hope that the appropriate sums $ \Xi[\mathcal{F}](S, T) $ are concentrated, as these are proportional to the probability of a suitably defined random walk (on an auxiliary hypergraph) ending at $ T $, given that it starts at $ S $.

The more significant difficulty appears to be that we can no longer simply ignore \emph{backtracking} walks.
When each edge is contained in at least $ n^{\eps} $ cliques, we showed in \cref{fact:clique_path_noninj_trivial} that non-injective semi-copies of clique-paths contribute only an $ n^{-\delta} $ fraction to the weighted sum~$\Xi[\mathcal{F}] $, and thus can be ignored.
Imagine choosing a semi-copy of some given clique-path by embedding one clique $ F_i $ at a time, choosing uniformly at random among cliques containing the root set~$ S_{i - 1} $.
Closer to the conjectured threshold, when each edge is contained only in $ \Theta(\log{n}) $ cliques, there is always a $ \Theta(\frac{1}{\log{n}}) $ probability of simply choosing $ F_i = F_{i - 1} $; in other words, such backtracking semi-copies contribute non-negligibly to $ \Xi[\mathcal{F}] $.
This is especially problematic for the following reason.
Suppose $ \mathbf{F} = (F, \mathbf{S}) $ is a $ (0, c) $-centred clique-path and $ \Phi $ is a semi-copy of $ \mathbf{F} $ which backtracks $ c $ times (only a constant number).
Then, in the hypergraph $ \Phi(F) $, it may be the case that the roots are only at distance $ 1 $ apart, meaning that the number of such copies is not well-concentrated.

\section*{AI declaration}

We made use of ChatGPT Pro 5.5 and 5.6 Sol in the process of verifying and editing this paper.
It identified a few minor issues in previous versions of the argument, which resulted in changes to the paper.
We also used it while attempting to find more efficient proofs of several lemmas, as well as to assist in the creation of \Cref{fig:clique_path}.
However, none of the ideas or text present in the final paper are originally due to ChatGPT, including the new arguments introduced to fix the problems which it pointed out.
The authors take full accountability for the text and mathematical content of this paper.

\bibliographystyle{amsplain}
\bibliography{sources}

\end{document}